\documentclass[a4paper,12pt,fullpage]{article}
\usepackage{eurosym}
\usepackage{appendix}
\usepackage[T1]{fontenc}
\usepackage{epsf,epsfig,subfigure}
\usepackage{amssymb}
\usepackage{amsfonts,mathrsfs}
\usepackage{amsmath}
\usepackage{cases}
\usepackage{mathtools}
\usepackage[english]{babel}
\usepackage{graphicx,multirow}
\usepackage{indentfirst}
\usepackage[colorlinks=true]{hyperref}
\hypersetup{urlcolor=blue,linkcolor=red ,citecolor=blue,colorlinks=true}
\usepackage{bbm}
\usepackage{geometry}
\graphicspath{{./Pictures/}}
\usepackage[latin1]{inputenc}
\usepackage{listings}
\usepackage{amsmath}
\allowdisplaybreaks[4]
\usepackage{tikz}
\usetikzlibrary{intersections}
\usetikzlibrary{matrix}

\usepackage{multirow}
\allowdisplaybreaks[4]
\usepackage[latin1]{inputenc} 

\usepackage{paralist,graphics,epsfig,graphicx,epstopdf,mathrsfs}
\usepackage{float,color,comment,tabulary,booktabs}
\usepackage[normalem]{ulem}

\usepackage{dsfont}
\usepackage{color}

\usepackage{enumitem}
\newfloat{figure}{H}{lof}
\newfloat{table}{H}{lot}
\floatname{figure}{\figurename}
\floatname{table}{\tablename}

\numberwithin{equation}{section}

\newtheorem{theorem}{Theorem}[section]

\newtheorem{claim}{Claim}[section]

\newtheorem{definition}{Definition}[section]

\newtheorem{lemma}{Lemma}[section]

\newtheorem{proposition}{Proposition}[section]
\newtheorem{remark}{Remark}[section]

\newenvironment{proof}[1][Proof]{\noindent\textit{#1.} }{\hfill \rule{0.5em}{0.5em}}
\def\vs{\vskip.3cm}

\DeclareMathOperator{\dist}{dist}
\DeclareMathOperator{\supp}{supp}
\DeclareMathOperator{\diam}{diam}

\newcommand \dis {\displaystyle}

\newcommand{\ds}{\mathrm{d} s}

\newcommand{\NN}{\mathbb{N}}
\newcommand{\R}{\mathbb{R}}
\newcommand{\Z}{\mathbb{Z}}

\newcommand{\bqq}{\begin{equation}}
\newcommand{\eqq}{\end{equation}}
\newcommand{\bqs}{\begin{equation*}}
\newcommand{\eqs}{\end{equation*}}

\begin{document}

	\title{\textbf{Propagation phenomena in KPP-bistable periodic patchy environments$\ $\thanks{This work has received funding from the French ANR ReaCh (ANR-23-CE40-0023-02).}}}
	\author{Quentin Griette$^{\hbox{\small{ a}}}$, Fran{\c{c}}ois Hamel$^{\hbox{\small{ b}}}$, Mingmin Zhang$^{\hbox{\small{ c}}}$, Min Zhao$^{\hbox{\small{ b}}}$\\
		\date{}\\
		\footnotesize{$^{\hbox{a }}$Universit\'e Le Havre Normandie, Laboratoire de Math\'ematiques Appliqu\'ees du Havre, France}\\
		\footnotesize{$^{\hbox{b }}$Aix Marseille Univ, CNRS, I2M, Marseille, France}\\
		\footnotesize{$^{\hbox{c }}$School of Mathematical Sciences, University of Science and Technology of China, P.R. China}}

	\maketitle

\begin{abstract}
	This paper first investigates the propagation dynamics of solutions to the Cauchy problem for a one-dimensional reaction-diffusion equation in a spatially periodic environment consisting of two distinct patch types. The novelty of this work lies in the systematic analysis of a KPP-bistable heterogeneous framework. In this setting, the respective patch lengths, the linear stability of the zero solution and the positive periodic steady state, and the magnitude of the initial data play crucial roles in the long-time dynamics. We first establish persistence properties of the species, showing that uniform persistence holds when the zero steady state of the associated periodic patch model is unstable, while local persistence is obtained under additional suitable conditions. Using a dynamical systems approach, we further establish spreading properties and demonstrate the existence of pulsating traveling waves in two different cases, depending on whether the trivial solution is unstable or stable. Finally, we present two sets of sufficient conditions characterizing species extinction.
\end{abstract}




	
\section{Introduction}

In recent years, many scholars have devoted considerable attention to the study of propagation dynamics of reaction-diffusion equations in heterogeneous environments, which play an important role in biological invasion, population dynamics, and landscape ecology~\cite{AL2-2019,BHR1-2005,DL-2009,DR-2018,FG-1979,HLZ-2024,ML-2018,SKT-1986,SKW-2015,W-2002}. From the perspective of landscape ecology, natural habitats are often fragmented into distinct homogeneous patches (such as forests, grasslands, and swamps) separated by physical or ecological boundaries, including rivers or roads~\cite{HLZ-2024,ML-2013,SK-1997,SKT-1986}. Numerous ecological studies have shown that habitat fragmentation and edge effects significantly influence population density and spatial distribution, primarily by affecting individual movement at patch boundaries~\cite{DH-2000,DL-2009,L-1999,RFBS-2004,SB-2003,SKT-1986,VPHL-2007}. In such settings, the environment can be viewed as a periodic arrangement of patches with distinct local properties, which may lead to propagation dynamics markedly different from those observed in homogeneous media. This motivates the study of periodic patch models, which provide a natural and tractable framework for capturing both spatial heterogeneity and interface effects in a rigorous mathematical setting.

Reaction-diffusion equations have been widely used to model the spatial spread of biological species (e.g., bacteria, insects, and plants) and the invasive processes associated with biological invasions and epidemic outbreaks~\cite{AW-1975,AW-1978,FM-1979,FM-1977,Fisher-1937,KPP-1937}. In homogeneous environments, propagation phenomena are well understood, including spreading speeds and traveling wave solutions. In particular, bistable equations have been extensively studied, with the existence and uniqueness of traveling waves established by Fife and McLeod in~\cite{FM-1979,FM-1977}. In theses two works, the authors studied the bistable homogeneous reaction-diffusion equation
\begin{equation}\label{eqbistable}
u_t = d_2 u_{xx} + f_2(u), \qquad (t,x)\in\mathbb{R}\times\mathbb{R},
\end{equation}
where the nonlinearity $f_2$ satisfies assumption~\eqref{f2} below. They proved the existence of a traveling wave solution of the form $\phi(x-c_2 t)$ connecting $K_2$ to $0$. The wave profile $\phi:\mathbb{R}\to(0,K_2)$ satisfies
\begin{equation*}\label{TW}
	\left\{
	\begin{array}{l}
		d_2 \phi'' + c_2 \phi' + f_2(\phi) = 0 \quad \text{in } \mathbb{R}, \qquad \phi' < 0 \ \text{in } \mathbb{R},\\[1mm]
		\phi(-\infty) = K_2,\quad \phi(+\infty) = 0,\quad \phi(0) = \theta,
	\end{array}
	\right.
\end{equation*}
where the wave speed $c_2$ has the same sign as $\int_0^{K_2}\!\!f_2(s)\,\mathrm{d}s$~\cite{AW-1978, FM-1977}. The normalization condition $\phi(0)=\theta$, with $\theta\in(0,K_2)$, uniquely determines the profile $\phi$.

Many classical studies have investigated population dispersal in environments consisting of alternating favorable and unfavorable patches, with particular emphasis on the role of interface behavior in invasion dynamics~\cite{HLZ-2022,HLZ-2024,ML-2013,ML-2015,SK-1997}. Shigesada, Kawasaki, and Teramoto~\cite{SK-1997,SKT-1986} proposed a model describing the spread of a single species in a periodically heterogeneous environment with alternating patches. In their framework, individuals undergo logistic growth and diffusion with patch-dependent parameters. They showed that when a population initially invades a localized region, the dynamics exhibit a dichotomy between extinction and propagation, depending on the patch sizes as well as the diffusion and growth rates. In the case of successful invasion, the population evolves into a periodic traveling wave with a well-defined spreading speed, which can be characterized via a dispersion relation.

Subsequent works have incorporated additional ecological mechanisms. In particular, Maciel and Lutscher~\cite{ML-2015,ML-2018} studied invasion dynamics in periodic patch models with strong Allee effects and interface dependent movement preferences. Using homogenization techniques together with classical results for propagation in homogeneous environments, they derived approximate expressions for spreading speeds and showed that dispersal behavior can significantly influence both invasion outcomes and propagation rates in heterogeneous landscapes. More recently, Hamel, Lutscher, and Zhang~\cite{HLZ-2024} introduced and analyzed a class of periodic patch models with interface conditions that explicitly account for movement across habitat boundaries. For monostable (KPP-type) nonlinearities, they established well-posedness of the Cauchy problem, characterized long-time dynamics and spreading properties, and proved the existence of pulsating traveling waves. Moreover, Hamel, Lutscher, and Zhang~\cite{HLZ-2022} investigated propagation and blocking phenomena in two-patch reaction-diffusion models under three different configurations of nonlinearities, namely the KPP-KPP, KPP-bistable, and bistable-bistable cases. 

However, in real world, spatial heterogeneity, such as periodic environments, not only affects the growth rate of species but also their dispersal capacity and behavior, which plays a crucial role in studying the dispersal of invasive species. Freidlin and G\"artner~\cite{FG-1979} studied the propagation of concentration waves in periodic and random media. Weinberger studied the propagation speed and traveling waves of growth and migration models in periodic habitats~\cite{W-2002}. Berestycki, Hamel, and Nadirashvili studied the propagation speed of KPP-type problems with periodic environments~\cite{BHN-2005}. Berestycki, Hamel, and Roques investigated species persistence, biological invasion, and pulsating travelling fronts in a periodically fragmented environment model with arbitrary spatial dimensions~\cite{BHR1-2005,BHR2-2005}. Ducasse and Rossi studied the  blocking and invasion for reaction-diffusion equations in periodic media~\cite{DR-2018}. Eberle considered the front blocking in the presence of gradient drift of a bistable reaction-diffusion equation in straight cylinders in dimension $n\geq 3$~\cite{E-2019}. Ding, Hamel, and Zhao~\cite{DHZ-2017} investigated the existence and qualitative properties of pulsating fronts in spatially periodic reaction-diffusion equations with bistable nonlinearities.

Despite these advances, the analysis of propagation dynamics in periodic patch models remains challenging. The main difficulties stem from the discontinuities of coefficients at patch interfaces and the interaction of distinct local dynamics across heterogeneous regions. In particular, while KPP-type nonlinearities are relatively well understood (see~\cite{HLZ-2022,HLZ-2024}), the case of mixed nonlinearities, such as the coupling of KPP-type and bistable dynamics in a periodic patch framework, has not been resolved.

In this paper, we study a reaction-diffusion model in a periodically fragmented environment consisting of two types of patches, referred to as patch 1 (KPP-type) and patch 2 (bistable). We consider a one-dimensional equation for the population density $u(t,x)$, given by
\begin{equation}\label{eq}
\left\{
\begin{aligned}
u_t - d(x) u_{xx} &= f(x,u), && t>0,\ x\in\mathbb{R}\setminus S, \\
u(t,x^-) &= u(t,x^+), \quad u_x(t,x^-) = \sigma u_x(t,x^+),
&& t>0,\ x\in S_1, \\
u(t,x^-) &= u(t,x^+), \quad \sigma u_x(t,x^-) = u_x(t,x^+),
&& t>0,\ x\in S_2,
\end{aligned}
\right.
\end{equation}
where the diffusion coefficient $d(x)$ and the reaction term $f(x,u)$ depend on the local patch type. The habitat is assumed to consist of a spatially periodic arrangement of patches with period $l = l_1 + l_2$, where $l_i>0$ $(i=1,2)$ denotes the length of patches of type $i$. More precisely, the real line is decomposed into intervals of the form $[nl - l_1, nl + l_2]$, $n\in\mathbb{Z}$, each consisting of two adjacent patches: $(nl - l_1, nl)$ of type 1 and $(nl,nl+l_2)$ of type 2. Following~\cite{HLZ-2022}, we denote by $S_1 = l\mathbb{Z}$ the interface points between $(nl - l_1, nl)$ and $(nl, nl + l_2)$, and by
$S_2 = \{ s + l_2 : s \in l\mathbb{Z} \}$ the interface points between
$(nl, nl + l_2)$ and $(nl+ l_2, (n+1)l)$. We define $S = S_1 \cup S_2$ as the set of all interface points in $\mathbb{R}$. Population movement across patch interfaces is described by transmission conditions imposed at the interface sets $S_1$ and $S_2$, where superscripts $\pm$ denote one-sided limits. These conditions account for discontinuities in the flux and are governed by a parameter $\sigma > 0$, defined by
$$
\sigma = \dfrac{1-\alpha}{\alpha},
$$
where $\alpha\in(0,1)$ denotes the probability that an individual at the interface chooses to move to the adjacent patch of type 1, and $1-\alpha$ the probability that it moves to the patch of type 2.
Interface conditions play an important role in reaction-diffusion patch models. Classical works by Shigesada, Kawasaki and Teramoto~\cite{SKT-1986} adopted continuity conditions for both population density and flux across patch interfaces. Subsequently, motivated by the work of Ovaskainen and Cornell~\cite{OC-2003}, Maciel and Lutscher~\cite{ML-2013} proposed a new class of interface conditions that preserve flux continuity while allowing density discontinuity at patch boundaries, thereby incorporating patch preference and unequal diffusion rates. In contrast, the interface conditions considered in this paper, given in~\eqref{eq}, maintain density continuity while allowing discontinuous flux across interfaces, derive in~\cite{HLZ-2024}. Different types of interface conditions have been shown to substantially affect persistence criteria and spreading speeds in heterogeneous periodic environments~\cite{AL2-2019,HLZ-2024,ML-2015}.

Within each patch, the environment is homogeneous, and the coefficients take the piecewise constant form
\begin{equation}\label{patch}
d(x)=
\begin{cases}
d_1, & x \in (nl - l_1, nl), \\
d_2, & x \in (nl, nl + l_2),
\end{cases}
\qquad
f(x,s)=
\begin{cases}
f_1(s), & x \in (nl - l_1, nl), \\
f_2(s), & x \in (nl, nl + l_2).
\end{cases}
\end{equation}

Throughout this paper, we assume that the functions $f_i$ belong to $C^1(\R)$ $(i =1, 2)$ and that there exists $K_i>0$ such that $f_1$ is of KPP type and $f_2$ is of bistable type, as specified in the following assumptions:
\begin{equation}\label{f1}
\left\{\begin{array}{l}
f_1(0)=f_1\left(K_1\right)=0,\quad 0<f_1(s) \leq f_1^{\prime}(0) s \text {  for } s \in\left(0, K_1\right), \\
f_1^{\prime}\left(K_1\right)<0, f_1<0 \text { in }(-\infty, 0) \cup\left(K_1,+\infty\right)
\end{array}\right.
\end{equation}
and
\begin{equation}\label{f2}
\left\{\begin{array}{l}
f_2(0)=f_2(\theta)=f_2\left(K_2\right)=0 \text { for some } \theta \in\left(0, K_2\right), \\
f_2^{\prime}(0)<0, \quad f_2^{\prime}(\theta)>0, \quad f_2^{\prime}\left(K_2\right)<0, \\
f_2<0 \text { in }(0, \theta) \cup\left(K_2,+\infty\right), \quad f_2>0 \text { in }(-\infty, 0) \cup\left(\theta, K_2\right) .
\end{array}\right.
\end{equation}

Furthermore, unless otherwise specified, we always write $I$ for an arbitrary patch in $\R$ of either type, i.e., either $I= (nl-l_1,nl)$ or $I= (nl,nl + l_2)$.

Unlike~\cite{HLZ-2024}, where KPP-type nonlinearities are assumed in all patches, the periodic patch model~\eqref{eq}-\eqref{patch} considered here involves an alternation between KPP and bistable
dynamics. This alternation substantially increases the complexity of the analysis of propagation phenomena and long-time behavior, particularly with respect to propagation and extinction. Moreover,
the characterization of the spreading speed differs fundamentally from that in~\cite{HLZ-2024}. In the purely KPP setting, the spreading speed can be characterized by the principal eigenvalue associated with the linearization at the zero steady state. In contrast, the interaction between KPP-type and bistable type leads to a more intricate propagation mechanism, in which the speed is no longer
determined solely by local linear properties at the unstable state. A detailed discussion of these difficulties is provided in Section~\ref{sp}.

This paper provides the first analysis of propagation dynamics in a one-dimensional periodic environment consisting of alternating KPP and bistable patches. We will show that the long time behavior of the solution of system~\eqref{eq}-\eqref{patch} is governed by the interplay between patch lengths $l_1$ and $l_2$, initial population density, as well as the linear stability of both the zero solution and the positive periodic steady state solution. System~\eqref{eq}-\eqref{patch} is characterized by spatial periodicity, the integration of KPP and bistable dynamics, and discontinuous interface flux, all of which present substantial analytical challenges. 
By employing comparison principles and sub- and super-solution methods, we establish persistence properties of the species. More precisely, we prove uniform persistence when the zero solution is unstable, and local persistence under suitable parameter conditions, regardless of the stability of the zero solution. Furthermore, we investigate the existence of asymptotic spreading speeds and pulsating traveling waves. When the zero solution is unstable, we establish the existence of an asymptotic spreading speed $c^*$ and pulsating traveling waves, and show that $c^*$ coincides with the minimal wave speed of such waves. When the zero solution is stable, we prove the existence of pulsating traveling waves under appropriate assumptions on the patch size and nonlinearities. Finally, we derive sufficient conditions ensuring species extinction in two cases: when the initial data are sufficiently small, and when the KPP patch length $l_1$ is fixed whereas the bistable patch length $l_2$ is taken sufficiently large. 


\section{Main results}

To show propagation and extinction phenomena under various types of conditions on the sizes of the patches or on the growth functions in the patches, or on the stability of the trivial solution of system~\eqref{eq}-\eqref{patch}, we first introduce the eigenvalue problem.


\subsection{Eigenvalue problem}\label{eig}

We recall several results concerning the principal eigenvalue associated with the linearization of~\eqref{eq}-\eqref{patch} at the trivial steady state $0$. Following~\cite{HLZ-2024, ML-2013, SKT-1986}, there exists a principal eigenvalue $\lambda_1\in\mathbb{R}$, which can be characterized as the unique real number for which there exists a unique continuous function  $\phi:\mathbb{R}\to\mathbb{R}$ satisfying $\phi|_{\overline{I}}\in C^\infty(\overline{I})$ for each patch $I$, and
\begin{equation}\label{eigenv}
\begin{cases}
L_0\phi:=-d(x)\phi^{\prime\prime}(x) -f_s(x, 0)\phi(x)=\lambda_1\phi(x),  & x \in \mathbb{R} \backslash S, \\
\phi(x^{-}) =\phi(x^{+}),\quad \phi^{\prime}(x^{-})=\sigma \phi^{\prime}(x^{+}), & x \in S_1, \\
\phi(x^{-}) =\phi(x^{+}),\,\, \sigma \phi^{\prime}(x^{-})=\phi^{\prime}(x^{+}), & x \in S_2,\\
\phi\text{ is periodic}, \phi>0, \|\phi\|_{L^{\infty}(\R)}=1.
\end{cases}
\end{equation}
By periodic, we mean that $\phi(\cdot+l)=\phi$ in $\R$.
It is well-known that the sign of $\lambda_1$ determines the stability of the state $0$ of system~\eqref{eq}-\eqref{patch}. Under the KPP setting,~\cite{HLZ-2024}  showed that the state $0$ is unstable if $\lambda_1<0$ and stable if $\lambda_1\geq 0$. More precisely, if both $f_1$ and $f_2$ satisfy KPP-type condition as in~\eqref{f1}, with $f_i(s)/s$ $(i=1,2)$ decreasing for $s>0$, and if $u_0:\mathbb{R}\to\mathbb{R}^+$ is continuous, nonnegative, and compactly supported, then the solution $u$ of~\eqref{eq}-\eqref{patch} with initial datum $u_0$ converges locally uniformly in $x$ to a positive steady state as $t\to+\infty$ if $\lambda_1<0$, whereas it converges uniformly in $x$ to $0$ as $t\to+\infty$ if $\lambda_1\geq0$.

However, for the KPP-bistable mixing nonlinearity $f$ here, the situation is much more complex, and as will be seen, the sign of $\lambda_1$ is not the only parameter determining the large-time dynamics of solutions to~\eqref{eq}-\eqref{patch}. Furthermore, whether $0$ is stable at $\lambda_1=0$ remains unclear.

Moreover, consider the elliptic problem associated with~\eqref{eq}-\eqref{patch}:
\begin{equation}\label{sta1}
\begin{cases}
-d(x)p''(x)=f(x,p(x)), & x\in\mathbb{R}\setminus S,\\
p(x^-)=p(x^+),\quad p'(x^-)=\sigma p'(x^+), & x\in S_1,\\
p(x^-)=p(x^+),\quad \sigma p'(x^-)=p'(x^+), & x\in S_2.
\end{cases}
\end{equation}
It follows that the nonlinearity $f$ in~\eqref{sta1} satisfies condition~\cite[(2.11)]{HLZ-2024}, i.e.
\begin{equation*}
\begin{cases}
\forall\,x\in\mathbb{R}\!\setminus\!S,\ f(x,\cdot)\in C^1(\mathbb{R}),\ f(x,0)=0,\\
 \forall\,x\in\mathbb{R}\!\setminus\!S,\ \forall\,s\ge\max(K_1,K_2),\ f(x,s)\le0,\\
\forall\,x\in(nl-l_1,nl),\ f(x,\cdot)=f_1,\ \ \forall\,x\in(nl,nl+l_2),\ f(x,\cdot)=f_2.
\end{cases}
\end{equation*} 
Therefore, Theorem~2.3(i) and Proposition~4.1 of~\cite{HLZ-2024} apply and yield the following properties.

\begin{proposition}\label{proSTA}
	Assume that $0$ is an unstable solution of~\eqref{sta1} \emph{(i.e., $\lambda_1<0$)}. Then 
	\begin{itemize}
		\item[(i)]  there exists a positive, bounded, and periodic solution $p$ to system~\eqref{sta1} on $\mathbb{R}$;
		\item[(ii)]   if~$\tilde{p}$ is a bounded nonnegative continuous solution of the stationary problem~\eqref{sta1}, then, either~$\tilde{p}\equiv 0$ in $\R$, or $\inf_{\mathbb{R}}\tilde{p}>0$;	
		\item[(iii)] there exists a minimal positive bounded solution $q$ to~\eqref{sta1} such that any positive bounded solution $p$ of system~\eqref{sta1} satisfies
		$$
		0<q(x)\leq p(x), \quad \forall x\in\R.
		$$
		Furthermore, this minimal positive bounded solution $q$ is periodic on $\R$.
	\end{itemize} 
\end{proposition}

More precisely, Proposition~\ref{proSTA}(i) and (ii) follow respectively from Theorem 2.3(i) and Proposition 4.1 in~\cite{HLZ-2024}. In the presence of bistable patches, even if we assume a priori that $0$ is an unstable solution of~\eqref{sta1}, the uniqueness of positive steady states  is still not obvious, instead, we establish the existence of a minimal positive solution to system~\eqref{sta1}, and uniqueness then follows. This corresponds to Proposition~\ref{proSTA}(iii), whose proof will be given in Section~\ref{Sec-SP-TW}.

Since $f_2'(0)<0<f_1'(0)$ here, it is known from the analysis of formula~\cite[(2.16)]{HLZ-2024} that there exists a critical length $l_1^{c}>0$ given by
\begin{equation}
	\label{l1c}
l_1^{c}
:=
2\sqrt{\frac{d_1}{f_1'(0)}}
\arctan\left(
\sigma\sqrt{\frac{-d_1 f_2'(0)}{d_2 f_1'(0)}}
\tanh\!\left(
\frac{l_2}{2}\sqrt{\frac{-f_2'(0)}{d_2}}
\right)
\right),
\end{equation}
such that $0$ is stable (i.e., $\lambda_1> 0$) if $l_1< l_1^c$, and unstable (i.e., $\lambda_1<0$) if $l_1>l_1^c$. The quantity $l_1^c$ is increasing with respect to $l_2>0$, and as the size of $l_2$ tends to infinity (i.e., $l_2\to+\infty$), we obtain that 
\begin{equation}\label{L1c}
l_1^c\to L_1^{c}
:=
2\sqrt{\frac{d_1}{f_1'(0)}}
\arctan\left(
\sigma\sqrt{\frac{-d_1 f_2'(0)}{d_2 f_1'(0)}}
\right).
\end{equation}
Therefore, as long as $l_1\geq L_1^c$, the trivial solution of~\eqref{sta1} is unstable (i.e., $\lambda_1<0$), no matter how large the size of the bistable patches is. 

Similar to the preceding analysis, one can also show that there exists a critical length $l_2^c > 0$ given by
\begin{equation}\label{l2c}
l_2^c := 2\sqrt{-\frac{d_2}{f_2'(0)}} \operatorname{arctanh} \left[ \frac{1}{\sigma} \sqrt{-\frac{d_2 f_1'(0)}{d_1 f_2'(0)}} \tan \left( \frac{l_1}{2} \sqrt{\frac{f_1'(0)}{d_1}} \right) \right].
\end{equation}
This critical value $l_2^c$ is well-defined if and only if the following condition holds:
$$
0<\frac{1}{\sigma} \sqrt{-\frac{d_2 f_1'(0)}{d_1 f_2'(0)}} \tan \left( \frac{l_1}{2} \sqrt{\frac{f_1'(0)}{d_1}} \right)<1.
$$
This condition means that $0<l_1<L_1^c$, where $L_1^c$ is defined in~\eqref{L1c}. If the above condition does not hold, i.e., $l_1 \ge L_1^c$, then the critical value $l_2^c$ does not exist. In this case, the trivial solution of~\eqref{sta1} remains unstable (i.e., $\lambda_1<0$) for all $l_2>0$. Therefore, for any $l_1 \in (0, L_1^c)$, $l_2^c$ acts as a  stability threshold for the trivial solution of~\eqref{sta1}. Namely, $0$ is linearly stable (i.e., $\lambda_1 > 0$) for $l_2>l_2^c$, and becomes linearly unstable (i.e., $\lambda_1 < 0$) if $l_2 < l_2^c$. These observations lead to the following results.

\begin{proposition}\label{lem-STA}
	Let $l_1^c$, $L_1^c$, $l_2^c$ be given by~\eqref{l1c},~\eqref{L1c} and~\eqref{l2c}, respectively. We have:
	\begin{itemize}
		\item[\textnormal{(i)}]  For any fixed $l_2>0$, if $l_1 > l_1^c$, then $0$ is an unstable solution of~\eqref{sta1}. The hypothesis $l_1>l_1^c$ equivalently means that either $l_1\ge L_1^c$, or $0 < l_1 < L_1^c$ and $l_2 < l_2^c$.
		
		\item[\textnormal{(ii)}]  For any fixed $l_2>0$, if $l_1 < l_1^c$, then $0$ is a stable solution of~\eqref{sta1}. The hypothesis $l_1<l_1^c$ equivalently means that $0 < l_1 < L_1^c$ and $l_2 >l_2^c$.
	\end{itemize}
\end{proposition}

Furthermore, it follows from~\eqref{l1c} and~\eqref{l2c} that the critical length functions $l_1^c(l_2)$ and $l_2^c(l_1)$ are strictly increasing with respect to $l_2\in(0,+\infty)$ and $l_1\in(0,L_1^c)$, respectively. This monotonicity establishes a sharp transition in the stability of $0$ for system~\eqref{sta1}, as detailed below:
\begin{enumerate}[label=(\arabic*)]
	\item \textit{Threshold in $l_2$:} for any fixed $l_1\in(0,L_1^c)$, there exists a unique $l_2^* > 0$ such that $l_1^c(l_2^*) = l_1$. Consequently:
	\begin{itemize}
		\item If $l_2 < l_2^*$, then $l_1^c(l_2) < l_1$. According to Proposition~\ref{lem-STA}(i), this implies that $0$ is an unstable solution of~\eqref{sta1}. The same conclusion also holds if $l_1\geq L_1^c$, whatever $l_2$ be.
		\item If $l_2 > l_2^*$, then $l_1^c(l_2) > l_1$. By Proposition~\ref{lem-STA}(ii), $0$ is a stable solution of~\eqref{sta1}.
	\end{itemize}
	
	\item \textit{Threshold in $l_1$:} for any fixed $l_2 > 0$, there exists a unique $l_1^* \in (0, L_1^c)$ such that $l_2^c(l_1^*) = l_2$. It follows that:
	\begin{itemize}
		\item If $l_1 < l_1^*$, then $l_2^c(l_1) < l_2$, which implies that $0$ is a stable solution of~\eqref{sta1} by Proposition~\ref{lem-STA}(ii).
		\item If $l_1^* < l_1 < L_1^c$, then $l_2^c(l_1) > l_2$. By Proposition~\ref{lem-STA}(i), $0$ is an unstable solution of~\eqref{sta1}. The same conclusion also holds if $l_1\geq L_1^c$.
	\end{itemize}
\end{enumerate}


\subsection{Persistence}\label{per}

In this subsection, we investigate the persistence of solutions to problem~\eqref{eq}-\eqref{patch}. More precisely, we first establish the uniform persistence of solutions to the Cauchy problem~\eqref{eq}-\eqref{patch} with nonnegative, continuous, and compactly supported initial data. 

\begin{theorem}\label{thmPERS}
Assume that $0$ is an unstable solution of~\eqref{sta1} \emph{(i.e., $\lambda_1<0$)}. Let $u$ be the solution of~\eqref{eq}-\eqref{patch} with a nonnegative continuous and compactly supported initial datum $u_0\not\equiv0$. Then 
$$
\inf_{x\in \R}\left(\liminf_{t\to+\infty}u(t,x)\right)>0.
$$
\end{theorem}

We then turn to a local persistence result. Specifically, we prove that if the initial datum exceeds $\theta+\eta$ (for some given $\eta>0$) on a sufficiently large bistable patch, and if $\int_{0}^{K_2}f_2(s)\ds>0$, then extinction cannot occur, as the following result shows.

\begin{theorem}\label{thmNOEXT}
	Assume that $\int_0^{K_2} f_2(s) \mathrm{d} s>0$. Let $u$ be the solution of~\eqref{eq}-\eqref{patch} with a nonnegative continuous and compactly supported initial datum $u_0 \not \equiv 0$. Then, for any $\eta>0$, there is $l_2^{**}>0$ such that, if $l_2\geq l _2^{**}$ and $u_0 \geq \theta+\eta$ on an interval of size $l_2^{**}$ included in patch 2, then $u$ is locally persistent, namely, for every bounded interval $H\in\R$,
	\begin{equation}\label{noex}
		\inf_{x\in H}\left(\liminf_{t\to+\infty}u(t,x)\right)>0.
	\end{equation}
\end{theorem}

It is worth emphasizing that Theorem~\ref{thmNOEXT} holds
independently of the size of the KPP patch,  that is, the local persistence result remains valid no matter whether $0$ be a stable or unstable solution of~\eqref{sta1}. Some further comments on Theorems~\ref{thmPERS} and~\ref{thmNOEXT} are given in Section~\ref{sec25}.


\subsection{Spreading speed and pulsating traveling waves}\label{sp}

In view of the persistence result established in~Theorems~\ref{thmPERS} and~\ref{thmNOEXT}, we are now in a position to investigate the spreading properties of solutions and to characterize the associated spreading speed and pulsating traveling waves. We first consider the monostable case, where $0$ is an unstable steady state of~\eqref{eq}-\eqref{patch}, and establish the existence of spreading speeds and pulsating traveling waves, see Theorems~\ref{thmSP} and~\ref{thmTW} below. We then address the bistable case, in which $0$ is a stable steady state of problem~\eqref{eq}-\eqref{patch}, and prove the existence of pulsating traveling waves under some additional conditions on $l_1$ and $l_2$, see Theorem~\ref{bisTW} below.

We first introduce some notations and functional settings that will be used throughout this paper.

\medskip
\noindent\textbf{Notations.} 
Let $\mathcal{C}$ denote the space of all bounded and uniformly continuous functions from $\mathbb{R}$ to $\mathbb{R}$, equipped with the compact open topology. That is, we say that $u_n \to u$ in $\mathcal{C}$ as $n \to +\infty$ if $u_n \to u$ locally uniformly in $\mathbb{R}$. For $u,v \in \mathcal{C}$, we write $u \ge v$ if $u(x) \ge v(x)$ for all $x \in \mathbb{R}$, $u > v$ if $u \ge v$ and $u \not\equiv v$, and $u \gg v$ if $u(x) > v(x)$ for all $x \in \mathbb{R}$. Assume that $p \in \mathcal{C}$ is a positive, bounded, periodic solution of~\eqref{sta1} with $p \gg 0$ as guaranteed by Proposition~\ref{proSTA}, under the instability of $0$. Using this steady state, we define
\begin{equation}\label{Cp}
\mathcal{C}_p = \{v \in \mathcal{C} : 0 \le v \le p \}.
\end{equation}

\begin{theorem}\label{thmSP}
Assume that $0$ is an unstable solution of~\eqref{sta1} \emph{(i.e., $\lambda_1<0$)}. Let $q$ be the minimal solution of~\eqref{sta1} given in Proposition~\ref{proSTA}-(iii). Then there exists a positive real number
$c^*>0$
called the \emph{spreading speed}, such that for any nonnegative, continuous initial datum $u_0\in \mathcal{C}_q$, 
the solution $u(t,x)$ of~\eqref{eq}-\eqref{patch} satisfies:
\begin{enumerate}
\item [(i)] For any $c>c^*$, if $u_0$ is furthermore assumed to be compactly supported and satisfies $u_0(x)<q(x)$ for all $x\in\R$, then
$$
\lim_{t\to+\infty}\;
\sup_{|x|\ge ct} u(t,x)=0.
$$
\item [(ii)]  For any $0\leq c<c^*$, if $u_0\not\equiv0$, then
$$
\lim_{t\to+\infty}\;
\sup_{|x|\leq c t}
|u(t,x)-q(x)|=0.
$$
\end{enumerate}
\end{theorem}

We now turn to the connection between the asymptotic spreading speed 
$c^*$ and periodic (also called pulsating) traveling waves, whose definition is recalled below.

\begin{definition}
A bounded continuous solution $u: \mathbb{R} \times \mathbb{R} \rightarrow \mathbb{R}$ of problem~\eqref{eq}-\eqref{patch} is called a periodic traveling wave connecting $p$ to $0$ if it has the form $u(t, x)= W(x-c t, x)$, where $c \in \mathbb{R}$ and the function $W: \mathbb{R} \times \mathbb{R} \rightarrow \mathbb{R}$ has the properties: for each $s \in \mathbb{R}$ the map $x \mapsto W(x+s, x)$ is continuous and the map $x \mapsto W(s, x)$ is periodic, and for each $x \in \mathbb{R}$ the map $s \mapsto W(s, x)$ is decreasing with $W(-\infty, x)=p(x)$ and $W(+\infty, x)=0$.
\end{definition}

The following result shows that the asymptotic spreading speed $c^*$ given in 
Theorem~\ref{thmSP} coincides with minimal speeds of periodic traveling waves. 

\begin{theorem}\label{thmTW}
Assume that $0$ is an unstable solution of~\eqref{sta1} \emph{(i.e., $\lambda_1<0$)}. Let $c^*$ be the asymptotic spreading speed given in Theorem~\ref{thmSP}. Then there exists a periodic traveling wave $W(x-ct, x)$ connecting $q$ to $0$, if and only if $c\geq c^*$, where $q$ is the minimal positive periodic solution of system~\eqref{sta1}, given in Proposition~\ref{proSTA}-(iii).
\end{theorem}

In the following, we establish the existence of pulsating fronts for system~\eqref{eq}-\eqref{patch} under the assumptions that $l_1$ is sufficiently small, $l_2$ is sufficiently large, and ${\int_{0}^{K_2}\!f_2(s)\ds>0}$. Under these conditions, the steady state $0$ is stable, and the system admits a positive bounded periodic steady state $p_{l_1,l_2}$ that is strongly stable from below (see Proposition~\ref{pro_pl1l2} and Lemma~\ref{lem_strong}). Therefore, system~\eqref{eq}-\eqref{patch} exhibits a bistable structure, which enables us to establish the existence of pulsating fronts connecting $p_{l_1,l_2}$ and $0$. For simplicity, we write $p^*:=p_{l_1,l_2}$ throughout this paper, keeping in mind its dependence on $l_1$ and $l_2$. 

\begin{theorem}\label{bisTW}
	Assume that $\int_0^{K_2} f_2(s) \mathrm{d} s>0$. There exist $\hat{l}_1>0$ and $\hat{l}_2>0$ such that for all $0<l_1<\hat{l}_1$ and $l_2>\hat{l}_2$, system~\eqref{eq}-\eqref{patch} admits a pulsating front $u(t,x)=\widetilde{W}(x-ct, x)$ connecting $p^*$ to $0$ with speed $c>0$, where $p^*$ is a positive bounded periodic steady state of \eqref{eq}-\eqref{patch}.
\end{theorem}


\subsection{Extinction}

In this section, we analyze extinction phenomena by considering two distinct cases. We first prove that extinction occurs under the assumption that the principal eigenvalue $\lambda_1$ of system~\eqref{eigenv} is positive  and that the initial datum $u_0$ is sufficiently small, as the following result shows.

\begin{theorem}\label{thmEXT1}
Assume that $\lambda_1 > 0$, where $\lambda_1$ denotes the principal eigenvalue of system~\eqref{eigenv}. Let $u$ be the solution of~\eqref{eq}-\eqref{patch} with a nonnegative, continuous and compactly supported initial datum $u_0 \not \equiv 0$. There exists $\varepsilon>0$ such that, if $\|u_0\|_{L^{\infty}(\R)}\leq\varepsilon$, then the solution of $u$ to~\eqref{eq}-\eqref{patch} goes to extinction, that is, $\|u(t,\cdot)\|_{L^{\infty}(\R)}\allowbreak\to0$ as $t\to+\infty$.
\end{theorem}

When $u_0$ is not small enough but located in a bistable patch, we  show that the solution can still go extinct provided that the bistable patch is sufficiently large and not favorable i.e. $\int_{0}^{K_2} f_2(s)\ds<0$,  $l_1$ is relatively small, whereas $l_2$ is large enough, we also require that the support of $u_0$ is fully included in a bistable patch and sufficiently far enough from the interface, see the precise statement in~Theorem~\ref{thmEXT2} below. In contrast, when $l_2$ is fixed and $l_1$ is sufficiently small, whether extinction occurs remains an open problem.

\begin{theorem}\label{thmEXT2}
Assume that $\int_{0}^{K_2} f_2(s)\ds<0$. Let $R>0$, $A>0$, and $0<l_1< L_1^{c}$ be fixed. Let $u$ be the solution of~\eqref{eq}-\eqref{patch} with a nonnegative continuous and compactly supported initial datum $u_0 \not \equiv 0$. There exists $\delta>0$ large enough such that, if $l_2\geq R+2\delta$ and $\supp(u_0)$ is included in a bistable patch with $\diam(\supp(u_0)) \le R$,
$\dist(\supp(u_0), S) \ge \delta$, and $\|u_0\|_{L^{\infty}(\R)}\leq A$, then $\|u(t,\cdot)\|_{L^{\infty}(\R)}\to0$ as $t\to+\infty$.
\end{theorem}


\subsection{Discussion}\label{sec25}

The present work differs from classical frameworks in that we consider a periodic environment alternating between KPP and bistable patches. To the best of our knowledge, this is among the first studies to investigate propagation phenomena in a reaction-diffusion system combining monostable and bistable dynamics in a spatially periodic setting. This hybrid structure gives rise to new behaviors that cannot be captured by either framework alone. More precisely, we establish several sets of sufficient conditions for persistence and extinction, thereby providing a systematic characterization of the long-time dynamics of the system~\eqref{eq}-\eqref{patch}. These results highlight the crucial roles of the stability of the trivial solution and the periodic steady state, the relative sizes of the patches, and the magnitude of the initial data. In particular, when the trivial solution $0$ is unstable for system~\eqref{sta1} (e.g., $l_1>l_1^c$ or $0<l_1<L_1^c$ and $l_2<l_2^c$), solutions of~\eqref{eq}-\eqref{patch} with nonnegative, continuous, and compactly supported initial data exhibit uniform persistence and even propagation with positive spreading speed, see Theorems~\ref{thmPERS} and~\ref{thmSP}. Furthermore, pulsating traveling waves with a semi-infinite interval of admissible speeds exist in this case, see Theorem~\ref{thmTW}.

On the other hand, whether the trivial solution $0$ be stable or unstable for system~\eqref{sta1}, if the initial data are sufficiently large and concentrated within a sufficiently large bistable patch, and if $\int_{0}^{K_2}f_2(s)\ds>0$, then local persistence occurs, see Theorem~\ref{thmNOEXT}. For the homogeneous bistable equation~\eqref{eqbistable} in the whole line (for which $0$ is unstable), due to the invariance of~\eqref{eqbistable} by translation, the same hypotheses lead to the local convergence as $t\to+\infty$ of the solutions to $K_2$ and to their leftwards and rightwards propagation with speed $c_2>0$~\cite{AW-1978,FM-1977}. For our patchy problem~\eqref{eq}-\eqref{patch}, under the hypotheses of Theorem~\ref{thmNOEXT}, but without the instability of the trivial solution $0$, whether the solutions exhibit propagation remains an open problem, one obstacle in the proof being due to the heterogeneous nature of~\eqref{eq}-\eqref{patch} and its non-invariance by translation: roughly speaking, it is not clear whether initial large bumps in a given patch can cross the interfaces and then propagate to the other patches.

In addition, under the conditions that $l_1>0$ is sufficiently small, $l_2>0$ is sufficiently large, and $\int_{0}^{K_2} f_2(s)\ds>0$, we establish in Theorem~\ref{bisTW} the existence of pulsating traveling waves. Furthermore, if $\lambda_1>0$ and the initial data are sufficiently small, then extinction occurs, see Theorem~\ref{thmEXT1}. Lastly, even for large initial data, extinction may still occur if the initial data is supported within a bistable patch that is not favorable, see Theorem~\ref{thmEXT2}.

The main difficulty of the present patchy model~\eqref{eq}-\eqref{patch} lies in the interaction between two fundamentally different types of nonlinear dynamics, namely monostable (KPP-type) and bistable dynamics, within a spatially periodic heterogeneous environment. In KPP-type equations in periodic media (see~\cite{HLZ-2024}), the trivial steady state is linearly unstable, and solutions converge locally uniformly to the unique positive periodic steady state, with spreading speeds characterized by variational formulas involving principal eigenvalues. In contrast, classical bistable equations admit a stable trivial state, and propagation does not always occur; it depends on whether the initial data exceed a critical threshold (see, e.g.,~\cite{DM-2010}).

Another key difficulty lies in understanding the mechanisms governing persistence and extinction. Unlike the KPP setting, where spreading is largely determined by the linearization at the trivial state, the spreading behavior in the present model is influenced not only by the stability of the trivial solution but also by the stability of the periodic steady state, the relative sizes of the patches, and the magnitude of the initial data.

Although this paper has established results on persistence, extinction, asymptotic propagation speeds, and the existence of pulsating traveling waves for the KPP-bistable periodic patch model, several questions remain open. In particular, the presence of bistable nonlinearities suggests the possibility of propagation failure (or blocking) when the relative scales of the KPP and bistable patches vary. A rigorous characterization of the parameter regimes leading to such blocking remains an important direction for future work.

In addition, our results naturally raise the question of propagation phenomena in a periodically heterogeneous environment with the following reaction term:
\begin{equation*}\label{fnew}
	\tilde{f}(x,s)=
	\begin{cases}
		f_1(s), & x\in (nl-l_1,nl),\\
		-ms, & x\in (nl,nl+l_2),
	\end{cases}
\end{equation*}
where $n\in\mathbb Z$, $m>0$ denotes the mortality rate in patch 2, and $f_1$ satisfies \eqref{f1}. In this setting, patch 1 remains a favourable habitat governed by KPP dynamics, whereas patch 2 acts as a hostile habitat in which the population experiences a constant death rate. Clearly, $0$ is a stationary solution of \eqref{eq}--\eqref{patch} with $f$ replaced by $\tilde f$. Moreover, if the principal eigenvalue associated with the linearized problem (with $f_s(x,0)$ replaced by~$\tilde f_s(x,0)$ in \eqref{eigenv}) satisfies $\lambda_1<0$, then the arguments used in Proposition~\ref{proSTA} still guarantee the existence of a positive bounded periodic stationary solution of the corresponding elliptic problem \eqref{sta1}. Consequently, persistence of solutions may still occur despite the presence of hostile patches. On the other hand, when $\lambda_1>0$, one expects extinction for sufficiently small initial data, by analogy with Theorem~\ref{thmEXT1}. A natural and challenging problem is to determine the propagation dynamics of solutions in this setting and, in particular, to characterize the spreading speed and derive an explicit formula whenever possible. The analysis appears to be substantially different from the KPP-bistable patch structure considered in the present paper. Indeed, the hostile patches no longer admit any positive equilibrium, and therefore several arguments relying on the assumptions of bistable patch are no longer applicable. Understanding how the mortality rate~$m$, the patch sizes, and the initial data influence persistence, extinction, and spreading phenomena remains an interesting direction for our future work.

Furthermore, the homogenization limit, in which the patch lengths $l_1$ and $l_2$ simultaneously tend to zero, presents an intriguing challenge. Determining whether the hybrid system converges to an effective KPP-type, bistable, or intermediate equation is of significant theoretical interest. Similarly, investigating the propagation dynamics under extreme diffusion regimes, namely, fast diffusion ($d_1\to\infty$ and $d_2\to\infty$) and slow diffusion ($d_1\to0$ and $d_2\to0$) would further clarify the interplay between spatial heterogeneity and nonlinear response.

\medspace

\textbf{Outline of the paper.} 
In Section~\ref{SecPER}, we investigate the persistence of solutions to system~\eqref{eq}-\eqref{patch}, including the proofs of Theorems~\ref{thmPERS} and~\ref{thmNOEXT}. Our analysis is based on the study of an associated eigenvalue problem and the construction of suitable lower solutions via the comparison principle. In particular, we obtain both uniform and local persistence results. Section~\ref{Sec-SP-TW} is devoted to the study of asymptotic spreading behavior and the existence of pulsating traveling waves, distinguishing between the stable and unstable cases of the trivial steady state.  In this section, we also provide the proofs of Theorems~\ref{thmSP},~\ref{thmTW} and~\ref{bisTW}. Finally, in Section~\ref{Sec-EXT}, we apply Theorems~\ref{thmEXT1} and~\ref{thmEXT2} to characterize population extinction under two distinct scenarios.


\section{Persistence: proofs of Theorems~\ref{thmPERS} and~\ref{thmNOEXT}}\label{SecPER}

This section focuses on the persistence of solutions to system~\eqref{eq}-\eqref{patch}. To this end, the Dirichlet principal eigenvalue problem plays a central role. For the reader's convenience, we summarize the relevant results from~\cite[Section 4.2]{HLZ-2024} below. For any $R>0$, there exist a unique real number (principal eigenvalue) $\lambda_1^R$ and a unique nonnegative continuous function (principal eigenfunction) $\psi^R$ on $[-R, R]$, such that $\psi^R|_{\bar{I}}\in C^{\infty}(\bar{I})$ for each patch $I$ in $[-R,R]$, satisfying 
\begin{equation}\label{Dirichlet-eig}
	\begin{cases}
		-d(x)(\psi^R )^{''}(x)-f_s(x, 0)\psi^R(x)=\lambda_1^R\psi^R(x),  & x \in (-R,R) \backslash S, \\
		\psi^R(x^{-}) =\psi^R(x^{+}),\quad (\psi^R)'(x^{-})=\sigma( \psi^R)'(x^{+}), & x =nl\in(-R,R), \\
		\psi^R(x^{-}) =\psi^R(x^{+}), \quad\sigma (\psi^R)'(x^{-})=(\psi^R)'(x^{+}), & x =nl+l_2\in(-R,R),\\
		\psi^R>0 \text{ in }(-R,R), \psi^R(\pm R)=0,\|\psi^R\|_{L^{\infty}(-R,R)}=1.
	\end{cases}
\end{equation}
Moreover, Lemmas 4.3, 4.4, and 4.5 in~\cite{HLZ-2024} are collected below, indicating the relationships among the different notions of principal eigenvalues introduced above.
\begin{lemma}[{\rm \cite[Lemmas 4.3-4.5]{HLZ-2024}}]\label{lem-k}
Let $(\lambda_1, \phi)$ be the principal eigenpair of the eigenvalue problem~\eqref{eigenv}. For any $R > 0$, let $(\lambda_1^R, \psi^R)$ be the principal eigenpair of the Dirichlet eigenvalue problem~\eqref{Dirichlet-eig}. Then the following properties hold: 
	\begin{itemize}
		\item[(i)] For any $R>0$, one has $\lambda_1^R > \lambda_1$.
		\item[(ii)] The function $R \mapsto \lambda_1^R$ is decreasing for $R>0$, and
		$$
		\lim_{R\to +\infty}\lambda_1^R=\lambda_1.
		$$
	\end{itemize}
\end{lemma}

\medspace

\begin{proof}[Proof of Theorem~\ref{thmPERS}]
Let $u$ be the solution of~\eqref{eq}-\eqref{patch} with a nonnegative, continuous, and compactly supported initial datum $u_0 \not\equiv 0$.

First of all, Lemma~\ref{lem-k} implies that there is $R>0$ such that
\begin{align*}\label{lambda1}
\lambda_1^R<\frac{\lambda_1}{2}<0,
\end{align*}
where $(\lambda_1^R,\psi^R)$ denotes the principal eigenpair of~\eqref{Dirichlet-eig}. 
Since $\left.f(x, \cdot)\right|_I=f_i$ is of class $C^1(\mathbb{R})$ for each $x \in \mathbb{R} \backslash S$ belonging to a patch $I$ of type $i \in\{1,2\}$, it follows that there exists $\kappa_0>0$ small enough such that
$$
f(x, \kappa \psi^R(x)) \geq \kappa \psi^R(x) f_s(x, 0)+\frac{\lambda_1}{2} \kappa \psi^R(x),\quad x\in(-R,R)\backslash S
$$
 for all $\kappa \in(0, \kappa_0]$. Therefore, for any given $\kappa\in(0,\kappa_0]$,  the function $\kappa \psi^{R}$ satisfies
\begin{equation*}
\begin{aligned}
-d(x)\kappa (\psi^{R})^{\prime \prime}(x)-f(x, \kappa \psi^{R}(x)) \leq&-d(x) \kappa (\psi^{R})^{\prime \prime}(x)-\kappa \psi^{R}(x) f_s(x, 0)-\frac{\lambda_1}{2} \kappa \psi^{R}(x)\\
=&\frac{\lambda_1}{2} \kappa \psi^{R}(x)<0
\end{aligned}
\end{equation*}
for all $x \in (-R,R)\backslash S$, together with  the interface conditions in~\eqref{Dirichlet-eig}.  Fix now $\kappa^*\in(0,\kappa_0]$ such that  $\kappa^* \psi^{R}<u(1,\cdot)$ in $\R$.

Let $v(t,x)$ be the solution to the Cauchy problem~\eqref{eq}-\eqref{patch} for $(t,x)\in(0,+\infty)\times\R$ with initial datum $v_0$ given by 
$$
v_0(x) = 
\begin{cases}
	\kappa^* \psi^R(x), & x \in [-R, R],\\
	0, & \text{otherwise}.
\end{cases}
$$
Then, 
\begin{equation}\label{M1}
	v_0(x)<u(1,x)\leq M_1:=\max \left(K_1, K_2,\left\|u_0\right\|_{L^{\infty}(\mathbb{R})}\right)\quad \text{ for}~x\in\R.
\end{equation}
   Applying the comparison principle (Proposition~\ref{proCP}), we obtain that
\begin{equation}\label{vtr}
	v(t,x)< u(1+t, x)\leq M_1~~~~~\text{for}~~(t,x)\in(0,+\infty)\times\R.
\end{equation}
 
 On the other hand, since $v_0(x)$ is a generalized subsolution of system~\eqref{eq}-\eqref{patch},  it immediately follows that for any given $\tau>0$, there holds $v_0(x)<v(\tau,x)$ for $x\in\R$, which together with the comparison principle (Proposition~\ref{proCP}) implies that $v(t,x)<v(t+\tau,x)$ for  $(t,x)\in(0,+\infty)\times\R$. Namely, the function $v(t,x)$ is strictly increasing with respect to $t$ on $(0,+\infty)\times\R$. 
From the Schauder estimates in~\cite[Theorem 2.2]{HLZ-2024}, the solution $v(t, x)$ converges as $t \to +\infty$ locally uniformly for $x\in\mathbb{R}$ to a positive bounded solution $q(x)$ of~\eqref{sta1}. Furthermore, $v(t,\cdot)|_{\bar{I}}\to q|_{\bar{I}}$ in $C^2(\bar{I})$ for each patch $I\subset\R$ as $t\to+\infty$.
Thus, by passing~\eqref{vtr} to the limit as $t\to+\infty$, we have that 
\begin{equation*}\label{vtr1}
0<v_0(x)<q(x)\leq\liminf_{t\to+\infty}u(t,x)~~~\text{for}~x\in\R.
\end{equation*}
We further infer from Proposition~\ref{proSTA} that $\inf_{x\in \R}q(x)>0$, which leads to 
\begin{equation}\label{thm-per}
\inf_{x\in \R}\left(\liminf_{t\to+\infty}u(t,x)\right)\geq \inf_{x\in \R}q(x)>0.
\end{equation}
This completes the proof of Theorem~\ref{thmPERS}.
\end{proof}

Next, we investigate the local persistence of solutions to~\eqref{eq}-\eqref{patch}. We mainly use a carefully constructed subsolution and the comparison principle to prove Theorem~\ref{thmNOEXT}. We begin with the following auxiliary lemma, which establishes the existence of solutions to elliptic equations in large intervals; see~\cite[Theorem A]{BL-1980} and~\cite[Lemma 4.2]{HLZ-2022}.

\begin{lemma}\label{lem-Phi2}
	Assume that $\int_{0}^{K_2}f_2(s) \mathrm{d}s>0$. Then there exist $R>0$ and a function $\phi_2$ of class $ C^2([-R,R])$ such that
	\begin{align}
		\label{4.53}
		\begin{cases}
			\ d_2\phi_2''+f_2(\phi_2)=0~~&\text{in}~[-R,R],\cr
			\ 0\le \phi_2<K_2~~&\text{in}~[-R,R],\cr
			\ \phi_2(\pm R)=0,~~&\cr
			\ \displaystyle\mathop{\max}_{[-R,R]}\phi_2=\phi_2(0)>\theta.
		\end{cases}
	\end{align}
\end{lemma}

To prove Theorem~\ref{thmNOEXT}, we take an indirect approach to prove the following result as a first step.

\begin{lemma}\label{lem-nonex1}
	Assume that $\int_{0}^{K_2}f_2(s) \mathrm{d}s>0$. Let  $R>0$ and $\phi_2$ be as in Lemma~$\ref{lem-Phi2}$. Let $u$ be the solution to~\eqref{eq}-\eqref{patch} with a nonnegative continuous and compactly supported initial datum $u_0\not\equiv 0$. If 
	$u(T,\cdot)\geq \phi_2(\cdot-x_0)\text{ in }\R$ for some $T>0$ and $x_0\in\R$, then the conclusion~\eqref{noex} of Theorem~$\ref{thmNOEXT}$ holds true.
\end{lemma}	
\begin{proof}
	Let  $R>0$, $\phi_2\in C^2([-R,R])$, $T>0$, $x_0\in\R$, and $u_0$ be as in the statement. We choose $x_0\in\R$ such that $[x_0-R,x_0+R]\subset$ patch 2. Let $v$ and $w$ be, respectively, the solutions to~\eqref{eq}-\eqref{patch} with initial datum $v_0=M_1$, where $M_1$ is given by~\eqref{M1}, and $w_0$ is given by 
	\begin{equation*}\label{w0}
		w_0(x)=
		\begin{cases}
			\phi_2(x-x_0),& x\in [x_0-R,x_0+R],\\
			0,& x\in\mathbb{R}\setminus[x_0-R,x_0+R].
		\end{cases}
	\end{equation*}
	Then Proposition~\ref{proCP} yields $0<w(t,x)\le u(t+T,x)\le v(t,x)\le M_1$ for all $t> 0$ and  $x\in\mathbb{R}$. Moreover, as in the proof of Theorem~\ref{thmPERS}, $w$ is increasing with respect to $t$ in $[0,+\infty)\times\R$, whereas $v$ is nonincreasing with respect to $t$ in $[0,+\infty)\times\R$. From the parabolic estimates of~\cite[Theorem 2.2]{HLZ-2024}, the functions $w(t,\cdot)$ and $v(t,\cdot)$ converges as $t\to+\infty$, locally uniformly in $\R$, to positive continuous and bounded solutions $p_1$ and $p_2$ of~\eqref{sta1} such that $w(t,\cdot)|_{\bar{I}}\to p_1|_{\bar{I}}$ and $v(t,\cdot)|_{\bar{I}}\to p_2|_{\bar{I}}$ in $C^2(\bar{I})$ for each patch $I\subset\R$ as $t\to+\infty$, respectively. Moreover, 
	\begin{equation*}
		\label{4.54}
		0\le w_0<p_1\le \liminf_{t\to+\infty} u(t,\cdot)\le \limsup_{t\to+\infty} u(t,\cdot)\le p_2\le M_1,~~\text{locally uniformly in}~\mathbb{R}.
	\end{equation*}
	This implies that 
	\begin{equation*}
		\inf_{x\in H}\left(\liminf_{t\to+\infty}u(t,x)\right)>0 
	\end{equation*}
	for all bounded interval $H\in\R$. This completes the proof of Lemma~\ref{lem-nonex1}.
\end{proof} 

\vs

Now we are in a position to prove Theorem~\ref{thmNOEXT}.

\medspace

\begin{proof}[Proof of Theorem~$\ref{thmNOEXT}$]
	Fix any $\eta>0$. Let $l_2^{**}\ge2$ be a constant to be fixed later, and let $x_{l_2^{**}}\in\R$. Let $u$ be the solution to~\eqref{eq}-\eqref{patch} with a nonnegative continuous and compactly supported initial datum $u_0\not\equiv 0$ satisfying $u_0\ge \theta+\eta$ in an interval of size $l_2^{**}$ included in patch $2$, say $(x_{l_2^{**}}-l_2^{**}/2,x_{l_2^{**}}+l_2^{**}/2)$. Now let
	$u_{l_2^{**}}$ be the solution to the Cauchy problem~\eqref{eq}-\eqref{patch} with initial datum 
	\begin{align*}
		u_{{l_2^{**}}}(0,\cdot)=\begin{cases}
			\ \theta+\eta~&\text{in}~[x_{l_2^{**}}-l_2^{**}/2+1,x_{l_2^{**}}+l_2^{**}/2-1],\cr
			\ 0 &\text{in }\R\setminus(x_{l_2^{**}}-l_2^{**}/2,x_{l_2^{**}}+l_2^{**}/2),
		\end{cases}
	\end{align*}	
	and $u_{l_2^{**}}(0,\cdot)$ is affine in $[x_{l_2^{**}}-l_2^{**}/2,x_{l_2^{**}}-l_2^{**}/2+1]$ and in $[x_{l_2^{**}}+l_2^{**}/2-1,x_{l_2^{**}}+l_2^{**}/2]$. Since $u_{l_2^{**}}$ is uniformly bounded by $M_2:=\max\{K_1, K_2, \theta+\eta\}$ (independently of $l_2^{**}$), we may apply local (around the position $x_{l_2}^{**}$) parabolic estimates in bistable patch to conclude that, for any $A>0$,
	\begin{equation}
		\label{local convergence}
		u_{l_2^{**}}(t,x)\to\xi(t)~\text{as}~{l_2^{**}}\to+\infty~~\text{locally in}~t\ge 0,~\text{uniformly in}~x\in[x_{l_2^{**}}-A,x_{l_2^{**}}+A],
	\end{equation}
	where $\xi$ is the solution of the following ODE equation: 
	\begin{equation*}
		\begin{cases}
			\xi^{\prime}(t)=f_2(\xi(t)), &t>0,\\
			\xi(0)=\theta+\eta.
		\end{cases}
	\end{equation*}
	
	Let~$R>0$ and $\phi_2\in C^2([-R,R])$ be as in Lemma~\ref{lem-Phi2}, and pick $\varepsilon\in(0,K_2-\phi_2(0))$. By~\eqref{f2}, we have that
	$$
	\xi(t)\to K_2 \text{ as }t\to+\infty.
	$$
	Thus there is $T>0$ such that $\xi(T)\ge \phi_2(0)+\varepsilon$. By~\eqref{4.53} and~\eqref{local convergence}, we may choose $l_2^{**}\in(\max(2R,2),+\infty)$ sufficiently large such that, whenever $[x_{l_2^{**}}-R,x_{l_2^{**}}+R]$ is contained in patch 2, it holds that
	$$u_{l_2^{**}}(T,\cdot)>\xi(T)-\varepsilon\ge \phi_2(0)\ge\phi_2(\cdot-x_{l_2^{**}})~~\text{in}~[x_{l_2^{**}}-R,x_{l_2^{**}}+R].$$
	It follows from comparison principle that 
	\begin{equation}\label{uT}
		u(T,\cdot)\ge u_{l_2^{**}}(T,\cdot)>\phi_2(0)\ge\phi_2(\cdot-x_{l_2^{**}})~~\text{in}~[x_{l_2^{**}}-R,x_{l_2^{**}}+R].
	\end{equation}
	Furthermore, since $u_0\not\equiv0$ and $\phi_2(\cdot-x_{l_2^{**}})$ is extended by $0$ outside $[x_{l_2^{**}}-R,x_{l_2^{**}}+R]$, by the strong maximum principle and~\eqref{uT}, we obtain
	$$
	u(T,\cdot)>\phi_2(\cdot-x_{l_2^{**}})\text{ in }\R.
	$$
	The conclusion of Theorem~\ref{thmNOEXT} then follows from Proposition~\ref{proCPBD} and Lemma~\ref{lem-nonex1} applied with initial datum $\psi(\cdot-x_{l_2^{**}})$ (extended by $0$ outside $[x_{l_2^{**}}-R,x_{l_2^{**}}+R]$).
\end{proof}


\section{Spreading speed and pulsating traveling waves: proofs of Theorems~\ref{thmSP},~\ref{thmTW} and~\ref{bisTW}}\label{Sec-SP-TW}


\subsection{Proofs of Theorems~\ref{thmSP} and~\ref{thmTW}}

In this subsection, we give the proofs of Theorems~\ref{thmSP} and~\ref{thmTW}. The
main approach is based on the abstract dynamical systems theory for monostable evolution systems established in the seminal work in \cite{W-2002} and further developed in~\cite{LZ-2007,LZ-2010}.  

\medspace

\begin{proof}[Proof of Proposition~\ref{proSTA}(iii)]
	The proof strategy follows a logic similar to that of Theorem~\ref{thmPERS}. We shall therefore focus on the key differences and only briefly outline the overlapping steps.
	
Assume that $0$ is an unstable solution of~\eqref{sta1}, i.e. $\lambda_1<0$, where $\lambda_1$ is the principal eigenvalue of the eigenvalue problem~\eqref{eigenv}. Suppose that $p(x)$ is a positive bounded solution to~\eqref{sta1} on $\mathbb{R}$.   

 Let $v$ be defined as in the proof of Theorem~\ref{thmPERS}, and take $u(t,x)=p(x)$ for $(t,x)\in\R_+\times\R$ for notional convenience. Repeating the arguments in the proof of Theorem~\ref{thmPERS} gives that
\begin{equation}\label{2.3-1}
	0<v(t,x)< u(1,x)=p(x)\leq \max \left(K_1, K_2,\left\|p\right\|_{L^{\infty}(\mathbb{R})}\right)
\end{equation}
for $(t,x)\in(0,+\infty)\times\R$. Moreover, since $v(t,x)$ is increasing with respect to $t$ on $(0,+\infty)\times\R$, it follows from parabolic estimates in~\cite[Theorem 2.2]{HLZ-2024} that $v(t,x)$ converges as $t\to+\infty$ locally uniformly in $\R$ to  a positive continuous and bounded solution $q(x)$ of~\eqref{sta1} such that $v(t,\cdot)|_{\bar{I}}\to q|_{\bar{I}}$ in $C^2(\bar{I})$ for each patch $I\subset\R$ as $t\to+\infty$. Thus,
passing to the limit $t\to+\infty$ in~\eqref{2.3-1}, we have that
$$
0<  q(x)\leq p(x) \quad \text{for}~x\in\R.
$$
Since the positive  bounded solution $p(x)$ to~\eqref{sta1} is chosen arbitrarily, it follows that $q(x)$ is the minimal one among them.  Furthermore, if $q$ is a minimal positive bounded solution of~\eqref{sta1}, so is the function $x\to q(x +l)$. This implies that $q$ is periodic in $\R$.
This completes the proof.
\end{proof}

\vs

Now we are in a position to give the proof of Theorem~\ref{thmSP}. 

\medspace

\begin{proof}[Proof of Theorem~\ref{thmSP}] To prove the result, we adopt the method from \cite[Section~5]{LZ-2010} and introduce the following notations:
$\mathcal{H} = \R$, $\mathcal{\tilde{H}}=l\Z$, $X=Y=\R$, $\beta=q$ and $\mathcal{M}=\mathcal{C}_q$, where $q$ is the minimal continuous positive periodic solution of system~\eqref{sta1}, given in Proposition~\ref{proSTA}-(iii), and $\mathcal{C}_q:=\{v\in\mathcal{C}:0\le v\le q\}$. Let $u(t, x; u_0)$ be the unique classical solution to the Cauchy problem~\eqref{eq}-\eqref{patch} with initial condition $u(0, \cdot ; u_0)=u_0 \in \mathcal{C}_q$, given by~\cite[Theorem 2.2]{HLZ-2024}. By the monotonicity in~\cite[Theorem 2.2]{HLZ-2024} and comparision principle then implies that $u(t, \cdot; u_0) \in \mathcal{C}_q$ for every $u_0 \in \mathcal{C}_q$ and $t \geq 0$. Then we can define a family of maps $\left\{Q_t\right\}_{t \geq 0}$ in $\mathcal{C}_q$ by
\begin{equation}\label{Qt}
Q_t(u_0)(x)=u(t, x ; u_0), \quad u_0 \in \mathcal{C}_q, t \geq 0, x \in \mathbb{R}.
\end{equation}
In particular, $Q_0(u_0)=u_0$ for every $u_0 \in \mathcal{C}_q$, and $Q_t(0)=0$ for every $t \geq 0$. 
	Since $q$ is the minimal positive bounded periodic solution of \eqref{sta1} by Proposition \ref{proSTA}(iii), it is also a solution to~\eqref{eq}-\eqref{patch}. By uniqueness, $u(t, \cdot; q)=q$ for each $t \geq 0$, that is, $Q_t(q)=q$. 
It is clear that the family of maps $\left\{Q_t\right\}_{t \geq 0}:\mathcal{C}_q\to \mathcal{C}_q$ satisfies the following properties:
\begin{itemize}
\item [(1)] $Q_0(u_0)=u_0$ for all $u_0 \in \mathcal{C}_q$;
\item [(2)] $Q_{t_1}\left(Q_{t_2}(u_0)\right)=Q_{t_1+t_2}(u_0)$ for all $t_1, t_2 \geq 0$ and for all $u_0 \in \mathcal{C}_q$;
\item [(3)] the map $(t, u_0) \mapsto Q_t(u_0)$ is continuous from $[0,+\infty) \times \mathcal{C}_q$ into $\mathcal{C}_q$, with $\mathcal{C}_q$ equipped with the compact open topology, that is, $Q_{t_m}\left(u_{0,m}\right) \rightarrow Q_t(u_0)$ as $m \rightarrow+\infty$ locally uniformly in $\mathbb{R}$ if $t_m \rightarrow t$ and $u_{0,m} \rightarrow u_0$ locally uniformly in $\mathbb{R}$ as $m \rightarrow+\infty$, with $\left(t_m, u_{0,m}\right) \in[0,+\infty) \times \mathcal{C}_q$.
\end{itemize}
Therefore, $\left\{Q_t\right\}_{t \geq 0}:\mathcal{C}_q\to \mathcal{C}_q$ is a semiflow. We also say that $\left\{Q_t\right\}_{t \geq 0}$ is monotone in $\mathcal{C}_q$ if, for every $t \geq 0, Q_t(u_0) \geq Q_t\left(v_0\right)$ in $\mathbb{R}$ provided $u_0 \geq v_0$ in $\mathbb{R}$, with $u_0, v_0\in \mathcal{C}_q$. As in the proof of~\cite[Proposition 6.1]{HLZ-2024}, one can show that the family $\{Q_t\}_{t\ge0}$ defined in~\eqref{Qt} is a monotone semiflow on $\mathcal{C}_q$.  Furthermore, for every $u_0 \in \mathcal{C}_q, a \in l \mathbb{Z}, t \geq 0$ and $x \in \mathbb{R}$, there holds $Q_t(u_0(\cdot+a))(x)=Q_t(u_0)(x+a)$. 

Moreover, using~\eqref{thm-per} in Theorem~\ref{thmPERS}, we derive that
$$0< q(x)\le\liminf_{t\to+\infty}u(t,x;u_0)\le \limsup_{t\to+\infty}u(t,x;u_0)\le q(x),\quad x\in\R.$$
This implies that 
\begin{equation}\label{Q}
	\lim_{t\to+\infty}|Q_t[u_0](x)-q(x)|=0 \text{ locally uniformly for } x\in\R.
\end{equation}

As a consequence, we conclude that the solution maps $Q_t: \mathcal{C}_q \rightarrow \mathcal{C}_q$ satisfy the following properties:
\begin{itemize}
\item [(E1)] for each $t \geq 0$, $Q_t$ is periodic, that is, $Q_t\left(T_a(u_0)\right)=T_a\left(Q_t(u_0)\right)$ for all $a \in l \mathbb{Z}$ and $u_0 \in \mathcal{C}_q$, where $T_y$ is the translation operator defined by $T_y(\tilde{u}_0)=\tilde{u}_0(\cdot+y)$ for $\tilde{u}_0 \in \mathcal{C}_q$ and $y \in l \mathbb{Z}$;
\item [(E2)] the set $\left\{Q_t\left(\mathcal{C}_q\right): t \geq 0\right\} \subset \mathcal{C}_q$ is uniformly bounded and, for each $t \geq 0, Q_t: \mathcal{C}_q \rightarrow \mathcal{C}_q$ is continuous;
\item [(E3)] for each $t>0$, the map $Q_t: \mathcal{C}_q \rightarrow \mathcal{C}_q$ is compact with respect to the compact open topology (as a consequence of the regularity estimates 
of~\cite[Theorem 2.2]{HLZ-2024});
\item [(E4)] for each $t \geq 0, Q_t$ is order-preserving (i.e., monotone);
\item [(E5)] for each $t>0, Q_t$ admits exactly the two periodic fixed points 0 and $q$ in $\mathcal{C}_q$. Indeed, on the one hand, one knows that $0$ and $p$ are two fixed points; on the other hand, for each $u_0\in \mathcal{C}_q\backslash\{0, p\}$, one has $Q_{mt} (u_0) \to p$ locally uniformly in $\R$ as $m\to+\infty$ by~\eqref{Q} (and even uniformly in $\R$ if $u_0$ is periodic), hence $u_0$ can not be a fixed point of $Q_t$.
\end{itemize}

To prove statement~(i), it suffices to verify that the initial datum $u_0$ satisfies the assumptions of~\cite[Theorem 5.1(1)]{LZ-2010}. More precisely, we need to show that $u_0 \in \mathcal{C}_q$, with
$$
0\leq u_0\leq \varpi\ll q,
$$
for some $l$-periodic function $\varpi\in \mathcal{C}_q$, and that $u_0(x)=0$ for $x$ outside a bounded interval. Since $u_0$ is continuous, compactly supported, and satisfies $0\leq u_0(x)<q(x)$ for $x\in\R,$ it follows from \eqref{thm-per} and the periodicity of $q$ that
$q_{\min}:=\min_{x\in\mathbb R}q(x)>0.$ Consequently,
$$
\epsilon_0:=\inf_{x\in\mathbb R}\{q(x)-u_0(x)\}>0.
$$
Moreover, by the compact support of $u_0$, we have $\epsilon_0\leq q_{\min}$. Choosing $\epsilon\in(0,\epsilon_0]$, we define
$$
\varpi(x):=q(x)-\epsilon,\qquad x\in\mathbb R.
$$
Then $\varpi$ is continuous and $l$-periodic, hence $\varpi\in\mathcal C_q$. Furthermore,
$0\leq u_0\leq \varpi\ll q$ in $\R$, since $0<\epsilon\leq\epsilon_0\leq q_{\min}$.

Therefore, by~\cite[Theorem 5.1(1)]{LZ-2010}, the time-$1$ map $Q_1:\mathcal C_q\to\mathcal C_q$ admits rightward and leftward asymptotic spreading speeds $c_+^*,c_-^*\in\mathbb R$, in the sense that, if $u_0$ is continuous, compactly supported, and satisfies $0\leq u_0(x)<q(x)$ for $x\in\R$, then
\begin{equation}\label{sp1}
	\begin{cases}
		\displaystyle
		\lim_{n\to+\infty,\ x\geq cn}Q_n(u_0)(x)=0,
		& c>c_+^*,\\[1ex]
		\displaystyle
		\lim_{n\to+\infty,\ x\leq -c'n}Q_n(u_0)(x)=0,
		& c'>c_-^*.
	\end{cases}
\end{equation}
We next show that 
$$c_+^*=c_-^*>0.$$
After the spatial translation $x\mapsto x+l_1/2$, the center of patch~1 is shifted to the origin, and system~\eqref{eq}--\eqref{patch} becomes symmetric with respect to $x=0$. Define the reflection operator
$$
\mathcal R[u](x):=u(-x).
$$
Then $\mathcal C_q=\mathcal R[\mathcal C_q],$ and the semiflow $\{Q_t\}_{t\geq0}$ commutes with $\mathcal R$, namely, $Q_t[\mathcal R[u]]=\mathcal R[Q_t[u]].$ Hence, by~\cite[Remark 3.1]{LZ-2010}, we have that
$$
c^*:=c_+^*=c_-^*.
$$

On the other hand, \eqref{Q} implies that $\lim_{t\to+\infty}|u(t,\cdot;u_0)-q|=0$ locally uniformly in $\mathbb R$. Therefore, there exists $T\in\mathbb N$ such that
$u(T,\cdot;u_0)\geq u_0(\cdot\pm l)$ in $\R.$ Using~\cite[Theorem 2.2]{HLZ-2024} together with the monotonicity of the semiflow $\{Q_t\}_{t\geq0}$ in $\mathcal C_q$, we obtain
$u(2T,\cdot;u_0)\geq u_0(\cdot\pm2l)$ in $\R$.
By induction, $u(mT,\cdot;u_0)\geq u_0(\cdot\pm ml)$ in $\R$ for all $m\in\mathbb N.$
In other words, $Q_{mT}(u_0)\geq u_0(\cdot\pm ml)$ in $\R$ for all $m\in\mathbb N$. Combining this property with \eqref{sp1}, we deduce that $c^*\geq\frac{l}{T}>0.$

Finally, applying~\cite[Theorem 5.2(1)]{LZ-2010} together with \eqref{sp1}, we conclude that
$$
\lim_{t\to+\infty,\ |x|\geq ct}Q_t(u_0)(x)=0,
\qquad c>c^*.
$$

(ii) Similarly, using \cite[Theorem 5.1(2)]{LZ-2010} together with Theorem~\ref{thmPERS}, we conclude that the the time-$1$ map $Q_1: \mathcal{C}_q\to \mathcal{C}_q$ admits an asymptotic spreading speed $c^*>0$ in the following sense: if $u_0\in\mathcal C_q$ with $u_0\not\equiv0$,
then
\begin{equation}\label{sp2}
	\lim_{n\to+\infty,\ |x|\leq cn}
	|Q_n(u_0)(x)-q(x)|=0
\end{equation}
for every $0\leq c<c^*$.
Indeed, from the proof of Theorem~\ref{thmPERS} and the definition of $q$, for every initial datum $u_0\in \mathcal{C}_q$ satisfying $u_0\not\equiv0$,
\begin{equation*}
\lim_{t\to+\infty}|u(t,\cdot;u_0)-q|=0 \text{ locally uniformly in } \R.
\end{equation*}
Take any $\delta>0$, such that $\delta<\min_{\R} q$. Let $r_{\delta}>0$ be the constant given by \cite[Theorem 5.1(2)]{LZ-2010}. Then there exists $T>0$ such that
$u(T,x;u_0)\geq \delta$ on $[-r_{\delta}, r_{\delta}]$. Hence the initial datum $u(T,\cdot;u_0)$  satisfies the assumptions of \cite[Theorem 5.1(2)]{LZ-2010}. Therefore,
\begin{equation*}
	\lim_{n\to+\infty,\ |x|\leq cn}	|Q_n(u(T,\cdot;u_0))(x)-q(x)|=0,\quad 0\leq c<c^*.
\end{equation*}
By the semigroup property (2), namely, 
$$
Q_n(u(T,\cdot;u_0))=Q_n(Q_T(u_0))=Q_{n+T}(u_0),
$$
we have that
$$
\lim_{n\to\infty,\ |x|\le cn}
|Q_{n+T}(u_0)(x)-q(x)|=0,\quad 0\leq c<c^*,
$$
which implies that~\eqref{sp2} holds.

Therefore, applying~\cite[Theorem 5.2(2)]{LZ-2010} together with \eqref{sp2}, we conclude that if $u_0(x)\in\mathcal{C}_q$ with $u_0\not\equiv 0$, then
$$
\lim_{t\to+\infty,\ |x|\leq ct}
|Q_t(u_0)(x)-q(x)|=0.
$$
This completes the proof of Theorem~\ref{thmSP}.
\end{proof}

\medspace

\begin{proof}[Proof of Theorem~\ref{thmTW}] By applying~\cite[Theorems~5.2 and~5.3]{LZ-2010}, we directly obtain the existence of time-nonincreasing periodic traveling waves for problem~\eqref{eq}-\eqref{patch} with all and only those speeds $c\geq c^*$. Furthermore, an argument similar to that in the proof of~\cite[Theorem~2.8]{HLZ-2024} shows that these periodic traveling waves are in fact strictly monotone in time. This completes the proof of Theorem~\ref{thmTW}.
\end{proof}


\subsection{Proof of Theorem~\ref{bisTW}}

In this subsection, we present the proof of Theorem~\ref{bisTW}. To establish the existence of pulsating traveling waves, we apply the abstract theory developed in~\cite{FZ-2015} and further employed in~\cite{DHZ-2017}. However, the verification of the abstract assumptions in our setting is nontrivial due to the heterogeneous KPP-bistable structure. Specifically, we verify that the semiflow generated by~\eqref{eq}-\eqref{patch} satisfies the general assumptions for bistable monotone semiflows when $l_1$ is sufficiently small, $l_2$ is sufficiently large, and~$\int_{0}^{K_2}f_2(s)\ds>0$.

To do so, we first need to define precisely the notion of stability of positive periodic steady states. Let $\bar{p}:\R\to[0,p]$ denote a steady state of \eqref{eq}-\eqref{patch}, where $p$ is a positive periodic steady state of~\eqref{eq}-\eqref{patch}. For any $R>0$, let~$\bar{\lambda}_{1}^R(l_1,l_2,\bar{p})$ be the unique real number $\bar{\lambda}$ such that there exists a function $\psi^R:\mathbb{R}\to\mathbb{R}$ satisfying $\psi^R|_{\overline{I}}\in C^\infty(\overline{I})$ for each patch $I$, and
\begin{equation}\label{truncted}
	\begin{cases}
		-d(x)(\psi^R )^{''}(x)-f_s(x, \bar{p})\psi^R(x)=\bar{\lambda}\psi^R(x),  & x \in (-R,R) \backslash S, \\
		\psi^R(x^{-}) =\psi^R(x^{+}),\quad (\psi^R)'(x^{-})=\sigma (\psi^R)'(x^{+}), & x =nl\in(-R,R), \\
		\psi^R(x^{-}) =\psi^R(x^{+}), \quad\sigma (\psi^R)'(x^{-})=(\psi^R)'(x^{+}), & x =nl+l_2\in(-R,R),\\
		\psi^R>0 \text{ in }(-R,R), \psi^R(\pm R)=0,\|\psi^R\|_{L^{\infty}(-R,R)}=1.
	\end{cases}
\end{equation}
The real number $\bar{\lambda}_{1}^R(l_1,l_2,\bar{p})$ is the unique principal eigenvalue of equation~\eqref{truncted}, and $\psi^R$ is the corresponding eigenfunction.  In particular, when $\bar p\equiv 0$, problem~\eqref{truncted} reduces to the truncated eigenvalue problem~\eqref{Dirichlet-eig}. Thus, $\bar{\lambda}_{1}^R(l_1,l_2,0)=\lambda_{1}^R$. 

Furthermore, using the same arguments as in Section~\ref{eigenv}, we obtain a principal eigenvalue
$\bar{\lambda}_1:=\bar{\lambda}_1(l_1,l_2,\bar{p})\in\mathbb{R}$.
More precisely, $\bar{\lambda}_1$ is the unique real number for which there exists a unique continuous function  $\phi:\mathbb{R}\to\mathbb{R}$ satisfying $\phi|_{\overline{I}}\in C^\infty(\overline{I})$ for each patch $I$, and
\begin{equation}\label{eigenv2}
	\begin{cases}
		-d(x)\phi^{\prime\prime}(x) -f_s(x, \bar{p})\phi=\bar{\lambda}_1(l_1,l_2,\bar{p})\phi,  & x \in \mathbb{R} \backslash S, \\
		\phi(x^{-}) =\phi(x^{+}),\quad \phi^{\prime}(x^{-})=\sigma \phi^{\prime}(x^{+}), & x \in S_1, \\
		\phi(x^{-}) =\phi(x^{+}),\,\, \sigma \phi^{\prime}(x^{-})=\phi^{\prime}(x^{+}), & x \in S_2,\\
		\phi(x) \text{ is periodic}, \,\phi>0, \|\phi\|_{L^{\infty}(\R)}=1.
	\end{cases}
\end{equation}
Similarly when $\bar{p}\equiv0$, problem~\eqref{eigenv2} reduces to the eigenvalue problem~\eqref{eigenv} introduced in Section~\ref{eig} and hence $\bar{\lambda}_1(l_1,l_2,0)=\lambda_1.$ Moreover, we have that
\begin{equation}\label{prinbound1}
	-\max_{x\in\mathbb R}
	f_s(x,\bar{p})
	\le
	\bar{\lambda}_1(l_1,l_2,\bar p)
	\le
	-\min_{x\in\mathbb R}
	f_s(x,\bar{p}).
\end{equation}
Indeed, from the first equation of \eqref{eigenv2}, we obtain
\begin{equation}\label{lamphi}
	\bar{\lambda}_1(l_1,l_2,\bar p)=-d(x)\frac{\phi''(x)}{\phi(x)}-f_s(x,\bar p),\quad x\in\R\backslash S.
\end{equation}
Let $x_{\min}$ and $x_{\max}$ be points where $\phi$ attains its minimum and maximum, respectively. If $x_{\min}, x_{\max}\in \mathbb R\setminus S$, then
$$
\phi'(x_{\min})=0,\quad \phi''(x_{\min})\ge0,
\qquad
\phi'(x_{\max})=0,\quad \phi''(x_{\max})\le0.
$$
Using \eqref{lamphi}, we obtain
\begin{equation}\label{lam1}
	\bar{\lambda}_1(l_1,l_2,\bar p)\le-f_s(x_{\min},\bar p),\qquad	\bar{\lambda}_1(l_1,l_2,\bar p)\ge-f_s(x_{\max},\bar p).
\end{equation}
If $x_{\max}\in S$, then by interface conditions, $\phi'(x_{\max}^-)=\phi'(x_{\max}^+)=0.$
Then we have that $\phi''(x_{\max}^-)\leq0$ and $\phi''(x_{\max}^+)\leq0.$ This together with \eqref{lamphi} yields 
\begin{equation}\label{lam2}
	\bar{\lambda}_1(l_1,l_2,\bar p)\ge -f_s(x_{\max}^{\pm},\bar p).
\end{equation}
Similarly, if $x_{\min}\in S$, we obtain
\begin{equation}\label{lam3}
	\bar{\lambda}_1(l_1,l_2,\bar p) \le -f_s(x_{\min}^{\pm},\bar p).
\end{equation}
Therefore, using \eqref{lam1}, \eqref{lam2}, and \eqref{lam3}, we have that \eqref{prinbound1} holds.

 Furthermore, similar to Lemma~\ref{lem-k}, the principal eigenvalue $\bar{\lambda}_{1}^{R}(l_1,l_2,\bar p)$ is decreasing in $R$ and 
$$\lim_{R\to +\infty}\bar{\lambda}_{1}^R(l_1,l_2,\bar{p})=\bar{\lambda}_{1}(l_1,l_2,\bar{p}).$$

The following Proposition~\ref{pro_pl1l2} aims to show that system~\eqref{sta1} admits a linearly stable positive periodic solution $p^*$, provided $\int_{0}^{K_2}f_2(s)\ds>0$ and $l_2$ is sufficiently large.

\begin{proposition}\label{pro_pl1l2}
	Assume that $\int_0^{K_2} f_2(s)\,ds>0$. Then there exists $\hat{l}_2>0$ such that for any $l_1>0$ and $l_2>\hat{l}_2$, system~\eqref{sta1} admits a bounded positive and periodic solution $p^*$. Moreover, $p^*$ is linearly stable in the sense that 
	$\bar{\lambda}_1(l_1,l_2,p^*) \geq 0.$
\end{proposition}

\begin{proof} 
Let  $R>0$ and $\phi_2\in C^2([-R,R])$ be given by Lemma~\ref{lem-Phi2}. Assume  that $\hat{l}_2>2R$, then we can choose $x_0\in(0,l_2)$ such that $[x_0-R,x_0+R]\subset (0,l_2)$. We define the periodic function $\widetilde{w}_0$ by
	\begin{equation}\label{tildew0}
	\widetilde{w}_0(x)=
	\begin{cases}
		\phi_2(x-nl-x_0), & x\in [nl+x_0-R, nl+x_0+R],\, n\in\mathbb Z,\\
		0, & \text{otherwise}.
	\end{cases}
\end{equation}
Since $\left.f(x, \cdot)\right|_I=f_i$ is of class $C^1(\mathbb{R})$ for each $x \in \mathbb{R} \backslash S$ belonging to a patch $I$ of type $i \in\{1,2\}$, it follows that 
\begin{equation*}
	\begin{aligned}
		-d(x)(\widetilde{w}_0)^{\prime \prime}(x)-f(x, \widetilde{w}_0) =&-d_2(\widetilde{w}_0)^{\prime \prime}(x)-f_2( \widetilde{w}_0)=0
	\end{aligned}
\end{equation*}
for $x\in[nl+x_0-R, nl+x_0+R]$ and $n\in\Z$. Thus $\widetilde{w}_0$ is a stationary subsolution of~\eqref{eq}-\eqref{patch}. Set $M:=\max(K_1,K_2)$. Since $0\leq\phi_2<K_2$ in $[-R,R]$, we have that $0\leq\widetilde{w}_0\leq M$ in $\R$. Let $\widetilde{w}$ be the solution to~\eqref{eq}-\eqref{patch} with initial datum $\widetilde{w}_0$ given by \eqref{tildew0}. It is clear that $M$ is a supersolution of~\eqref{eq}-\eqref{patch}. Then by comparison principle (Proposition~\ref{proCP}), we have that $0\leq\widetilde{w}(t,x)\leq M$ for all $t\geq0$ and $x\in\R$. Furthermore, we have that $\widetilde{w}$ is increasing with respect to $t$ in $[0,+\infty)\times\R$, and from parabolic estimates in~\cite[Theorem 2.2]{HLZ-2022} that $\widetilde{w}(t,\cdot)$ converges as $t\to+\infty$, locally uniformly in $\R$, to a bounded positive periodic continuous solution $p^*$ of~\eqref{sta1} such that $\widetilde{w}(t,\cdot)|_{\bar{I}}\to p^*|_{\bar{I}}$ in $C^2(\bar{I})$ for each patch $I\subset\R$. Moreover, 
\begin{equation}\label{wp}
	0\leq\widetilde{w}_0(x)<p^*(x)\leq M~~~\text{for}~x\in\R.
\end{equation}
Thus, there exists a bounded positive and periodic solution $p^*$ of~\eqref{sta1}.

It is left to prove $\bar{\lambda}_1(l_1,l_2,p^*) \geq 0$. Otherwise, suppose that 
$\bar{\lambda}_1(l_1,l_2,p^*) <0$. Let $\phi$ be the unique positive eigenfunction of~\eqref{eigenv2} corresponding to $\bar{\lambda}_{1}(l_1,l_2,p^*)$. For any $\epsilon >0$, define 
$$
\bar{u}_{\epsilon}(x):=p^*(x)-\epsilon\phi(x),\quad x\in\R.
$$
Choose $\epsilon_0>0$ small enough such that $0<\bar{u}_{\epsilon}\le p^*$ in $\R$ for all $0<\epsilon\le\epsilon_0$ and
$$
f(x,p^*)-f(x,\bar{u}_{\epsilon})\geq \partial_sf(x,p^*)\times\epsilon\phi+\frac{\bar{\lambda}_{1}(l_1,l_2,p^*)}{2}\times \epsilon\phi\ \hbox{ in }\R\backslash S.$$
It then follows that for $x\in\R\backslash S$
\begin{equation}\label{supersol1}
	\begin{split}
		\partial_{t}\bar{u}_{\epsilon}-d(x)\partial_{xx}\bar{u}_{\epsilon}-f(x,\bar{u}_{\epsilon}) & =-d(x)\partial_{xx}\bar{u}_{\epsilon}-f(x,p^*)+f(x,p^*)-f(x,\bar{u}_{\epsilon})\vspace{3pt}\\
		&\geq -\bar{\lambda}_{1}(l_1,l_2,p^*)\times\epsilon\phi+\frac{\bar{\lambda}_{1}(l_1,l_2,p^*)}{2}\epsilon\phi>0
	\end{split}
\end{equation}
for all $0<\epsilon\le\epsilon_0$, and $\bar{u}_{\epsilon}$ satisfies the interface conditions in \eqref{sta1}. This implies that $\bar{u}_{\epsilon}$ is a strict supersolution of equation~\eqref{eq}-\eqref{patch}. Let $U_{\epsilon}$ be the solution to~\eqref{eq}-\eqref{patch} with initial datum $\bar{u}_{\epsilon}$ for any fixed $\epsilon\in(0,\epsilon_0]$. Then $U_{\epsilon}$ is decreasing with respect to $t$ in $[0,+\infty)\times\R$ for any fixed $\epsilon\in(0,\epsilon_0]$, and from parabolic estimates in~\cite[Theorem 2.2]{HLZ-2022} again that $U_{\epsilon}(t,\cdot)$ converges as $t\to+\infty$, locally uniformly in $\R$, to a bounded positive periodic continuous solution $p_{\epsilon}$ of~\eqref{sta1} such that $U_{\epsilon}(t,\cdot)|_{\bar{I}}\to p_{\epsilon}|_{\bar{I}}$ in $C^2(\bar{I})$ for each patch $I\subset\R$ as $t\to+\infty$. Since $\|\phi\|_{L^{\infty}(\mathbb{R})}=1$, up to decreasing $\epsilon$ if necessary, it follows from \eqref{wp} that
$$
0\leq \widetilde{w}_0(x)\leq \bar{u}_{\epsilon}(x), \quad\forall x\in \R.
$$ Then by using comparison principle (Proposition A.2) and~\eqref{wp}, we have that
$$
0\leq \widetilde{w}_0(x)<\widetilde{w}(t,x)\leq U_{\epsilon}(t, x)<\bar{u}_{\epsilon}(x)<p^*(x), \quad\forall t>0,\, x\in \R.
$$
Letting $t\to+\infty$ in the above inequality, we obtain for $0<\epsilon\leq\epsilon_0$
$$
0< p^*\leq p_{\epsilon}<p^*\quad\text{ in }\R,
$$
a contradiction. Thus we have that $\bar{\lambda}_1(l_1,l_2,p^*) \geq 0$. This completes the proof of Proposition~\ref{pro_pl1l2}.
\end{proof} 

\begin{remark}\label{rem_stability}
	The linear stability of $p^*$ obtained in Proposition~\ref{pro_pl1l2}
	does not require any restriction on $l_1$.
\end{remark}

Now we further show that the linearly stable bounded positive periodic solution $p^*$ of~\eqref{sta1}, obtained in Proposition~\ref{pro_pl1l2}, is strongly stable from below provided that $l_1$ is sufficiently small, $l_2$ is sufficiently large, and $\int_{0}^{K_2} f_2(s)\ds>0.$ For the definitions of strong stability from above and below, we refer the readers to~\cite[Definitions 2.1 and 4.2]{DHZ-2017}.

\begin{lemma}\label{lem_strong}
	Assume that $\int_0^{K_2} f_2(s) \mathrm{d} s>0$. Then, there exist $\hat{l}_1>0$ and $\hat{l}_2>0$ such that for every $0<l_1<\hat{l}_1$ and $l_2>\hat{l}_2$, the bounded positive periodic solution $p^*$ of~\eqref{sta1}, given in Proposition~\ref{pro_pl1l2}, is strongly stable from below, in the sense that there exists $\eta_0>0$ such that
\begin{equation}\label{strong stable}
	u(t,x;p^*(x)-\eta\phi(x))>p^*(x)-\eta\phi(x),\qquad \forall x\in\R,\, \forall\,\eta\in(0,\eta_0],
\end{equation}
	where $\phi$ is the principal eigenfunction of~\eqref{eigenv2} corresponding to the principal eigenvalue $\bar{\lambda}_1(l_1,l_2,p^*)$.
\end{lemma}

\begin{proof}
Let $\hat{l}_2>2R$, where $R$ is given by Lemma~\ref{lem-Phi2}. Note that $\hat{l}_2$ here may be larger than the one in Proposition~\ref{pro_pl1l2}. To prove \eqref{strong stable}, it suffices to show that the periodic function $p^*-\eta\phi$ is a strict subsolution of~\eqref{eq}-\eqref{patch} when $l_1>0$ is small enough and $l_2>\hat{l}_2$ is large enough. To this end, we first prove that $\bar{\lambda}_1(l_1,l_2,p^*)>0$. 

By Proposition~\ref{pro_pl1l2}, $\bar{\lambda}_1(l_1,l_2,p^*)\geq0$. We argue by contradiction and assume that there are sequences $l_1^m$ and $l_2^m$ in $(0,+\infty)$,  such that $l_1^m\to0$, $l_2^m\to+\infty$ as $m\to+\infty$, and $\bar{\lambda}_1(l_1^m,l_2^m,p_{m}^*)=0$ for each $m\in\NN$, where $p_m^*$ satisfies
	\begin{equation}\label{sta4}
		\begin{cases}
			-d(x)(p_m^*)''(x)=f(x, p_m^*(x)), & x\in\R\backslash S_m,\\
			p_m^*(x^-)=p_m^*(x^+),\quad (p_m^*)'(x^-)=\sigma (p_m^*)'(x^+), & x=nl^m,\\
			p_m^*(x^-)=p_m^*(x^+),\quad \sigma (p_m^*)'(x^-)=(p_m^*)'(x^+), & x=nl^m+l_2^m,
		\end{cases}
	\end{equation}
	and $0<p_m^*\leq M=\max(K_1,K_2)$ from~\eqref{wp},  where
	\begin{equation}\label{Sm}
		S_m:=\{nl^m: n\in\mathbb Z\}\cup\{nl^m+l_2^m: n\in\mathbb Z\},\quad l^m=l_1^m+l_2^m.
	\end{equation}
	In addition, $\phi_m:\mathbb{R}\to\mathbb{R}$ satisfies $\phi_m|_{\overline{I}_m}\in C^\infty(\overline{I}_m)$ for each patch $I_m\subset\mathbb{R}\setminus S_m$, and
	\begin{equation}\label{eigenv4}
		\begin{cases}
			-d(x)\phi_m^{\prime\prime}(x) -f_s(x, p_m^*)\phi_m=\bar{\lambda}_1(l_1^m,l_2^m,p_m^*)\phi_m,  & x \in \mathbb{R} \backslash S_m, \\
			\phi_m(x^{-}) =\phi_m(x^{+}),\quad \phi_m^{\prime}(x^{-})=\sigma \phi_m^{\prime}(x^{+}), &  x=nl^m, \\
			\phi_m(x^{-}) =\phi_m(x^{+}),\,\, \sigma \phi_m^{\prime}(x^{-})=\phi_m^{\prime}(x^{+}), &  x=nl^m+l_2^m,\\
			\phi_m \text{ is periodic}, \phi_m>0, \|\phi_m\|_{L^{\infty}(\R)}=1,
		\end{cases}
	\end{equation}
where
\begin{equation}\label{Im}
	I_m=(nl^m-l_1^m,nl^m)\quad\text{or}\quad I_m=(nl^m,nl^m+l_2^m), \qquad n\in\mathbb Z.
\end{equation}

To derive a contradiction, we proceed in two steps.

\medspace
	
	\noindent{\it Step 1:
	We first claim that}
	\begin{equation}\label{claim_pK2}
	p_m^*\to K_2\text{ as } m\to+\infty\text{ uniformly in } \R .
	\end{equation}
	Indeed, suppose by contradiction that the convergence is not uniform. Then there exist $\varepsilon_0>0$, a subsequence (still labeled by $m$) $(x_m)_{m\in\NN}\subseteq\mathbb R$ such that
	$|p_m^*(x_m)-K_2|\ge \varepsilon_0$ for all $m\in\mathbb N.$ Up to extracting a further subsequence, we may assume that either 
	\begin{equation*}\label{pmb}
		p_m^*(x_m) \le K_2-\varepsilon_0 \quad \text{or} \quad K_2+\varepsilon_0\leq p_m^*(x_m) \le  K_1, \quad \forall m\in\NN.
	\end{equation*}
	Define 
	$$
	v_{m}(x):=p_m^*(x+x_m),\quad \forall x\in\R.
	$$
	Then $(v_m)_{m\in\NN}$ satisfies
	\begin{equation}\label{vm}
		\begin{cases}
			-d(x+x_m) v_m''(x)= f(x+x_m,v_m), & x\in\R\backslash (S_m-x_m),\\
			v_m(x^-)=v_m(x^+),\quad v_m'(x^-)=\sigma v_m'(x^+), & x=nl^m-x_m, n\in\Z\\
			v_m(x^-)=v_m(x^+),\quad \sigma v_m'(x^-)=v_m'(x^+), & x=nl^m+l_2^m-x_m, n\in\Z,
		\end{cases}
	\end{equation}
	and 
	$$
	v_m(0) \le K_2-\varepsilon_0 \quad \text{or} \quad K_2+\varepsilon_0 \leq v_m(0) \leq K_1, \quad \forall m\in\NN.
	$$
	We divide the argument into two cases according to the behavior of $\dist(x_m, S_m)$.
	
	\medspace
	
	\noindent
	{\it Case 1.1: $\dist(x_m,S_m)\to+\infty$ as $m\to+\infty$, up to extraction of a subsequence.}
	Fix $R_1>0$. By assumption, there exists $m_{R_1}\in\NN$ such that $\dist(x_m,S_m)>R_1$ for all $m\ge m_{R_1}.$
	This implies that $(x_m-R_1,x_m+R_1)\cap S_m=\varnothing,$ or equivalently, $(-R_1,R_1)\cap(S_m-x_m)=\varnothing.$
	Since $l_1^m\to0$ and $l_2^m\to+\infty$ as $m\to+\infty$, the interval $(-R_1,R_1)$ becomes entirely contained within a bistable patch for sufficiently large $m$. Thus, $v_m$ for all $m$ large satisfies
	$$
	-d_2 v_m''=f_2(v_m) \text{ in }(-R_1,R_1).
	$$
	By standard elliptic estimates, up to a subsequence, $v_m\to v_\infty$ as $m\to+\infty$ in $C^2_{\rm loc}(\R)$, where $v_{\infty}$ satisfies 
	\begin{equation*}
		-d_2v_\infty''(x)=f_2(v_\infty(x)),\quad x\in\mathbb R,
	\end{equation*}
	and $v_\infty(0)\leq K_2-\varepsilon_0$ or $K_2+\varepsilon_0\leq v_\infty(0)\leq K_1.$ We claim that $0\leq v_\infty(x)\leq K_2$ for all $x\in\R$. Indeed, $0\leq v_\infty(x)\leq M$ for all $x\in\R$. The function $w(t,x):=v_\infty(x)$ is a stationary solution of
	\begin{equation}\label{w-infty}
		\partial_tw=d_2\partial_{xx}w+f_2(w),\quad(t,x)\in(0,\infty)\times\mathbb R.
	\end{equation}
	Let $\bar v_\infty$ be the solution of
	\begin{equation*}
	\begin{cases}
		\bar{v}_{\infty}^{\prime}(t)=f_2(\bar{v}_{\infty}(t)),\\
		\bar{v}_{\infty}(0)=M.
	\end{cases}
\end{equation*}
	Since $M\geq K_2$ and $f_2$ satisfies~\eqref{f2}, we have 
	\begin{equation}\label{vinf}
	 \bar v_\infty(t)\to K_2 \text{ as } t\to+\infty.
	\end{equation}
	Moreover, $(t,x)\mapsto\bar v_\infty(t)$ is a spatially homogeneous solution of~\eqref{w-infty}. Since $0\leq w(0,x)=v_\infty(x)\leq M=\bar v_\infty(0)$ for $x\in\R,$
	the comparison principle yields
	$$
	0\leq w(t,x)=v_\infty(x)\leq\bar v_\infty(t),
	\qquad (t,x)\in[0,\infty)\times\mathbb R.
	$$
	Letting $t\to+\infty$ and using~\eqref{vinf}, we obtain $0\leq v_\infty(x)\leq K_2$ for all $x\in\R.$ Thus, the assumption $K_2+\varepsilon_0\leq v_\infty(0)\leq K_1$ is impossible. Hence,
$
v_\infty(0)\leq K_2-\varepsilon_0.
$
In conclusion, $v_{\infty}$ satisfies 
\begin{equation}\label{vinfty}
	\begin{cases}
		-d_2 v_\infty''(x)=f_2(v_\infty(x))\quad \text{ for }x\in\R, \\
		0\leq v_{\infty}\leq K_2.
	\end{cases}
\end{equation}

Furthermore, it follows from~\eqref{wp} that there exists $\varepsilon>0$ such that $K_2-\varepsilon \le p_m^*\leq M$ holds in $\R$ for all $m\in\NN$. Passing to the limit as $m \to +\infty$, we obtain $K_2-\varepsilon\leq v_{\infty} \le K_2$ in $\R$. This together with~\eqref{vinfty} and $f_2(K_2)=0$ yields $v_{\infty} \equiv K_2$ in $\R$. 
	This contradicts the assumption $v_\infty(0) \le K_2 - \varepsilon_0$ for $\varepsilon_0>0$. Hence, Case~1.1 is ruled out.

\medspace
	
\noindent	{\it Case~1.2: $\limsup_{m\to+\infty}\dist(x_m,S_m)<+\infty$.} Under this assumption, there exist $C>0$ and a sequence
	$y_m\in S_m$ such that $|x_m-y_m|\le C.$ Define $\eta_m:=x_m-y_m.$ Without loss of generality, we can choose $y_m\in \{nl^m: n\in\mathbb Z\}$.
	Up to extraction of a subsequence, $\eta_m\to\eta_\infty\in\R$ as $m\to+\infty$. By elliptic estimates, up to extraction of a subsequence, we have that $v_m\to v_\infty$ in $C^2_{\rm loc}(\R\setminus\{-\eta_\infty\})$ as $m\to+\infty$, where $v_\infty$ satisfies
	\begin{equation}\label{vm2}
		\begin{cases}
			-d_2 v_{\infty}''(x)= f_2(v_{\infty}(x)), & x\in(-\infty,-\eta_{\infty})\cup(-\eta_{\infty},+\infty),\\
			0\leq v_{\infty}\leq K_2.
		\end{cases}
	\end{equation}
	Furthermore, from the interface conditions in~\eqref{sta4}, we have that
	$$
	(p_m^*)'(y_m^-)=\sigma (p_m^*)'(y_m^+),\quad \sigma (p_m^*)'((y_m-l_1^m)^-)=(p_m^*)'((y_m-l_1^m)^+).
	$$
	Since $l_1^m\to0$ as $m\to+\infty$ and $\| (p_m^*)'' \|_{L^{\infty}(nl^m-l_1^m,nl^m)}\le\frac{f_1'(0)K_1}{d_1}$, we have that
	$$
	\left|(p_m^*)'((y_m-l_1^m)^-)-(p_m^*)'(y_m^+)\right|\le \frac{f_1'(0)K_1}{d_1} l_1^m.
	$$
	This implies that
	$$
	\left|(p_m^*)'((y_m-l_1^m)^-)-(p_m^*)'(y_m^+)\right|\to0\text{ as }m\to+\infty.
	$$
	Recalling $v_m(x)=p_m^*(x+x_m)$, and passing to the limit in the interface conditions of~\eqref{vm} therefore yields
	$$
	v_\infty'((-\eta_\infty)^-)=v_\infty'((-\eta_\infty)^+).
	$$
	Thus, $v_\infty\in C^1(\mathbb R)$. Combining this with \eqref{vm2}, we have that 
	$v_\infty\in C^2(\mathbb R)$ and satisfies
	\begin{equation*}
		\begin{cases}
			-d_2 v_{\infty}''(x)= f_2(v_{\infty}(x)), & x\in\R,\\
			 0\leq v_{\infty}\leq K_2.
		\end{cases}
	\end{equation*}
	Similar to the proof of Case~1.1, we have that $v_{\infty}\equiv K_2$ in $\R$. This contradicts the assumption that either $v_\infty(0) \le K_2 - \varepsilon_0$ or $v_\infty(0) \ge K_2 + \varepsilon_0$. Hence, Case~1.2 is ruled out. This proves~\eqref{claim_pK2}.
	
	\medspace
	
	 \noindent{\it Step 2: We show that}
	 $$
	 \bar{\lambda}_1(l_1,l_2,p^*)>0.
	 $$
	 Since $\phi_m$ is periodic and $\|\phi_m\|_{L^\infty(\mathbb R)}=1$, we may choose $z_m\in(-l_1^m,l_2^m)$ such that $\phi_m(z_m)=1.$ For any $m\in\mathbb{N}$, we distinguish two cases according to the location of $z_m$.
	
	\medspace
	
\noindent{\it Case~2.1: $\limsup_{m\to+\infty} \dist(z_m, S_m)>0$.} Using the assumption
	$\bar{\lambda}_1(l_1^m,l_2^m,p_m^*)=0$ for each $m\in\mathbb N$,~\eqref{claim_pK2}, and a similar argument to that in Step~1, there exists a nonnegative continuous function $\phi_{\infty}$ such that, up to extraction of a subsequence, $\phi_m(\cdot+z_m)\to\phi_{\infty}$ in $C^2_{\rm loc}(\R)$ as $m\to+\infty$. Furthermore, $\phi_{\infty}$ solves
	\begin{equation}\label{phiinfty1}
		\begin{cases}
			-d_2\phi_{\infty}^{\prime\prime} -f_2'( K_2)\phi_{\infty}=0 & \text{in }\R, \\
			\phi_{\infty}(0)=1, \quad 0<\phi_{\infty}\leq1 & \text{in }\R.
		\end{cases}
	\end{equation}
	Any solution to \eqref{phiinfty1} is of the form
$$
\phi_\infty(x)=C_1e^{\mu x}+(1-C_1)e^{-\mu x},\qquad\text{ where }\mu=\sqrt{-\frac{f_2'(K_2)}{d_2}}>0,
$$
where $C_1\in\R$. Since $\mu>0$, every solution is unbounded either as $x\to+\infty$ or as $x\to-\infty$.
This contradicts the fact that $0<\phi_\infty\le1$ in $\mathbb R.$ Moreover, $\phi_\infty\equiv1$ cannot be a solution to~\eqref{phiinfty1}, since $f_2'(K_2)<0$. Therefore Case~2.1 cannot occur.
	
		\medspace
		
\noindent	{\it Case~2.2: $\lim_{m\to+\infty} \dist(z_m, S_m)=0$.} Using~\eqref{claim_pK2}, $\bar{\lambda}_1(l_1^m,l_2^m,p_m^*)=0$ for each $m\in\mathbb{N}$, and a similar argument to that in Case~1.2 of Step~1, there exists a nonnegative continuous function $\phi_{\infty}$ such that, up to extraction of a subsequence, $\phi_m(\cdot+z_m)\to\phi_{\infty}$ in $C^2_{\rm loc}(\R)$ as $m\to+\infty$. Furthermore, $\phi_{\infty}$ solves \eqref{phiinfty1}. The same argument as in Case~2.1 shows that Case~2.2 cannot occur. Hence $\bar{\lambda}_{1}(l_1,l_2,p^*)>0$. 
		
	We finally show that there exists $\eta_0>0$ such that $p^*-\eta\phi$
	is a strict subsolution of~\eqref{eq}-\eqref{patch} for every $\eta\in(0,\eta_0]$. Since $\|\phi\|_{L^{\infty}(\R)} = 1$ and $p^*>0$ in $\R$, we may choose $\eta_0>0$ sufficiently small such that $p^*-\eta\phi>0$ in $\R$ and
	$$
	f(x,p^*)
	-f(x,p^*-\eta\phi)
	\le
	\eta\,\partial_s f(x,p^*)\phi
	+\frac{\bar\lambda_1(l_1,l_2,p^*)}{2}\eta\phi
	$$
	for all $0<\eta\le\eta_0$ and $x\in\mathbb R\setminus S$. It then follows that for $x\in\R\backslash S$
	\begin{equation*}\label{subsol2}
		\begin{split}
		 -d(x)(p^*-\eta\phi)''-f(x,p^*-\eta\phi)&=d(x)\eta\phi''+f(x,p^*)-f(x,p^*-\eta\phi)\vspace{3pt}\\
			&\leq -\bar{\lambda}_{1}(l_1,l_2,p^*)\eta\phi+\frac{\bar{\lambda}_{1}(l_1,l_2,p^*)}{2}\eta\phi<0
		\end{split}
	\end{equation*}
	for all $0<\eta\leq\eta_0$, as well as the interface conditions in \eqref{sta1}. This implies that $p^*-\eta\phi$ is a strict subsolution of equation~\eqref{eq}-\eqref{patch}. As a consequence, the solution $u(t,x;p^*(x)-\eta\phi(x))$ of \eqref{eq}-\eqref{patch} with initial condition $p^*(x)-\eta\phi(x)$ is increasing in $t$. Thus we have that \eqref{strong stable} holds. We complete the proof of Lemma~\ref{lem_strong}.
\end{proof}

\medspace

Recall that if $0<l_1<L_1^c$ and $l_2>l_2^c$, then Proposition~\ref{lem-STA}(ii) shows that $0$ is a stable solution to problem~\eqref{sta1}. On the other hand, Proposition~\ref{pro_pl1l2} establishes the existence of a linearly stable bounded positive periodic solution $p^*$ of ~\eqref{sta1} when $\int_0^{K_2} f_2(s)\mathrm{d}s>0$ and $l_2$ is large enough. It then follows from the Dancer-Hess connecting orbit theorem (see, e.g., \cite[Proposition 9.1]{H-1991}) that  there exists at least one periodic steady state $\bar{p}$  satisfying $0 < \bar{p} < p^*$  in $\R$. Furthermore, under the assumption $\int_0^{K_2} f_2(s)\mathrm{d}s>0$, we show in the following lemma that all such intermediate periodic steady states are unstable whenever $l_1$  is sufficiently small and $l_2$ is sufficiently large.

\begin{lemma}\label{unstable}
	Assume that $\int_0^{K_2} f_2(s) \mathrm{d} s>0$, there exist $\hat{l}_1>0$ and $\hat{l}_2>0$ such that $\bar{\lambda}_1(l_1,l_2,\bar{p})<0$ for every $0<l_1<\hat{l}_1$ and $l_2>\hat{l}_2$, and for every periodic steady state $\bar{p}$ of \eqref{eq}-\eqref{patch} with $0<\bar{p}<p^*$.
\end{lemma}

\begin{proof}
	Assume by contradiction that there exist sequences $(l_1^m)_{m\in\mathbb N}$ and $(l_2^m)_{m\in\NN}$ in $(0,+\infty)$ such that $l_1^m\to0$, $l_2^m\to+\infty$ as $m\to+\infty$, together with sequences $(\bar p_m)_{m\in\mathbb N}$ and $(\phi_m)_{m\in\mathbb N}$ satisfying the following properties. For each $m\in\mathbb N$, the function $\bar p_m$ solves
	\begin{equation}\label{sta3}
		\begin{cases}
			-d(x)\bar{p}_m''(x)=f(x, \bar{p}_m(x)), & x\in\R\backslash S_m,\\
			\bar{p}_m(x^-)=\bar{p}_m(x^+),\quad \bar{p}_m'(x^-)=\sigma \bar{p}_m'(x^+), & x=nl^m, n\in\Z,\\
			\bar{p}_m(x^-)=\bar{p}_m(x^+),\quad \sigma \bar{p}_m'(x^-)=\bar{p}_m'(x^+), & x=nl^m+l_2^m, n\in\Z,
		\end{cases}
	\end{equation}
	with $0<\bar{p}_m<p_m^*$ in $\R$, where $S_m$ is defined by \eqref{Sm} and $p_{m}^*$ is given by~\eqref{sta4}. Moreover, $\phi_m:\mathbb R\to\mathbb R$ satisfies
	$\phi_m|_{\overline I_m}\in C^\infty(\overline I_m)$ for every patch $I_m$ as in~\eqref{Im}, and solves the eigenvalue problem~\eqref{eigenv4} with $p^*_m$ replaced by $\bar p_m$, corresponding to the principal eigenvalue $\bar{\lambda}_1(l_1^m,l_2^m,\bar{p}_m)\geq0$. Furthermore, by using \eqref{prinbound1}, we actually have that
	\begin{equation*}\label{prinbound3}
		-\max_{x\in\mathbb R,\ s\in[0,p_m^*]}
		f_s(x,s)
		\le
		\bar{\lambda}_1(l_1^m,l_2^m,\bar p_m)
		\le
		-\min_{x\in\mathbb R,\ s\in[0,p_m^*]}
		f_s(x,s).
	\end{equation*}
	The sequence $\big(\bar{\lambda}_1(l_1^m,l_2^m,\bar{p}_m) \big)_{m\in\NN}$ is then bounded. Up to extraction of some subsequence, there is a real number $\tilde{\lambda}_1\geq 0$ such that $\bar{\lambda}_1(l_1^m,l_2^m,\bar{p}_m) \to\tilde{\lambda}_1$ as $m\to+\infty$.
	
	We first claim that
	\begin{equation}\label{claim6.1}
			\text{for every }m\in\mathbb N,\ \text{there exists }
		x_m\in(-l_1^m,l_2^m)
		\text{ such that }
		\bar p_m(x_m)=\theta.
	\end{equation}
	We postpone the proof of \eqref{claim6.1} until the end of this argument. We continue with the proof of Lemma~\ref{unstable}.

Define 
	$$
	w_m(x):=\bar{p}_m(x+x_m),\quad x\in\R.
	$$
	Each function $w_m:\R\to\R$ is continuous and $w_m|_{\bar{I}_m}\in C^2(\bar{I}_m)$ for each patch $I_m$ as in~\eqref{Im} (shifted by $-x_m$) is a solution to
	\begin{equation*}\label{wm}
		\begin{cases}
			-d(x+x_m) w_m''(x)= f(x+x_m,w_m), & x\in\R\backslash (S_m-x_m),\\
			w_m(x^-)=w_m(x^+),\quad w_m'(x^-)=\sigma w_m'(x^+), & x=nl^m-x_m,\\
			w_m(x^-)=w_m(x^+),\quad \sigma w_m'(x^-)=w_m'(x^+), & x=nl^m+l_2^m-x_m,\\
			w_m(0)=\theta.
		\end{cases}
	\end{equation*}
	Similar to the proof of~Lemma~\ref{lem_strong}, we now distinguish two cases according to the behavior of $\dist(x_m,S_m).$
	
	\medspace
	
\noindent{\it Case~1: $\dist(x_m,S_m)\to+\infty$ as $m\to+\infty$, up to extraction of a subsequence.} By the same arguments as in Case~1.1 of Lemma~\ref{lem_strong}, after passing to a subsequence if necessary, we obtain $w_m\to w_\infty$ in $C^2_{\rm loc}(\R)$, where $w_{\infty}$ satisfies 
	\begin{equation}\label{w1}
		\begin{cases}
			-d_2 w_\infty''(x)=f_2(w_\infty(x))\quad &x\in\R, \\
			w_\infty(0)=\theta, \; 0<w_\infty< K_2.
		\end{cases}
	\end{equation}
	The bounds $0<w_\infty<K_2$ follow from the strong maximum principle. Similarly, by normalizing $\phi_m$ in such a way that $\phi_m(x_m)=1$, there exists a nonnegative continuous function $\phi_{\infty}$ such that, up to extraction of some subsequence, $\phi_m(\cdot+x_m)\to\phi_{\infty}$ in $C_{\rm loc}^2(\R)$ as $m\to+\infty$, and $\phi_{\infty}$ solves
	\begin{equation}\label{phiinfty}
		\begin{cases}
			-d_2\phi_{\infty}^{\prime\prime} -f_2'( w_{\infty})\phi_{\infty}=\tilde{\lambda}_1\phi_{\infty}  & \text{ in }\R, \\
			\phi_{\infty}(0)=1,\, 0<\phi_{\infty}\leq1 &\text{ in }\R
		\end{cases}
	\end{equation}
	(notice here that the function $\phi_{\infty}$ may not be periodic, since the convergence is only local as $m\to+\infty$.)
	
	By~\eqref{f2} and $\int_{0}^{K_2}f_2(s)\ds>0$, according to phase diagrams of equation~\eqref{w1}, the solution $w_{\infty}$ can only be of one of the following three types: either a constant function, or a non-constant periodic function, or a ground state solution such that $w_{\infty}(\pm\infty)=0$. In what follows, one will get a contradiction in each of these three cases.
	
	\medspace
	
	\noindent{\it Case~1.1: $w_{\infty}$ is a constant solution, that is $w_{\infty}\equiv \theta$ in $\R$}. In this case, $\phi_{\infty}$ obeys the linear equation 
	$$
	\phi_{\infty}''=-\dfrac{f_2'(\theta)+\tilde{\lambda}_1}{d_2}\phi_{\infty}\quad \text{ in }\R.
	$$
	Since $f_2'(\theta)>0$ and~$\tilde{\lambda}_1\geq 0$, it follows that the positive function $\phi_{\infty}$ is strictly concave in $\R$, which is impossible. Hence, Case~1.1 is ruled out.
	
	\medspace
	
\noindent	{\it Case~1.2: $w_{\infty}$ is a non-constant periodic solution}. In this case, $w_{\infty}'$ is a non-signed periodic function satisfying
	\begin{equation}\label{eqpinfty'}
		d_2(w_{\infty}')''+ f_2'(w_{\infty})w_{\infty}'=0\ \hbox{ in }\R,
	\end{equation}
	whereas $\phi_{\infty}$ solves $d_2\phi_{\infty}^{\prime\prime}(x) +(f_2'( w_{\infty})+\tilde{\lambda}_1)\phi_{\infty}(x)=0$ in $\R$. Since $\tilde{\lambda}_1\geq 0$, it follows from Sturm comparison theorem that $\phi_{\infty}$ must vanish somewhere, which is impossible since $\phi_{\infty}>0$ in $\R$. Hence, Case~1.2 is ruled out too.
	
	\medspace
	
\noindent	{\it Case~1.3: $w_{\infty}$ is a non-periodic solution and $\lim_{x\to\pm\infty}w_{\infty}(x)=0$}. Denote $F(s)=\int_0^sf_2(u)du$ for all $s\in[0,K_2]$. From the assumptions~\eqref{f2} and $\int_0^{K_2}f_2(u)\,du>0$, there is a real number $\bar{s}\in (\theta,K_2)$ such that $F(0)=F(\bar{s})=0$, $F(s)<0$ for all $0<s<\bar{s}$ and $F(s)>0$ for all $\bar{s}<s\le K_2$.  It then follows that there is $\bar{x}\in\R$ such that $w_{\infty}(\bar{x})=\bar{s}$, $w_{\infty}'(\bar{x})=0$, $w_{\infty}'>0$ in~$(-\infty,\bar{x})$ and~$w_{\infty}'<0$ in $[\bar{x},+\infty)$. Notice also by~\eqref{f2} that
	$$
	w_{\infty}''(\bar{x})=-\frac{f_2(w_{\infty}(\bar{x}))}{d_2}=-\frac{f_2(\bar{s})}{d_2}<0,
	$$
	and that there is $\underline{x}<\bar{x}$ such that $w_{\infty}''(x)=-f_2(w_{\infty}(x))/d_2>0$ for all $x\le\underline{x}$. Furthermore, $\lim_{x\to-\infty}w_{\infty}''(x)=\lim_{x\to-\infty}w_{\infty}'(x)=0$. Denote
	$$
	\bar{q}(x)=\phi_{\infty}'(x)w_{\infty}'(x)-\phi_{\infty}(x)w_{\infty}''(x)\ \hbox{ for }x\in\R.
	$$
	It follows from~\eqref{phiinfty} and~\eqref{eqpinfty'} that $\bar{q}'(x)=-\tilde{\lambda}_1\phi_{\infty}(x)w_{\infty}'(x)/d_2\le0$ for all $x\le\bar{x}$, whence
	\begin{equation}\label{qx}
		\bar{q}(x)\ge \bar{q}(\bar{x})=-\phi_{\infty}(\bar{x})w_{\infty}''(\bar{x})>0\ \hbox{ for all }x\le\bar{x}.
	\end{equation}
	Therefore, $\phi_{\infty}'(x)w_{\infty}'(x)\ge\phi_{\infty}(x)w_{\infty}''(x)>0$ for all $x\le\underline{x}\,(<\bar{x})$. In particular, $\phi_{\infty}'(x)>0$ for all $x\le\underline{x}$ and, since $\phi_{\infty}$ is positive, the limit $\phi_{\infty}(-\infty)\in[0,+\infty)$ exists. By~\eqref{phiinfty}, the function $\phi_{\infty}''$ has a finite limit as $x\to-\infty$ and it follows then from elementary arguments that $\phi_{\infty}'(x)\to0$ as $x\to-\infty$. Lastly, since $w'_{\infty}(-\infty)=w_{\infty}''(-\infty)=0$, one gets that $\bar{q}(x)\to0$ as $x\to-\infty$, which contradicts~\eqref{qx}. As a consequence, Case~1.3 is ruled out too. 

	\medspace
	
	\noindent
	{\it Case 2: $\limsup_{m\to+\infty}\dist(x_m,S_m)<+\infty$.} By the same argument as in Case~1.2 of Lemma~\ref{lem_strong}, after passing to a subsequence if necessary, we obtain $w_m\to w_\infty$ in $C^2_{\rm loc}(\R)$ as $m\to+\infty$, where $w_{\infty}$ satisfies 
	\begin{equation*}
		\begin{cases}
			-d_2 w_{\infty}''(x)= f_2(w_{\infty}(x)), & x\in\R,\\
			w_{\infty}(0)=\theta,\quad 0<w_{\infty}<K_2.
		\end{cases}
	\end{equation*}
	Using the similar arguments as in Cases~2.1 and~2.2 of Lemma~\ref{lem_strong}, there exists a nonnegative continuous function $\phi_{\infty}$ such that, up to extraction of some subsequence, $\phi_m(\cdot+x_m)\to\phi_{\infty}$ as $m\to+\infty$ in $C_{\rm loc}^2(\R)$, and $\phi_{\infty}$ satisfies~\eqref{phiinfty}. Therefore, repeating the arguments in Case~1 above yields a contradiction. The conclusion of Lemma~\ref{unstable} is therefore achieved, assuming~\eqref{claim6.1}.
	
	\medspace
	
	We now turn to the proof of~\eqref{claim6.1}. Suppose, by contradiction, that $\bar p_m(x)\neq\theta$ for all $x\in\mathbb R$. By continuity and $l^m$-periodicity of $\bar{p}_m$, one would have either
$$0<\bar{p}_m(x)<\theta\ \hbox{ for all $x\in\R$},$$
or $\theta<\bar{p}_m(x)<p_m^*(x)$ for all $x\in\R$. We only consider the first case, as the second one can be handled in a similar manner. For the first case, we proceed in two steps to derive a contradiction.
	
	\medspace
	
	\noindent {\it Step 1: For each $m\in\mathbb{N}$, if $\dis\max_{\R}\bar{p}_m=\bar{p}_m(x_0)$, then $x_0\in(nl^m-l_1^m,nl^m)$ for some $n\in\Z$.}
	Indeed, on the one hand, from~\eqref{f2}, \eqref{sta3}, and the assumption $0<\bar{p}_m<\theta$ in $\R$, we have that 
	\begin{equation}\label{pm1}
		(\bar{p}_m)^{\prime\prime}(x)=- \dfrac{f_2(\bar{p}_m(x))}{d_2}>0,\quad  x \in (nl^m,nl^m+l_2^m), n\in\Z.
	\end{equation}
	This implies that $\bar{p}_m$ is a convex function in each bistable patch. On the other hand, if $K_1 \ge \theta$ or $0< \bar{p}_m(x)\leq K_1 < \theta$ for $x\in\R$, it follows from~\eqref{f1} and~\eqref{sta3} that
	$$
	(\bar{p}_m)''(x)= - \frac{f_1(\bar{p}_m(x))}{d_1}<0,\quad x \in (nl^m-l_1^m,nl^m),\ n\in\Z.
	$$
	Hence, $\bar p_m$ is concave on each KPP patch when $K_1\geq\theta$ or $0< \bar{p}_m(x)\leq K_1 < \theta$ for $x\in\R$. Furthermore, we show that $\bar p_m$ cannot be strictly monotone on $(nl^m,nl^m+l_2^m)$, $n\in\Z$. Suppose, for instance, that $\bar p_m$ is increasing on $(nl^m,nl^m+l_2^m)$. Then $\bar p_m'((nl^m+l_2^m)^-)>0$, and by the interface condition \eqref{sta3}, $\bar p_m'((nl^m+l_2^m)^+)>0$. If $\bar p_m$ remains increasing on the subsequent KPP patch, this contradicts the periodicity of $\bar p_m$ and $0<\bar{p}_m<\theta$. On the other hand, if $\bar p_m'$ becomes non-positive somewhere in that KPP patch, then, since $\bar p_m$ is concave there, $\bar p_m'$ cannot become positive again, again contradicting the periodicity of $\bar p_m$.
	The case where $\bar p_m$ is decreasing on $(nl^m,nl^m+l_2^m)$ can be treated similarly. Therefore $\bar p_m$ is not strictly monotone on any bistable patch. Since $\bar p_m$ is convex on each bistable patch, it follows that $\bar p_m$ attains its minimum at each bistable patch.	
	
	Next, let $x_0$ be a point where $\bar p_m$ attains its global maximum. If $x_0\in S_m$, then $(\bar p_m)'(x_0)=0$ by the interface condition. Since $\bar{p}_m(x)$ is strictly concave for $x\in(nl^m-l_1^m,nl^m)$, the derivative is decreasing there. Therefore, we have that $(\bar{p}_m)^{\prime}((nl^m)^{-})<0$, and $(\bar{p}_m)^{\prime}((nl^m-l_1^m)^{+})>0$. Similarly, since $\bar{p}_m(x)$ is strictly convex for $x\in(nl^m,nl^m+l_2^m)$, the derivative is strictly increasing there. Hence, $(\bar{p}_m)^{\prime}((nl^m)^{+})<0$ and $(\bar{p}_m)^{\prime}((nl^m+l_2^m)^{-})>0$. Therefore  $\bar{p}_m$ is increasing on both sides of $nl^m$ and decreasing on both sides of $nl^m-l_1^m$, contradicting the fact that $x_0\in S_m$ is a maximum point. Therefore, $\bar p_m$ cannot attain its maximum on $S_m$. For each $m\in\NN$, if $x_0\in(nl^m,nl^m+l_2^m)$ for some $n\in\Z$, a similar argument leads to a contradiction. This ensures that $\max_{\mathbb{R}}\bar{p}_m$ is attained in an open KPP patch when $K_1\geq\theta$ or $0< \bar{p}_m(x)\leq K_1 < \theta$ for $x\in\R$.
		
	For $K_1<\theta$, if $K_1 < \bar{p}_m < \theta$ in $\mathbb{R}$, then $\bar{p}_m$ is strictly convex in $\mathbb{R}$, contradicting its periodicity and boundedness. Thus, there exists $\bar{x} \in \mathbb{R}$ such that $0 < \bar{p}_m(\bar{x}) \le K_1<\theta$. Furthermore, we claim that $0 < \bar{p}_m \le K_1<\theta$ in $\mathbb{R}$. Indeed, by contradiction, we assume that $K_1<\max_{\mathbb{R}} \bar{p}_m <\theta$. Since $x_0$ is the global maximum point of $\bar{p}_m$, if $x_0 \in (nl^m, nl^m + l_2^m)$ for every $n \in \mathbb{Z}$, it follows from~\eqref{pm1} that $(\bar{p}_m)''(x_0) > 0$, which is a contradiction. If $x_0 \in (nl^m-l_1^m, nl^m)$ for some $n \in \mathbb{Z}$, it follows from~\eqref{f1} and~\eqref{sta3} that 
	$$
	(\bar{p}_m)''(x_0)= - \frac{f_1(\bar{p}_m(x_0))}{d_1}>0,
	$$
	which is a contradiction. If $x_0 \in S_m$, we obtain a contradiction by using the similar arguments as above. Thus, we have $\max_{\mathbb{R}} \bar{p}_m \le K_1$, which implies $0 < \bar{p}_m(x) \le K_1$ for all $x \in \mathbb{R}$. Then by using the same argument as above, we have that any global maximum point $x_0$ cannot lie in the bistable patches or on the interfaces $S_m$.\\
	
	\noindent {\it Step 2: Contradiction.} For $m\in\NN$, without loss of generality, by using Step 1 and periodicity of $\bar{p}_m$, we may assume that there exists $x_0^m\in(-l_1^m,0)$ such that $\bar p_m'(x_0^m)=0.$ Furthermore, for $x\in(-l_1^m,0)$, it follows from~\eqref{sta3} and \eqref{f1} that 
	$$
	\| \bar p_m'' \|_{L^{\infty}(-l_1^m,0)}\le\frac{f_1'(0)K_1}{d_1}.
	$$ 
	Hence, we have that
	$$
	|\bar p_m'(0^-)-\bar p_m'(x_0^m)|\le\frac{f_1'(0)K_1}{d_1}l_1^m.
	$$
	Since $l_1^m \to 0$ as $m \to +\infty$, it follows immediately that $\bar{p}_m'(0^-) \to 0$ as $m\to+\infty$. By applying the interface conditions at $x=0$, we also have $\bar{p}_m'(0^+) \to 0$ as $m \to +\infty$. Then, up to a subsequence, we may assume that one of the following three alternatives holds:
	\begin{equation*}
		\lim_{m\to+\infty}	\bar p_m(0)=
		\begin{cases}
			0,\\
			\theta,\\
			\gamma\in(0,\theta).
		\end{cases}
	\end{equation*}
	We now show that each of these cases leads to a contradiction.
	\begin{itemize}
		\item If $\lim_{m\to+\infty}\bar{p}_m(0) = 0$, recalling that $\lim_{m\to+\infty}\bar{p}_m'(0^\pm) = 0$ and that $0$ is a stable solution of~\eqref{sta3}, the continuous dependence of solutions to the ODE in the KPP patch implies that $\bar p_m$ converges locally uniformly to $0$ in $\R$ as $m\to+\infty$. The convergence is in fact uniform on $\mathbb R$. Indeed, since $x_0^m\in(-l_1^m,0)$ is the global maximum point of $\bar p_m$. By Step~1, we have that 
		$\max_{\mathbb R}\bar p_m =\bar p_m(x_0^m)\to0$, as $m\to+\infty$. This together with $
		0<\bar p_m(x)\le \max_{\mathbb R}\bar p_m$ for all $x\in\mathbb R,$ yields $
		\|\bar p_m\|_{L^\infty(\mathbb R)}\to0 \text{ as }m\to+\infty.$
		This contradicts the assumption that $0<\bar p_m(x)<\theta$ for all $x\in\mathbb R$.
		\item If $\lim_{m\to+\infty}\bar p_m(0)=\theta$,  then by using interface conditions and integrating the first equation of~\eqref{sta3} from $-l_1^m$ to $0$, we have that
		\begin{equation*}\label{flux_limit}
			\bar{p}_m'(0^-) - \bar{p}_m'\big((-l_1^m)^+\big) = \sigma \int_{-l_1^m}^{0} -\frac{f_1( \bar{p}_m(s))}{d_1} \mathrm{d}s.
		\end{equation*}
		Using~\eqref{f1} and $l_1^m \to 0$ as $m\to+\infty$, we have that $\bar{p}_{\infty}'(0^+) = \bar{p}_{\infty}'(0^-)$, where $\bar{p}_{\infty} := \lim_{m\to+\infty}\bar{p}_m$. Since $l_2^m \to +\infty$ as $m\to+\infty$, $\bar{p}_{\infty}$ satisfies
		\begin{equation*}
			\begin{cases}
				-d_2\bar{p}_{\infty}''(x)=f_2(\bar{p}_{\infty}(x)), & x\in\R,\\
				\bar{p}_{\infty}(0) = \theta, \, \bar{p}_{\infty}'(0) = 0.
			\end{cases}
		\end{equation*}
		The Cauchy-Lipschitz theorem then implies $\bar{p}_{\infty} \equiv \theta$ on $\mathbb{R}$. This contradicts the instability of $\theta$ in the bistable patch.
		\item If $\lim_{m\to+\infty}\bar{p}_m(0)=\gamma\in(0,\theta)$, then by similar arguments as above, passing to the limit in~\eqref{sta3} as $m\to+\infty$ yields $\bar{p}_\infty$ satisfying
		\begin{equation}\label{ode_bistable}
			\begin{cases}
				-d_2\bar{p}_{\infty}''(x) = f_2(\bar{p}_{\infty}(x)), & x \in \mathbb{R},\\[1mm]
				\bar{p}_{\infty}(0) = \gamma, \quad \bar{p}_{\infty}'(0) = 0,
			\end{cases}
		\end{equation}
		together with $0 \le \bar{p}_{\infty}(x) \le \theta$ for $x \in \mathbb{R}$. The planar system associated with \eqref{ode_bistable}, namely $u'=v$ and $v'=-\frac{1}{d_2}f_2(u)$, is Hamiltonian with the first integral 
		$$
		H(u,v) := \frac{d_2}{2}v^2 + F_2(u), \quad \text{where } F_2(u) := \int_0^u f_2(s)\mathrm{d}s.
		$$
		Thus, the trajectory of $\bar{p}_{\infty}$ is confined to the energy level set
		 $H\big(\bar{p}_{\infty}(x), \bar{p}_{\infty}'(x)\big) \equiv F_2(\gamma)$ for all $x \in \mathbb{R}$. Since $f_2(\theta) = 0$ and $f_2'(\theta) > 0$, we have $F_2''(\theta) = f_2'(\theta) > 0$. This implies that $(\theta,0)$ is a strict local minimum of $H(u,v)$, designating it as a spatial center of the phase portrait. Given that $\gamma \in (0, \theta)$, the trajectory passing through $(\gamma, 0)$ is a non-trivial closed periodic orbit surrounding $(\theta, 0)$. Therefore, $\bar p_\infty$ is periodic in $x$ and oscillates around $\theta$. In particular, there exists some $x_0 \in \mathbb{R}$ such that $\bar{p}_{\infty}(x_0) > \theta$. This contradicts the fact that $0\leq \bar p_\infty(x)\leq \theta$ for $x\in\mathbb R.$
	\end{itemize}
	This completes the proof of \eqref{claim6.1}, and hence the proof of Lemma~\ref{unstable}.
\end{proof}

\medspace

A consequence of Lemma~\ref{unstable} is the following non-existence result. 

\begin{lemma}\label{noexstate}
	Assume that $\int_0^{K_2} f_2(s) \mathrm{d} s>0$. There exist $\hat{l}_1>0$ and $\hat{l}_2>0$ such that, for every $0<l_1<\hat{l}_1$ and $l_2>\hat{l}_2$, and for every periodic steady state $\bar{p}$ of~\eqref{eq}-\eqref{patch} such that $0<\bar{p}<p^*$, there is no steady state $v$ of~\eqref{eq}-\eqref{patch} such that $0<v<\bar{p}$ and there is no steady state~$w$ of~\eqref{eq}-\eqref{patch} such that $\bar{p}<w<p^*$.
\end{lemma}

\begin{proof} 
Let $\hat{l}_1$ and $\hat{l}_2$ be given by Lemma~\ref{unstable}. Without loss of generality, we can assume that $\hat{l}_1<L_1^c$ and $\hat{l}_2>\max\{2R,l_2^c\}$, where $R>0$ is given in Lemma~\ref{lem-Phi2}. Then let $l_1$ and $l_2$ be such that $0<l_1<\hat{l}_1$ and $l_2>\hat{l}_2$. We only prove the first conclusion, since the proof of the second one is similar. Let $\bar{p}$ be a periodic steady state of~\eqref{eq}-\eqref{patch} such that $0<\bar{p}<p^*$ and let $v$ be a steady state of~\eqref{eq}-\eqref{patch} such that $0\le v<\bar{p}$. Our goal is to show that $v\equiv0$ in $\R$.
	
	\medspace
	
\noindent	{\it Step 1: We show that~$\sup_{x\in\R}\big(v(x)-\bar{p}(x)\big)<0$.} First, since $\bar{\lambda}_{1}(l_1,l_2,\bar{p})<0$ by Lemma~\ref{unstable}, and there is~$R_2>l/2$ such that~$\bar{\lambda}_{1}^{R_2}(l_1,l_2,\bar{p})<0,$ 
	where~$\bar{\lambda}_{1}^{R_2}(l_1,l_2,\bar{p})$ is the principal eigenvalue of~\eqref{truncted} with $R$ is replaced by $R_2$. For any $\varepsilon >0$, define
	\begin{equation}\label{defveps}
		v_{\varepsilon}(x)=\left\{
		\begin{array}{lll}
			\bar{p}(x)-\varepsilon\psi_{R_2}(x)\; & {\rm if}\,\,\,  |x|<R_2,\vspace{3pt}\\
			\bar{p}(x)\; & {\rm if} \,\,\,|x|\geq R_2,
		\end{array}\right.
	\end{equation}
	where $\psi_{R_2}$ is a fixed positive eigenfunction of~\eqref{truncted} corresponding to $\bar{\lambda}_{1}^{R_2}(l_1,l_2,\bar{p})$. With similar calculation as the proof of Proposition~\ref{pro_pl1l2}, there is $\varepsilon_0>0$ small enough such that $0<v_{\varepsilon}\le\bar{p}$ in $\R$ for all $0<\varepsilon\le\varepsilon_0$, and $v_{\varepsilon}$ satisfies
	\begin{equation}\label{supersol}
	\begin{split}
		\partial_{t}v_{\varepsilon}-d(x)\partial_{xx}v_{\varepsilon}-f(x,v_{\varepsilon}) >0 \text{ in }x\in(-R_2,R_2)\backslash S,
	\end{split}
\end{equation}
 as well as the interface conditions. This together with the fact that $\bar{p}$ is a stationary solution of equation~\eqref{eq}-\eqref{patch} and that $v'_{\varepsilon}(x^-)\geq v'_{\varepsilon}(x^+)$ at $x=\pm R_2$ imply that $v_{\varepsilon}$ is a supersolution of equation~\eqref{eq}-\eqref{patch}.
	
	Now, for any $k\in\Z$ and for any $0<\varepsilon\le\varepsilon_0$, the function $v_{\varepsilon}(\cdot-kl)$ is also a supersolution of~\eqref{eq}-\eqref{patch}. Since $v<\bar{p}$ in $\R$, it follows then from the strong maximum principle that, for every $0<\varepsilon\le\varepsilon_0$, there holds $v(x)<v_{\varepsilon}(x-kl)$ for all $x\in(kl-R_2,kl+R_2)$ and for all~$k\in\Z$. Since $R_2>l/2$ and $\psi_{R_2}$ is continuous and positive in $(-R_2,R_2)$, one infers that $\sup_{x\in\R}\big(v(x)-\bar{p}(x)\big)<0$.
	
	\medspace
	
\noindent	{\it Step 2: We show that~$v\equiv 0$ in $\R$.} Let $\phi$ be a principal eigenfunction of the periodic problem~\eqref{eigenv2}, associated with the principal eigenvalue~$\bar{\lambda}_1(l_1,l_2,\bar{p})$. With similar calculations as in the proof of Proposition~\ref{pro_pl1l2}, there is~$\eta_0>0$ such that for all~$0<\eta\le\eta_0$, the periodic function $\bar{p}-\eta\phi$ satisfies $v<\bar{p}-\eta\phi<\bar{p}$ and 
is indeed a strict supersolution of~\eqref{eq}-\eqref{patch}. As a consequence, the solution $u(t,x;\bar{p}-\eta_0\phi)$ of~\eqref{eq}-\eqref{patch} with initial condition $\bar{p}-\eta_0\phi$ is decreasing in $t>0$ and, from parabolic estimates in~\cite[Theorem~2.2]{HLZ-2024}, it converges as $t\to+\infty$, locally uniformly in $\R$, to a periodic steady state $p_{\infty}$ of~\eqref{eq}-\eqref{patch} such that $u(t,\cdot;\bar{p}-\eta_0\phi)|_{\bar{I}}\to p_{\infty}|_{\bar{I}}$ in $C^2(\bar{I})$ for each patch $I\subset\R$ as $t\to+\infty$. Moreover, 
$$
0\le v\le p_{\infty}<\bar{p}-\eta_0\phi<\bar{p}<p^*\hbox{ in }\R.
$$
	If $p_{\infty}\not\equiv0$ in $\R$, then $0<p_{\infty}<p^*$ from the strong maximum principle, whence $p_{\infty}$ is unstable from Lemma~\ref{unstable}, in the sense that $\bar{\lambda}_1(l_1,l_2,p_{\infty})<0$. Therefore, as above, by calling $\varphi$ a principal periodic eigenfunction of the periodic problem~\eqref{eigenv2} associated with $\bar{\lambda}_1(l_1,l_2,p_{\infty})$, it follows that the functions~$p_{\infty}+\kappa\varphi$ are subsolutions of~\eqref{eq}-\eqref{patch} for all $\kappa>0$ small enough. In particular, since $p_{\infty}<\bar{p}-\eta_0\phi$ in~$\R$ and both functions are periodic and continuous, there is $\kappa_0>0$ such that $p_{\infty}+\kappa_0\varphi$ is a subsolution of~\eqref{eq}-\eqref{patch} and $p_{\infty}+\kappa_0\varphi<\bar{p}-\eta_0\phi$ in $\R$, whence $p_{\infty}+\kappa_0\varphi<u(t,\cdot;\bar{p}-\eta_0\phi)$ in $\R$ for all $t>0$, from the comparison principle. Finally, passing to the limit as $t\to+\infty$ gives $p_{\infty}+\kappa_0\varphi\le p_{\infty}$ in $\R$, which is impossible. Hence, $p_{\infty}\equiv 0$ and $v\equiv 0$ in $\R$, and the proof of Lemma~\ref{noexstate} is complete.
\end{proof}

\medspace

As a consequence of Lemma~\ref{noexstate}, for every $0<l_1<\hat{l}_1$ and $l_2>\hat{l}_2$, and any periodic steady state $\bar{p}$ satisfying $0<\bar{p}<p^*$, system~\eqref{eq}-\eqref{patch} restricted to 
$$
E_1=\{u\in C(\R,[0,p^*])\,|\,0\leq u\leq \bar{p}\hbox{ in }\R\},
$$
and to
$$
E_2=\{u\in C(\R,[0,p^*])\,|\, \bar{p}\leq u\leq p^*\hbox{ in }\R\},
$$
respectively, has a monostable structure. 

In order to prove the existence of pulsating fronts for~\eqref{eq}-\eqref{patch}, one will verify a counter-propagation condition on the spreading speeds of these subsystems, as defined in~\cite{FZ-2015}. To do so, denote
$$
\mathcal{C}^-(0,\bar p) = 
\left\{
u\in C(\mathbb{R},[0,p^*])
\;\middle|\;
\begin{aligned}
	&0\le u\le \bar p,\,\limsup_{x\to+\infty}(u(x)-\bar p(x))<0,\\
	&u(x)=\bar p(x) \quad\,\text{ for } x\ll -1
\end{aligned}
\right\},
$$
$$
\mathcal{C}^+(\bar{p},p^*)= 
\left\{
\begin{aligned}
	u\in C(\mathbb{R},[0,p^*])
\end{aligned}
\;\middle|\;
\begin{aligned}
	&\bar{p}\leq u \le p^*,\,\liminf_{x\to-\infty}(u(x)-\bar{p}(x))>0,	\\
	&u(x)=\bar{p}(x) \quad\hbox{ for } x \gg 1
\end{aligned}
\right\}.
$$

\begin{lemma}\label{spreadingspeed}
Assume that $\int_0^{K_2} f_2(s) \mathrm{d} s>0$. Let $\hat{l}_1>0$ and $\hat{l}_2>0$ be given by Lemma~\ref{noexstate}. For every $0<l_1<\hat{l}_1$ and $l_2>\hat{l}_2$, and for every periodic steady state $\bar{p}$ of~\eqref{eq}-\eqref{patch} such that $0<\bar{p}<p^*$, there are some real numbers $c^+>0$ and $c^->0$ such that
	\begin{equation}\label{rightspeed}
		\begin{cases}
			\displaystyle\mathop{\limsup}_{t\to+\infty,\, x\geq -c^-t} u(t,x;u_0)=0 & \hbox{for all }u_0\in\mathcal{C}^-(0,\bar{p}),\vspace{3pt}\\
			\displaystyle\mathop{\limsup}_{t\to+\infty,\, x\leq c^+t} |u(t,x;\tilde{u}_0)-p^*(x)|=0 & \hbox{for all }\tilde{u}_0\in\mathcal{C}^+(\bar{p},p^*).
		\end{cases}
	\end{equation}
\end{lemma}

\begin{proof}
	We only give the proof of the first assertion~\eqref{rightspeed}, since the arguments for the other one are similar. First, one claims that for any $u_0\in\mathcal{C}^-(0,\bar{p})$ and any constant $C\in\R$, there holds
	\begin{equation}\label{convhalf}
		\lim_{t\to+\infty} u(t,x;u_0)=0\quad \hbox{uniformly for}\,\,x\in[C,+\infty).
	\end{equation}
	So, fix any $u_0\in\mathcal{C}^-(0,\bar{p})$, any real number $C$, any $k_0\in\NN$ such that $C\ge-k_0l$, and let~$\eta>0$ be arbitrary. There are then $\varepsilon>0$ small enough and $m_0\in\NN$ large enough such that
	\begin{equation}\label{u0}
	  u_0(\cdot+ml)\leq v_{\varepsilon}\ \hbox{ for all }m\ge m_0,\ m\in\NN,
	\end{equation}
	where $v_{\varepsilon}$ is defined in~\eqref{defveps} and is a strict supersolution of~\eqref{eq}-\eqref{patch}, in the sense of~\eqref{supersol}, as well as the interface conditions hold. It follows from the comparison principle that $u(t,x;v_{\varepsilon})<v_{\varepsilon}(x)$ and $u(t,x;v_{\varepsilon})$ is decreasing in $t>0$. By parabolic estimates in~\cite[Theorem~2.2]{HLZ-2024}, $u(t,x;v_{\varepsilon})$ converges as $t\to+\infty$ locally uniformly in~$x\in\R$ to a stationary solution $v_{\varepsilon}^{\infty}$ of equation~\eqref{eq}-\eqref{patch} and $u(t,x;v_{\varepsilon})|_{\bar{I}}\to v_{\varepsilon}^{\infty}|_{\bar{I}}$ in $C^2(\bar{I})$ for each patch $I\subset\R$, with $0\leq v_{\varepsilon}^{\infty}<\bar{p}$. Lemma~\ref{noexstate} and the strong maximum principle imply that~$v_{\varepsilon}^{\infty}\equiv 0$. Therefore, there is $T>0$ such that
	$$
	0\le u(t,y;v_{\varepsilon})\le\eta\ \hbox{ for all }t\ge T\hbox{ and }|y|\le(k_0+m_0+1)l.
	$$
	For any $x\ge C\,(\ge-k_0l)$, there is $a_x\in\Z$ such that $a_x\ge-k_0$ and $a_xl\le x\le(a_x+1)l$. With $m_x=k_0+m_0+a_x\ge m_0$, one has $|x-m_xl|\le|x-a_xl|+(k_0+m_0)l\le(k_0+m_0+1)l$. Hence, from the comparison principle and the periodicity of~\eqref{eq}-\eqref{patch}, it follows from~\eqref{u0} that, for all $t\ge T$,
	$$
	0\le u(t,x;u_0)=u(t,x-m_xl;u_0(\cdot+m_xl))\le u(t,x-m_xl;v_{\varepsilon})\le\eta.
	$$
	The claim~\eqref{convhalf} is thereby proved.
	
	Next, we fix a real number $\delta$ such that $0<\delta<\min_{\R}\bar{p}$ and a function $w_0\in\mathcal{C}^-(0,\bar{p})$ such that $\delta\le w_0\le\bar{p}$ in $\R$ and $w_0=\delta$ in $\R^+$. From~\eqref{convhalf} applied to $w_0$, and since $0\le u(t,x;w_0)\le\bar{p}(x)$ for all $t\ge0$ and $x\in\R$, there is a time $t_1>0$ such that $0\le u(t_1,x;w_0)\le w_0(x+l)$ for all~$x\in\R$. From the comparison principle, it follows by immediate induction that
	\begin{equation}\label{nt1}
			0\le u(mt_1,x;w_0)\le w_0(x+ml)\ \hbox{ for all }m\in\NN\hbox{ and }x\in\R.
	\end{equation}
	
	Finally, one shows that the first assertion in~\eqref{rightspeed} holds with any positive constant $c^-$ such that $0<c^-<l/t_1$. Fix any function $u_0\in\mathcal{C}^-(0,\bar{p})$. By~\eqref{convhalf} and $u(t,\cdot;u_0)\le\bar{p}$, there is $T>0$ such that $0\le u(T,\cdot;u_0)\le w_0$, whence
	\begin{equation}\label{nt1bis}
		0\le u(T+mt_1,x;u_0)\le u(mt_1,x;w_0)\le w_0(x+ml)\ \hbox{ for all }m\in\NN\hbox{ and }x\in\R
	\end{equation}
	by~\eqref{nt1} and the comparison principle. Let us now argue by contradiction and assume that ${\limsup}_{t\to+\infty,\, x\ge-c^-t} u(t,x;u_0)>0$. Then there are some sequences $(\tau_k)_{k\in\NN}$ in $(0,+\infty)$ and $(x_k)_{k\in\NN}$ in $\R$ such that $x_k\ge-c^-\tau_k$ for all $k\in\NN$, $\tau_k\to+\infty$ as $k\to+\infty$ and $\liminf_{k\to+\infty}u(\tau_k,x_k;u_0)>0$. For $k$ large enough, one can write $\tau_k=T+m_kt_1+\tilde{\tau}_k$ with $m_k\in\NN$, $0\le\tilde{\tau}_k\le t_1$ and $m_k\to+\infty$ as $k\to+\infty$. Write also $x_k=x'_k+x''_k$ with $x'_k\in l\Z$ and $x''_k\in(-l_1,l_2]$. Up to extraction of a subsequence, one can assume that $\tilde{\tau}_k\to\tau\in\R$ and $x''_k\to y\in[-l_1,l_2]$ as $k\to+\infty$. For $k$ large enough, denote
	$$u_k(t,x)=u(t+\tau_k,x+x'_k;u_0)\ \hbox{ for }t\ge-\tau_k,\ x\in\R.$$
	From parabolic estimates in~\cite[Theorem~2.2]{HLZ-2024}, the functions $u_k$ converge locally uniformly in $\R^2$, up
	to extraction of a subsequence, to a solution $u_{\infty}(t,x)$ of~\eqref{eq}-\eqref{patch} defined
	for all $(t,x)\in\R^2$ such that $u_k|_{\R\times\bar{I}}\to u_{\infty}|_{\R\times\bar{I}}$ in $C_{t,x}^{1,2}(\R\times\bar{I})$ for each patch $I\subset\R$, and
	$0\le u_{\infty}(t,x)\le\bar{p}(x)$ for all $(t,x)\in\R^2$, while
	$u_{\infty}(0,y)>0$. Furthermore, for any given~$h\in\Z$ and $x\in\R$, one has, for all $k$ large enough,
	\begin{equation}\label{ukmt1}
		0\le u_k(-ht_1-\tilde{\tau}_k,x)=u(T+m_kt_1-ht_1,x+x'_k;u_0)\le w_0(x+x'_k+(m_k-h)l)
	\end{equation}
	by~\eqref{nt1bis}. Note that
	\begin{equation*}
		\begin{aligned}
			x'_k+m_kl&\geq x_k-l_2+m_kl\\
			&\geq -c^-(T+m_kt_1+t_1)+l_1+(m_k-1)l\\
			&=(m_k-1)(-c^-t_1+l)-c^-(T+2t_1)+l_1.
		\end{aligned}
	\end{equation*}
	Thus $x'_k+m_kl\to+\infty$ as~$k\to+\infty$ since $c^-<l/t_1$, $m_k\to+\infty$ and
	$l_1<\hat{l}_1$. As a consequence, it follows from~\eqref{ukmt1} and the definitions
	of~$u_{\infty}$ and~$w_0$ that $0\le u_{\infty}(-ht_1-\tau,x)\le\delta\le w_0(x)$ for all $h\in\Z$
	and $x\in\R$. One infers that $u_{\infty}\equiv 0$ in $\R^2$. Indeed, for any $(t,x)\in\R^2$, one
	has, for all $h\in\NN$ large enough,
	$$0\le u_{\infty}(t,x)=u(t+ht_1+\tau,x;u_{\infty}(-ht_1-\tau,\cdot))\le u(t+ht_1+\tau,x;w_0).$$
	The property~\eqref{convhalf} applied with $w_0$ implies that $u(t+ht_1+\tau,x;w_0)\to0$ as $h\to+\infty$ whence~$u_{\infty}(t,x)=0$ for all $(t,x)\in\R^2$, which contradicts $u_{\infty}(0,y)>0$. Therefore, the first assertion of~\eqref{rightspeed} is shown and the proof of Lemma~\ref{spreadingspeed} is complete.
\end{proof}

\vs

Based on the above preparations, one is ready to prove Theorem~\ref{bisTW}.

\medspace

\begin{proof}[Proof of Theorem~\ref{bisTW}] 
	Let $\hat{l}_1>0$ and $\hat{l}_2>0$ be given as in Lemma~\ref{noexstate}. Fix any $0<l_1<\hat{l}_1$ and $l_2>\hat{l}_2$. For any $t\geq 0$, define $\widetilde{Q}_t: \mathcal{C}_{p^*}\to \mathcal{C}_{p^*}$ by
	\begin{equation*}\label{semiflow}
		\widetilde{Q}_t[u_0]=u(t,\cdot;u_0),
	\end{equation*}
	where $\mathcal{C}_{p^*}$ is defined by~\eqref{Cp} with $p$ replaced by $p^*$.
	By arguments similar to those used in the proof of Theorem~\ref{thmTW}, together with Lemmas~\ref{unstable} and~\ref{spreadingspeed}, the semiflow $(\widetilde{Q}_t)_{t\ge0}$ satisfies the following properties:
	\begin{itemize}
		\item[(A1)](Periodicity) $T_y\big[\widetilde{Q}_t[\varphi]\big]=\widetilde{Q}_t\big[T_y[\varphi]\big]$ for all $\varphi\in  \mathcal{C}_{p^*}$, $t>0$ and $y\in l\Z$, where $T_y: \mathcal{C}_{p^*}\to  \mathcal{C}_{p^*}$ is the translation operator defined by $T_y[\psi]=\psi(\cdot-y)$.
		\item[(A2)](Continuity) For any $t>0$, $\widetilde{Q}_t$ is continuous with respect to the compact open topology.
		\item[(A3)](Monotonicity) For any $t>0$, $\widetilde{Q}_t$ is order preserving in the sense that $\widetilde{Q}_t[\varphi_1]\geq \widetilde{Q}_t[\varphi_2]$ whenever $\varphi_1\ge\varphi_2$ in $ \mathcal{C}_{p^*}$.
		\item[(A4)](Compactness) For any $t>0$, $\widetilde{Q}_t$ is compact with respect to the compact open topology.
		\item[(A5)](Bistability) Define $\mathcal{C}_{per} :=\{v\in\mathcal{C}: v(x)=v(x+l), \,0\leq v(x)\leq p^*(x),\,\forall x\in\R\}$.
		For any $t~\!>~\!0$,~$\widetilde{Q}_t$ maps~$\mathcal{C}_{per}$ to itself and is strongly monotone on $\mathcal{C}_{per}$ in the sense that $\inf_{x\in\R}\big(\widetilde{Q}_t[\varphi_1](x)-\widetilde{Q}_t[\varphi_2](x)\big)>0$ whenever $\varphi_1\ge\varphi_2$ in $\mathcal{C}_{per}$ with $\varphi_1\not\equiv\varphi_2$. 	
			Furthermore, on the one hand, the function $p^*\in \mathcal{C}_{per}$ is a stationary solution of~\eqref{eq}-\eqref{patch}. It follows from~Lemma~\ref{lem_strong} that ~$p^*$ is strongly stable from below, namely for every~$t>0$ there is~$\eta_0>0$ such that
			$$
		 \inf_{x\in\R}\big(\widetilde{Q}_t[p^*-\eta\phi](x)-(p^*(x)-\eta\phi(x))\big)>0,\quad \forall 0<\eta\le\eta_0,
			$$ 
where $\phi$ denotes the periodic principal eigenfunction of~\eqref{eigenv2} with~$\bar{\lambda}_1(l_1,l_2,p^*)>0$. On the other hand, from Proposition~\ref{lem-STA}(ii) that $0$ is a stable solution (i.e. $\lambda_1>0$) of problem~\eqref{eq}-\eqref{patch} and it is strongly stable from above when $l_1<\hat{l}_1<L_1^c$ and $l_2>\hat{l}_2>l_2^c$, in the sense of~\cite{FZ-2015}, namely, for every~$t>0$ there is~$\eta_0>0$ such that 
			$$
			\sup_{x\in\R}\big(\widetilde{Q}_t[\eta](x)-\eta\big)<0 \quad 0<\eta\le\eta_0.
			$$ 
	
		Lastly, any stationary solution~$0<\bar{u}<p^*$ in~$\mathcal{C}_{per}$ is strongly unstable from above and below in the sense of~\cite{FZ-2015} since for every $t>0$, there is~$\epsilon_0>0$ such that 
		$$
		\inf_{x\in\R}\big(\widetilde{Q}_t[\overline{p}+\epsilon\varphi](x)-(\overline{p}(x)+\epsilon\varphi(x))\big)>0
		$$ 
		and $$
		\sup_{x\in\R}\big(\widetilde{Q}_t[\overline{p}-\epsilon\varphi](x)-(\overline{p}(x)-\epsilon\varphi(x))\big)<0
		$$ in $\R$ for all $0<\epsilon\le\epsilon_0$, where $\varphi$ denotes the periodic principal eigenfunction of~\eqref{eigenv2} with~$\lambda=\bar{\lambda}_1(l_1,l_2,\bar{p})>0$.  Indeed, for every $t>0$, the inequalities~$\widetilde{Q}_t[\overline{p}+\epsilon\varphi]>\overline{p}+\epsilon\varphi$ and $\widetilde{Q}_t[\overline{p}-\epsilon\varphi]<\overline{p}-\epsilon\varphi$ in~$\R$ for all $0<\epsilon\le\epsilon_0$ follow from the fact that the periodic functions $\overline{p}+\epsilon\varphi$ and $\overline{p}-\epsilon\varphi$ are respectively strict sub- and supersolutions of the elliptic equation associated with~\eqref{eq}-\eqref{patch},  which can be verified by calculations similar to~\eqref{supersol1}.
		\item[(A6)](Counter-propagation)  For each stationary solution $\bar{p}\in \mathcal{C}_{per}$ with $0<\bar{p}<p^*$, one has~$c^-_*(0,\bar{p})+c^+_*(\bar{p},p^*)>0$, where $c^-_*(0,\bar{p})$ and $c^+_*(\bar{p},p^*)$ are the spreading speeds defined by
		$$\begin{aligned}
			c^-_*(0,\bar{p}) & = \sup\big\{c\in\R\,\big|\, \limsup_{t\to+\infty,\,x\geq-ct} u(t,x;u_0)=0,\quad \forall\,u_0\in\mathcal{C}^-(0,\bar{p})\big\},\vspace{3pt}\\
			c^+_*(\bar{p},p^*) & = \sup\big\{c\in\R\,\big|\, \limsup_{t\to+\infty,\,x\leq ct} |u(t,x;\tilde{u}_0)-p^*|=0,\quad \forall\,\tilde{u}_0\in\mathcal{C}^+(\bar{p},p^*)\big\}.
		\end{aligned}$$
		Indeed, Lemma~\ref{spreadingspeed} implies that $c^-_*(0,\bar{p})\ge c^->0$ and $c^+_*(\bar{p},p^*) \ge c^+>0$, with the notations of Lemma~\ref{spreadingspeed}. Following~\cite{FZ-2015}, $c^-_*(0,\bar{p})$ is called the leftward spreading speed of equation~\eqref{eq}-\eqref{patch} on $\mathcal{C}^-(0,\bar{p})$, and $c^+_*(\bar{p},p^*)$ the rightward spreading speed of equation~\eqref{eq}-\eqref{patch} on~$\mathcal{C}^+(\bar{p},p^*)$ (Lemma~\ref{spreadingspeed} is then stronger than the counter-propagation condition given in~\cite{FZ-2015}, which is just defined as the positivity of the sum of these spreading speeds).
	\end{itemize}
	
	Having in hand the properties (A1)-(A6),  we then see from \cite[Proposition 3.1 and Theorems~3.4 and~4.1, and Remark~4.1]{FZ-2015} that for any $0<l_1<\hat{l}_1$ and $l_2>\hat{l}_2$, equation~\eqref{eq}-\eqref{patch} admits a pulsating front~$\widetilde{W}(x-ct,x)$ with speed $c\in\R$ such that~$\widetilde{W}(\xi,x)$ is nonincreasing in~$\xi$.
	
	Furthermore, we show that $c\geq 0$, up to decreasing $\hat{l}_1$ and increasing $\hat{l}_2$ if necessary. By contradiction, we assume that there exist sequences $l_1^m, l_2^m \in (0, +\infty)$ such that $l_1^m \to 0$ and $l_2^m \to +\infty$ as $m \to +\infty$, with the corresponding wave speeds satisfying $c_m < 0$ for all $m \in \mathbb{N}$.  Let $R>0$ and $\phi_2\in C^2([-R,R])$ be given by Lemma~\ref{lem-Phi2}. For sufficiently large $m$ such that $l_2^m \ge 2R$, we can choose $x_0^m \in (0, l_2^m)$ such that $[x_0^m-R, x_0^m+R] \subset (0, l_2^m)$. We then define a compactly supported function~$\omega_0^m$ on $\R$ by
	\begin{equation*}\label{tildev0}
		\omega_0^m(x)=
		\begin{cases}
			\phi_2(x-x_0^m), & x\in [x_0^m-R, x_0^m+R], \\
			0, & \text{otherwise}.
		\end{cases}
	\end{equation*}
	Similar to the proof of Proposition~\ref{pro_pl1l2}, $\omega_0^m$ is a subsolution to~\eqref{eq}-\eqref{patch}. Let $u^m(t,x)$ be the solution to~\eqref{eq}-\eqref{patch} with the initial datum $u^m(0,x) = \omega_0^m(x)$ for $t>0$ and $x\in\R$. We have that $u^m(t,x)$ is nondecreasing in $t$, and since $\omega_0^m \not\equiv 0$, the strong maximum principle implies that $u^m(t,x) > 0$ for all $t > 0$ and $x \in \R$. In particular, we have
	\begin{equation}\label{lower_bound}
		u^m(t, x_0^m) \ge u^m(0, x_0^m) = \phi_2(0) > \theta, \quad \forall t > 0.
	\end{equation}
	
	On the other hand, consider the pulsating front $\widetilde{W}(x-c_m t, x)$ connecting $p_m^*$ to $0$, with $p_m^*$ satisfying~\eqref{sta4}. By~\eqref{claim_pK2}, we have $p_m^*\to K_2$ as $m\to+\infty$ uniformly in $\R$.
	Recall from Lemma~\ref{lem-Phi2} that $0\leq \phi_2<K_2$ with $\max_{[-R,R]}\phi_2=\phi_2(0)>\theta$. Thus, we can fix $\varepsilon>0$ small enough such that $K_2-\varepsilon>\phi_2(0)$. Then, for all sufficiently large $m$, we have that
	$$
	\lim_{\xi \to -\infty} \widetilde{W}(\xi, \cdot) = p_m^*\ge K_2 - \varepsilon\text{ uniformly in }\R.
	$$ 
	We can choose $\xi_0 \ll-1$ such that $\omega_0^m(x) \le \widetilde{W}(x + \xi_0, x)$ for $x\in\R$. Applying the comparison principle (Proposition~\ref{proCP}) yields
	\begin{equation}\label{upper_bound}
		u^m(t, x) \le \widetilde{W}(x - c_m t + \xi_0, x), \quad \forall t > 0, \, x \in \R.
	\end{equation}
	Since $c_m < 0$ and $x_0^m\in(0,l_2^m)$ for all $m\in\NN$, we have $\lim_{t \to +\infty} (x_0^m - c_m t + \xi_0) = +\infty$ for all $m\in\NN$. Recall that $\lim_{\xi \to +\infty} \widetilde{W}(\xi, x) = 0$ uniformly with respect to $x \in \R$. Therefore, using parabolic estimate in~\cite[Theorem~2.2]{HLZ-2022} and passing to the limit as $t \to +\infty$ in \eqref{upper_bound} gives
	$$
	\lim_{t \to +\infty} u^m(t, x_0^m) \le \lim_{t \to +\infty} \widetilde{W}(x_0^m - c_m t + \xi_0, x_0^m) = 0.
	$$
	This directly contradicts ~\eqref{lower_bound}. Thus, we have that $c\geq0$. 
	
To end the proof, one needs to show that $c>0$ for all $l_1<\hat{l}_1$ and $l_2>\hat{l}_2$, up to decreasing $\hat{l}_1$ and increasing $\hat{l}_2$ again if necessary. Assume by contradiction that there are sequences $l_1^m$ and $l_2^m$ in $(0,+\infty)$,  such that $l_1^m\to0$, $l_2^m\to+\infty$ as $m\to+\infty$, and $c_{m}=0$ for all $m\in\NN$. Namely, for each $m\in\NN$, there is function $\widetilde{W}_{m}:\R\to\R$ which is continuous, and $\widetilde{W}_{m}\in C^2(\bar{I}_m)$ for each patch $I_m$ as in~\eqref{Im}. Moreover, $\widetilde{W}_{m}$ satisfies
\begin{equation*}\label{stationary}
	\begin{cases}
		-d(x) \widetilde{W}_{m}''(x)= f(x,\widetilde{W}_{m}), & x\in\R\backslash S_m,\\
		\widetilde{W}_{m}(x^-)=\widetilde{W}_{m}(x^+),\quad \widetilde{W}_{m}'(x^-)=\sigma \widetilde{W}_{m}'(x^+), & x=nl^m,\\
		\widetilde{W}_{m}(x^-)=\widetilde{W}_{m}(x^+),\quad \sigma \widetilde{W}_{m}'(x^-)=\widetilde{W}_{m}'(x^+), & x=nl^m+l_2^m,\\
		\widetilde{W}_{m}(-\infty)=p_{m}^*,\, \widetilde{W}_{m}(+\infty)=0\ \hbox{ and}\ 0<\widetilde{W}_{m}<p_{m}^*\hbox{ in }\R.
	\end{cases}
\end{equation*}
	Since $\int_0^{K_2}f_2(s)\,ds>0$ and $f_2$ satisfies~\eqref{f2}, there are $\theta^*\in(\theta,K_2)$ and $\gamma\in(\theta^*,K_2)$ such that
	\begin{equation}\label{upbound}
	\int_0^{\theta^*} f_2(s)\,ds=0, \hbox{ and } \int_0^\tau f_2(s)\,ds>0\hbox{ for }\tau\in [\gamma,K_2].
	\end{equation}
	 For every $m\in\NN$, there is $y_m\in\R$ such that $\widetilde{W}_m(y_m)=\gamma$. Write~$y_m=y_m'+\tilde{y}_m$, with $y_m'\in l^m\Z$ and $\tilde{y}_m\in (-l_1^m,l_2^m)$, and set $\widetilde{w}_m(x)=\widetilde{W}_m(x+y_m)$ for $x\in\R$ and~$m\in\NN$. Since $d(x)$ and $f(x,\cdot)$ are periodic in $x$, each function $\widetilde{w}_m$ obeys
	 \begin{equation*}
	 	\begin{cases}
	 		-d(x+\tilde{y}_m) \widetilde{w}_{m}''(x)= f(x+\tilde{y}_m,\widetilde{w}_{m}), & x\in\R\backslash S_m-\tilde{y}_m,\\
	 		\widetilde{w}_{m}(x^-)=\widetilde{w}_{m}(x^+),\quad \widetilde{w}_{m}'(x^-)=\sigma \widetilde{w}_{m}'(x^+), & x=nl^m-\tilde{y}_m,\\
	 		\widetilde{w}_{m}(x^-)=\widetilde{w}_{m}(x^+),\quad \sigma \widetilde{w}_{m}'(x^-)=\widetilde{w}_{m}'(x^+), & x=nl^m+l_2^m-\tilde{y}_m,\\
	 		\widetilde{w}_{m}(-\infty)=p_m^*,\, \widetilde{w}_{m}(+\infty)=0,\\
	 		\widetilde{w}_m(0)=\gamma,\,0<\widetilde{w}_{m}<p_m^*\hbox{ in }\R.
	 	\end{cases}
	 \end{equation*}
	 Similar to the proof of Lemma~\ref{lem_strong}, up to extraction of some subsequence, $\widetilde{w}_m\to \widetilde{w}_{\infty}$ as $m\to+\infty$ in $C^2_{\rm loc}(\R)$, where the function $0\le \widetilde{w}_{\infty}\le K_2$ solves
	\begin{equation}\label{limiteq}
		-d_2\widetilde{w}_{\infty}''+f_2(\widetilde{w}_{\infty})=0\hbox{ in }\R
	\end{equation}
	and $\widetilde{w}_{\infty}(0)=\gamma$, whence $0< \widetilde{w}_{\infty}<K_2$ in $\R$ from the strong elliptic maximum principle. As for equation~\eqref{w1} used in Lemma~\ref{unstable}, it follows from~\eqref{f2} that there are three types of solutions to equation~\eqref{limiteq}: the constant solutions (equal to $\theta$), the non-constant periodic solutions and the non-periodic ground state solutions converging to $0$ at $\pm\infty$.  In all cases, multiplying the equation~\eqref{limiteq} by $\widetilde{w}_{\infty}'$ and integrating on suitable intervals, it follows easily that
	$$
	\int_{\underline{\tau}}^{\bar{\tau}}f_2(s)ds=0,\quad \text{ where } 0\le\inf_{\R}\widetilde{w}_{\infty}=\underline{\tau}\le\overline{\tau}=\max_{\R}\widetilde{w}_{\infty}<K_2.
	$$ 
	It follows then from~\eqref{f2} and~\eqref{limiteq} that $\underline{\tau}\le\theta\le\bar{\tau}$ and that 
	$$
	\int_0^{\tau}f_2(s)ds\le0,\quad \forall0\le \tau\le\bar{\tau}.
	$$
	In particular, since~$\widetilde{w}_{\infty}(0)=\gamma\in(\theta^*,\bar{\tau}]$, one gets $$\int_0^{\gamma}f_2(s)ds\le0.$$ One has then reached a contradiction with~\eqref{upbound} and the proof of Theorem~\ref{bisTW} is thereby complete.
\end{proof}


\section{Extinction: proofs of Theorems~\ref{thmEXT1} and~\ref{thmEXT2}}\label{Sec-EXT}

In this section, we first give the proof of Theorem~\ref{thmEXT1}, that is, the solution $u$ of system~\eqref{eq}-\eqref{patch} will be extinct when $\lambda_1>0$, and the initial value $u_0$ is sufficiently small.

\medspace

\begin{proof}[Proof of Theorem~\ref{thmEXT1}]
Let $u$ be the solution to~\eqref{eq}-\eqref{patch} with a nonnegative continuous and compactly supported initial datum $u_0 \not \equiv 0$, satisfying $\|u_0\|_{L^{\infty}(\mathbb{R})} \leq \varepsilon$ with $\varepsilon>0$ to be determined later.
Let $(\lambda_1,\phi)$ be  the principal eigenpair of
problem~\eqref{eigenv}, with $\lambda_1>0$. Taking \(\kappa_0>0\)  small enough, 
for any \(\kappa\in(0,\kappa_0]\), there exists  a sufficiently small constant $0<\eta<\lambda_1$ such that
\begin{equation}\label{f}
	\|f_s(x,0)\kappa \phi(x)-f(x,\kappa \phi(x))\|_{L^{\infty}(\R)}\leq \eta\|\kappa \phi(x)\|_{L^{\infty}(\R)}=\kappa\eta.
\end{equation}

Given any $\kappa\in(0,\kappa_0]$, define
$$
\bar u(t,x)=\kappa e^{-\lambda t}\phi(x),\quad t\ge0,\ x\in\R,
$$
where $\lambda>0$ is chosen sufficiently small so that $\lambda_1\geq\lambda+\eta$. Since $0<\phi(x)\leq1$ and $\lambda_1>0$, together with~\eqref{eigenv} and~\eqref{f}, it follows that
\begin{equation*}
\begin{aligned}
\bar{u}_t - d(x) \bar{u}_{xx}-f(x,\bar{u})=&(\lambda_1-\lambda)\bar{u}+f_s(x,0)\bar{u}-f(x,\bar{u})\\
\geq&(\lambda_1-\lambda-\eta)\bar{u}\\
\geq&0.
\end{aligned}
\end{equation*}
In addition, it is clear that the interface conditions in~\eqref{eq} are satisfied.

We can select $\varepsilon > 0$ satisfying $0 < \varepsilon \leq \kappa \min_{\mathbb{R}} \phi(x)$ such that$$0 < u_0(x) \le \bar{u}(0,x), \quad x \in \mathbb{R}.
$$
Thus by using comparison principle (Proposition~\ref{proCP}), we obtain that
$$
0<u(t,x)\leq\bar{u}(t,x),\quad t\geq 0,~~ x\in\R.
$$
Passing to the limit $t\to+\infty$ then implies that
$$
\|u(t,\cdot)\|_{L^{\infty}(\R)}\to0 \text{ as } t\to+\infty.
$$
The proof of Theorem~\ref{thmEXT1} is thereby complete.
\end{proof}

\medspace

To prove Theorem~\ref{thmEXT2}, we introduce the following two key lemmas.  First, we use Lemma~\ref{lem-ex22} to demonstrate that initially compactly supported solution $u$ will remain small enough within finite times in the region far from the initial support. Subsequently, Lemma~\ref{lem-ex21} ensures the extinction of $u$ for $0<l_1<L_1^c$, provided that $l_2$ is sufficiently large and the initial datum $u_0$ (independent of $l_2$) is chosen to be sufficiently small.

\begin{lemma}\label{lem-ex22}
Let $R>0$, $A>0$ be fixed. Let $u$ be the solution of~\eqref{eq}-\eqref{patch} with a nonnegative continuous and compactly supported initial datum $u_0 \not \equiv 0$. For any $T>0$, $\varepsilon>0$, and $l_2\geq1$, there exists $\delta>0$ such that if $\diam(\supp(u_0)) \le R$, and $\|u_0\|_{L^{\infty}(\R)}\leq A$, then $0\leq u(t,x)\leq
\varepsilon$, for all $0\leq t\leq T$ and $\dist(x, \supp(u_0))\ge \delta$.
\end{lemma}
\begin{proof}
We proceed by contradiction. Suppose there exist $T > 0$,
$\tilde{\varepsilon} > 0$, $l_2 > 0$, and a sequence of solutions $(u_m)_{m\in\NN}$ to~\eqref{eq}-\eqref{patch} with initial datum
 $(u_{0}^m)_{m\in\NN}$ satisfying $\diam(\supp(u_{0}^m)) \le R$ and $\|u_{0}^m\|_{L^{\infty}(\mathbb{R})} \le A$, such that for each $m \in \mathbb{N}$, there exists $(t_m, x_m)$ with
\begin{equation}\label{contra_seq}
0 \le t_m \le T, \quad \dist(x_m, \supp(u_{0}^m)) \ge m, \quad \text{and} \quad u_m(t_m, x_m) > \tilde{\varepsilon}.
\end{equation}
Furthermore, without loss of generality, we may assume, up to translation, that
$\operatorname{supp}(u_0^m)\subset [-R/2, R/2]$ for all $m\in\mathbb{N}$.
Combining this with~\eqref{contra_seq} implies $|x_m|\to\infty$ as $m\to\infty$.

Next, we construct a supersolution that dominates the entire sequence 
$(u_m)_{m\in\NN}$. Let $\bar{u}_0$ be a compactly supported function, bounded in $L^{\infty}(\R)$ and belonging to $C^{2,\gamma}(\bar{I})$ for each $\gamma\in(0,1)$ and each patch $I\subset\R\backslash S$. We assume that
$$
\bar{u}_0\ge A\hbox{ in } [-R/2, R/2], \text{ and } \text{supp}(\bar{u}_0) \subset [-R/2-1, R/2+1].
$$
Moreover, $\bar{u}_0$ is chosen so that its spatial derivative vanishes at interface conditions contained in its support.
By construction, it follows that
\begin{equation*}\label{baru0}
0 \le u_{0}^m(x)  \le \bar{u}_0(x)\leq\|\bar{u}_0\|_{L^{\infty}(\R)}, \quad \forall x \in \mathbb{R}, \forall m \in \mathbb{N}.
\end{equation*}
Let $\bar{u}(t,x)$ be the unique solution to system~\eqref{eq}-\eqref{patch} with initial datum $\bar{u}_0$. Applying the comparison principle (Proposition~\ref{proCP}), we have
\begin{equation}\label{un_bar_u}
0 < u_m(t, x) \le \bar{u}(t, x)\leq\max\{K_1,K_2, \|\bar{u}_0\|_{L^{\infty}(\R)}\}, \quad \forall (t, x) \in [0, T] \times \mathbb{R}, \forall m \in \mathbb{N}.
\end{equation}
Combining this with assumption~\eqref{contra_seq}, we have that
\begin{equation}\label{super_contra}
\bar{u}(t_m, x_m) \ge u_m(t_m, x_m) > \tilde{\varepsilon}.
\end{equation}
Let $x_m = k_m l + \bar{x}_m$ with $k_m \in \mathbb{Z}$ and $\bar{x}_m \in (-l_1, l_2]$. Since $|x_m| \to \infty$ as $m\to+\infty$, and $\bar{x}_m$ is bounded, it follows that $|k_m| \to \infty$ as $m \to \infty$. We define the spatially shifted sequence as follows:
$$
v_m(t, x) := \bar{u}(t, x + k_ml), \quad (t, x) \in [0, T] \times \mathbb{R}.
$$
Clearly, $v_m$ satisfies the system~\eqref{eq}--\eqref{patch}. By
\eqref{un_bar_u}, the sequence $(v_m)_{m\in\mathbb{N}}$ is uniformly bounded. Furthermore, as $\bar{u}_0 \in C^{2,\gamma}(\bar{I})$ within each patch $I \subset \mathbb{R} \setminus S$, standard parabolic regularity estimates extend uniformly up to the initial time $t=0$. Following the arguments in~\cite[Theorem 2.2]{HLZ-2024}, we obtain
$$
\lVert v_{m}
\rVert_{C_{t;x}^{1,\gamma; 2,\gamma}([0,T]\times \bar{I})} \le C,
$$
where the constant $C$ depends only on $l_{1,2}$, $d_{1,2}$, $f_{1,2}$, $\sigma$, $T$, and $\|\bar{u}_0\|_{L^{\infty}(\mathbb{R})}$. By Arzel\`a-Ascoli theorem, there exists a subsequence (still denoted $v_m$) such that 
$v_m$ converges as $m\to +\infty$ locally uniformly for $x\in\mathbb{R}$ to $v_{\infty}$, where $v_{\infty}$ is a solution of the equation~\eqref{eq}-\eqref{patch} for $(t,x)\in[0,T]\times\R$. Moreover, $v_m|_{\bar{I}}\to v_{\infty}|_{\bar{I}}$ in $C^2(\bar{I})$ for each patch $I\subset\R,$ as $m\to+\infty$.

In addition, we may assume $t_m \to t^* \in [0, T]$ and $\bar{x}_m \to \bar{x}_\infty \in [-l_1, l_2]$. It follows from~\eqref{super_contra} that 
\begin{equation}\label{v_limit_fixed}
v_m(t_m, \bar{x}_m) = \bar{u}(t_m, x_m) > \tilde{\varepsilon}.
\end{equation}
Taking $m \to \infty$ in~\eqref{v_limit_fixed}, we obtain $v_\infty(t^*, \bar{x}_\infty) \ge \tilde{\varepsilon}$. On the other hand, for any fixed $x \in \mathbb{R}$, we have that $x + k_ml\notin\text{supp}(\bar{u}_0)$ as $|k_m| \to \infty$. Thus, we have that
$$
v_\infty(0, x) = \lim_{m \to \infty} \bar{u}_0(x + k_m l) = 0, \quad \forall x \in \mathbb{R}.
$$
By the strong maximum principle, $v_\infty \equiv 0$ on $[0, T] \times \mathbb{R}$, which contradicts $v_\infty(t^*, \bar{x}_\infty) \ge \tilde{\varepsilon} > 0$.
This completes the proof of Lemma~\ref{lem-ex22}.
\end{proof}

\begin{lemma}\label{lem-ex21}
	Let $0<l_1<L_1^c$ be fixed. Let $v$ be the solution of~\eqref{eq}-\eqref{patch} with a nonnegative continuous initial datum $v_0 \not \equiv 0$. There exist $\varepsilon_0>0$ and $\tilde{l}_2>0$ such that if $l_2>\tilde{l}_2$ and $0\leq v_0\leq \varepsilon_0$, then $\|v(t,\cdot)\|_{L^{\infty}(\R)}\to0$ as $t\to+\infty$.
\end{lemma}

\begin{remark}
We first remark that Lemma~\ref{lem-ex21} holds regardless of the sign of $\int_{0}^{K_2}f_2(s)\ds$. By Lemma~\ref{lem-ex21} and Proposition~\ref{lem-STA}, we have $\lambda_1 > 0$ for $0 < l_1 < L_1^c$ and sufficiently large $l_2$, which matches the assumption of Theorem~\ref{thmEXT1}. However, the conclusion of Theorem~\ref{thmEXT1} cannot be directly applied to prove Theorem~\ref{thmEXT2}. The reason is that the admissible upper bound for the initial datum, $\varepsilon$, in Theorem~\ref{thmEXT1} depends on $\lambda_1$, which itself depends on $l_2$. In contrast, the constant $\varepsilon_0$ in Lemma~\ref{lem-ex21} is independent of $l_2$ large enough. Consequently, a different argument (as shown in Lemma~\ref{lem-ex21}) is required to ensure the result holds for arbitrary small initial datum for $0<l_1<L_1^c$ and $l_2$ large enough.
\end{remark}

\begin{proof}[Proof of Lemma~\ref{lem-ex21}]
Fix $0<l_1<L_1^c$. According to Proposition~\ref{lem-STA}, for any $l_2>l_2^c$, $0$ of~\eqref{sta1} is stable, which implies that $\lambda_1^{l_2}>0$. Here, $\lambda_1^{l_2}$ denotes the principal eigenvalue of the eigenvalue problem~\eqref{eigenv}, associated with the principal eigenfunction $\psi^{l_2}:\R\to\R$ that satisfy
\begin{equation}\label{eigenv-l2}
	\begin{cases}
		-d(x)(\psi^{l_2})^{\prime\prime}(x) -f_s(x,0)\psi^{l_2}(x)=\lambda_1^{l_2}\psi^{l_2}(x),  & x \in \mathbb{R} \backslash S, \\
		\psi^{l_2}(x^{-}) =\psi^{l_2}(x^{+}),\quad (\psi^{l_2})^{\prime}(x^{-})=\sigma (\psi^{l_2})^{\prime}(x^{+}), & x \in S_1, \\
		\psi^{l_2}(x^{-}) =\psi^{l_2}(x^{+}),\,\, \sigma (\psi^{l_2})^{\prime}(x^{-})=(\psi^{l_2})^{\prime}(x^{+}), & x \in S_2,\\
		\psi^{l_2}\text{ is periodic}, \psi^{l_2}>0, \|\psi^{l_2}\|_{L^{\infty}(\R)}=1.
	\end{cases}
\end{equation}

To obtain a uniform positive lower bound for $\lambda_1^{l_2}$, we study the limiting eigenvalue problem as $l_2 \to +\infty$. By standard elliptic estimates, as $l_2 \to +\infty$, the sequence of pairs $(\lambda_1^{l_2}, \psi^{l_2})$ converges (up to a subsequence) to $(\lambda_1^{\infty}, \psi^{\infty})$, where the convergence $\psi^{l_2} \to \psi^{\infty}$ is locally uniform in $\mathbb{R}$, and $\psi^{l_2}|_{\bar{I}}\to \psi^{\infty}|_{\bar{I}}$ in $C^2(\bar{I})$ for each patch $I \subset \mathbb{R}$ as $l_2 \to +\infty$. Moreover, $(\lambda_1^{\infty}, \psi^{\infty})$ satisfies the corresponding limiting eigenvalue problem. 

Note that by~\eqref{L1c} and Proposition~\ref{lem-STA}, $\lambda_1^{\infty} = 0$ if and only if $l_1 = L_1^c$. Since  $\lambda_1^{l_2} > 0$ and $0<l_1 < L_1^c$, we have that $\lambda_1^{\infty} > 0$. Thus, there exists $\tilde{l}_2 \ge l_2^c$ such that for all $l_2 > \tilde{l}_2$,
\begin{equation}\label{lam*}
	\lambda_1^{l_2} \ge \frac{\lambda_1^{\infty}}{2} > 0.
\end{equation}
	
	Now we give some properties of the function $\psi^{l_2}$. 
	\begin{claim}\label{Claim6.1}
		There holds $\dis\max_{\R}\psi^{l_2}=\psi(x_0)=1$ with $x_0\in(nl-l_1,nl)$ for some $n\in\Z$.
	\end{claim}
	Indeed, from the first equation of~\eqref{eigenv-l2}, we have that 
	\begin{equation}\label{psi''}
		(\psi^{l_2})^{\prime\prime}(x)=- \dfrac{(f_s(x,0)+\lambda_1^{l_2})}{d(x)}\psi^{l_2},\quad  x \in \mathbb{R} \backslash S.
	\end{equation}
	Since $\psi^{l_2}$ has a constant sign on each patch, $(\psi^{l_2})^{\prime\prime}$ also has a constant sign for each patch. Thus $\psi^{l_2}$ is either concave or convex within each patch. On the one hand, we derive from~\eqref{lam*},~\eqref{psi''}, and~\eqref{f1} that 
	$$
	(\psi^{l_2})^{\prime\prime}(x)=-\dfrac{\psi^{l_2}(x)}{d_1}(\lambda_1^{l_2}+f_1^{\prime}(0))<0,\quad x\in(nl-l_1,nl), n\in\Z.
	$$
	Thus, $\psi^{l_2}(x)$ is a concave function for $x\in(nl-l_1,nl)$. On the other hand, by applying~\eqref{eigenv-l2} at minimal and maximal points of the positive continuous periodic function $\psi^{l_2}$, whether these points be in patches or on the interfaces, it follows that
	$$
	0<\lambda_1^{l_2}\leq-f_2^{\prime}(0).
	$$
	Using~\eqref{psi''} and the above inequality, we have that
	$$
	(\psi^{l_2})^{\prime\prime}(x)=-\dfrac{\psi^{l_2}(x)}{d_2}(\lambda_1^{l_2}+f_2^{\prime}(0))\geq0,\quad x\in(nl,nl+l_2), n\in\Z.
	$$
	This implies that $\psi^{l_2}$ is a convex function in each bistable patch. Furthermore, since $\psi^{l_2}$ is periodic and continuous in each patch, it attains a maximum at some point $x_0$. 
	If $x_0\in S$, then by the interface conditions, we have that 
	$(\psi^{l_2})^{\prime}(x_0)=0$. Since $\psi^{l_2}(x)$ is concave for $x\in(nl-l_1,nl)$, and the derivative is decreasing. Therefore, we have that $(\psi^{l_2})^{\prime}((nl)^{-})>0$, and $(\psi^{l_2})^{\prime}((nl-l_1)^{+})<0$.
	Similarly, since $\psi^{l_2}(x)$ is convex for $x\in(nl,nl+l_2)$, and the derivative is non-decreasing. Hence, we have that $(\psi^{l_2})^{\prime}((nl)^{+})\leq0$ and $(\psi^{l_2})^{\prime}((nl+l_2)^{-})\geq0$.
	The above results contradict the interface condition requiring the left and right derivatives at $x=nl$ or $x=nl+l_2$ to have the same sign. Therefore, $\psi^{l_2}$ cannot attain a maximum at the interface. If the maximum $x_0$ is attained in a bistable patch, a similar argument would lead to a contradiction. This ensures that $\max_{\R}\psi^{l_2}=1$ is indeed attained in each open KPP patch. 
	
	\begin{claim}\label{clam2}
		There exists $\eta>0$, for all $l_2>\tilde{l}_2$, such that
		\begin{equation}\label{eta}
			\psi^{l_2}(x)\geq \eta\ \hbox{ for all $x\in(nl-l_1,nl)$ and $n\in\Z$}.
		\end{equation}
	\end{claim}
	In fact, by contradiction, we assume that there exists a sequence $(l_2^m)_{m\in \NN}$ with $l_2^m>\tilde{l}_2$ such that 
	$$
	\lim_{m\to+\infty}\min_{x\in[-l_1,0]}\psi^{l_2^m}(x)=0.
	$$ 

	Since $0<\lambda_1^{l_2}\leq -f_2^{\prime}(0)$, $0<\psi^{l_2}\leq 1$, and $(\psi^{l_2})^{\prime\prime}$ is uniformly bounded, it follows from standard elliptic estimates that, up to a subsequence, the sequence $(\psi^{l_2^m})_{m\in\NN}$ converges in $C^2([-l_1-\tilde{l}_2,\tilde{l}_2])$ as $m\to+\infty$ to some function $\psi^{l_2^{\infty}}$, which satisfies
	\begin{equation}\label{eigenv-l2infty}
		\begin{cases}
			-d(x)(\psi^{l_2^{\infty}})^{\prime\prime}(x) -f_s(x,0)\psi^{l_2^{\infty}}(x)=\lambda_1^{l_2^{\infty}}\psi^{l_2^{\infty}}(x),  & x \in (-l_1-\tilde{l}_2,\tilde{l}_2) \backslash \{0,-l_1\}, \\
			\psi^{l_2^{\infty}}(x^{-}) =\psi^{l_2^{\infty}}(x^{+}),\quad (\psi^{l_2^{\infty}})^{\prime}(x^{-})=\sigma (\psi^{l_2^{\infty}})^{\prime}(x^{+}), & x =-l_1, \\
			\psi^{l_2^{\infty}}(x^{-}) =\psi^{l_2^{\infty}}(x^{+}),\,\, \sigma (\psi^{l_2^{\infty}})^{\prime}(x^{-})=(\psi^{l_2^{\infty}})^{\prime}(x^{+}), & x =0.
		\end{cases}
	\end{equation}
	Moreover, $\psi^{l_2^{\infty}}(x)\geq0$ for $x\in[-l_1-\tilde{l}_2,\tilde{l}_2]$, and
	$\dis\max_{x\in[-l_1,0]}\psi^{l_2^{\infty}}(x)=1$ by Claim~\ref{Claim6.1}. By assumption, there exists $\bar{x}\in[-l_1,0]$ such that $\psi^{l_2^{\infty}}(\bar{x})=0$. We distinguish two cases. Assume first that $\bar{x}\in(-l_1,0)$. Since $\psi^{l_2^{\infty}}(x)\geq0$, for $x\in[-l_1-\tilde{l}_2,\tilde{l}_2]$, we have
	that $(\psi^{l_2^{\infty}})^{\prime}(\bar{x})=0$. From the first
	equation of~\eqref{eigenv-l2infty} and using the Cauchy-Lipschitz
	theorem, we have that $\psi^{l_2^{\infty}}\equiv0$, which is
	impossible because $\dis\max_{x\in[-l_1,0]}\psi^{l_2^{\infty}}(x)=1$.
	Thus $\psi^{l_2^{\infty}}(x)>0$ for $(-l_1,0)$ and $\bar{x}\in\{0,-l_1\}$. Hence from the Hopf lemma
	 yields $\psi^{l_2^{\infty}}(\bar{x}^{-})<0$ and $\psi^{l_2^{\infty}}(\bar{x}^{+})>0$, contradicting
	  the interface conditions in~\eqref{eigenv-l2infty}. Thus we complete the proof of Claim~\ref{clam2}.
	
	Let $v$ be the solution of~\eqref{eq}-\eqref{patch} with a nonnegative continuous initial datum $0\lneqq v_0 \leq \varepsilon_0$. For $\kappa_0>0$ sufficiently small and any $\kappa\in(0,\kappa_0]$, define
	\begin{equation}\label{barv_0}
		\bar{v}_0(x):=\kappa(\psi^{l_2}(x)+\gamma), \quad x\in \R,
	\end{equation}
	where $\gamma>0$ will be determained later, and $\psi^{l_2}$ is a positive and bounded solution of~\eqref{eigenv-l2}. Now we will show that $\bar{v}_0(\cdot)$ is a supersolution of~\eqref{eq}-\eqref{patch} in $\R$. It is clear that interface conditions are satisfied. On the one hand, 
	using~\eqref{f1},~\eqref{lam*}, and~\eqref{eta}, we have that for $ x\in(nl-l_1,nl)$
	\begin{equation*}
		\begin{aligned}
			-d_1 \bar{v}_0^{\prime \prime}(x)-f_1(\bar{v}_0(x))\geq&-d_1 \kappa(\psi^{l_2})^{\prime \prime}(x)-f_1^{\prime}(0)\kappa(\psi^{l_2}(x)+\gamma)\\
			=&\kappa(\lambda_1^{l_2} \psi^{l_2}(x)-f_1^{\prime}(0)\gamma)\\
			\geq& \kappa\left(\dfrac{\lambda_1^{\infty}}{2} \eta-f_1^{\prime}(0)\gamma\right)\\
			\geq&0,
		\end{aligned}
	\end{equation*}
	the last inequality holds by taking 
	$$0<\gamma\leq \frac{\eta\lambda_1^{\infty}}{2f_1^{\prime}(0)}.$$ 
	On the other hand, since $0<\psi^{l_2}\leq1$, $\gamma>0$ is bounded, and $\kappa_0$ small enough, there exists a positive constant $\eta_1$ satisfying
	\begin{equation}\label{eta1}
		0<\eta_1<\dfrac{-f_2^{\prime}(0)\gamma}{1+\gamma},
	\end{equation} 
	such that
	\begin{equation}\label{fL}
		\|f_2'(0)\bar{u}-f_2(\bar{u})\|_{L^{\infty}(\R)}\leq \eta_1\|\bar{u}\|_{L^{\infty}(\R)}.
	\end{equation}
	It follows from~\eqref{fL} that for $x\in(nl,nl+l_2)$
	\begin{equation*}
		\begin{aligned}
			-d_2 \bar{v}_0^{\prime \prime}(x)-f_2(\bar{v}_0(x))\geq&-d_2 \kappa(\psi^{l_2})^{\prime \prime}(x)-(f_2^{\prime}(0)+\eta_1)\bar{v}_0(x)\\
			=&\kappa(\lambda_1^{l_2} \psi^{l_2}(x)-f_2^{\prime}(0)\gamma-\eta_1(\psi^{l_2}(x)+\gamma))\\
			\geq& \kappa(\lambda_1^{l_2} \psi^{l_2}(x)-f_2^{\prime}(0)\gamma-\eta_1(1+\gamma))\\
			\geq&0,
		\end{aligned}
	\end{equation*}
	the last inequality follows from $\lambda_1^{l_2} > 0$, $f_2'(0) < 0$, and~\eqref{eta1}.
	
	Furthermore, by assumption $0\leq v_0\leq\varepsilon_0$ and taking $0<\varepsilon_0\leq \kappa\gamma$ independent of $l_2$, we have that
	$$
	v_0(x)\leq \bar{v}_0(x), \quad x\in\R.
	$$
	Let $\bar{v}$ be the solutions to~\eqref{eq}-\eqref{patch} with initial datum $\bar{v}_0$ given by~\eqref{barv_0}. Using comparison principle (Proposition~\ref{proCP}), we have that
	\begin{equation*}\label{u-bar}
		0<v(t, x) \leq \bar{v}(t, x) \text { for all } t\geq0 \text { and } x \in \mathbb{R}.
	\end{equation*}
	Furthermore, we have that $\bar{v}(t,x)$ is nonincreasing with respect to $t$ in $[0,+\infty)\times\R$ and
	\begin{equation*}\label{u_0bar}
		\bar{v}(t,x)\leq\bar{v}_0(x), \quad \forall (t,x)\in[0,+\infty)\times\R.
	\end{equation*}
	By using Schauder estimates in~\cite[Theorem 2.2]{HLZ-2024}, it follows that there is a nonnegative classical bounded stationary solution $p$ of~\eqref{eq}-\eqref{patch} such that $\bar{v}(t, \cdot) \rightarrow p$ as $t \rightarrow+\infty$ locally uniformly in $\mathbb{R}$, and $\bar{v}(t, \cdot)|_{\bar{I}}\to p|_{\bar{I}}$ in $C^2(\bar{I})$ for each patch $I\subset\R$ as $t\to+\infty$. 
	Since $\psi^{l_2}$ is bounded from below by a positive constant (because it is positive, periodic and continuous), and since $p$ is bounded, one can define
	$$
	\kappa^*=\inf \{\kappa>0, \kappa(\psi^{l_2}+\gamma)>p \text { in } \mathbb{R}\} \in[0,\kappa_0].
	$$
	Similar to the proof of~\cite[Theorem 2.3 ]{HLZ-2024}, we can show that $\kappa^*=0$. This implies that $p \equiv 0$. Thus we have that $\|u(t,\cdot)\|_{L^{\infty}(\R)}\to0$ as $t\to+\infty$. We complete the proof of Lemma~\ref{lem-ex21}.
\end{proof}

\vs

Now we are ready to prove Theorem~\ref{thmEXT2}.

\medspace

\begin{proof}[Proof of Theorem~\ref{thmEXT2}]
	Let $R,A>0$ be fixed. Let $u$ be the solution to the Cauchy problem~\eqref{eq}-\eqref{patch} with a
	nonnegative continuous and compactly supported initial datum $u_0
	\not \equiv 0$ such that $\supp(u_0)$ is included in a bistable patch with $\diam(\supp(u_0)) \le R$ and
	$\dist(\supp(u_0), S) \ge \delta$.
	
	  Since  $l_2\geq \max\{R+2\delta, \tilde{l}_2\}$ with $\delta>0$ sufficiently large and $\tilde{l}_2\geq l_2^c$ as in Lemma~\ref{lem-ex21}, we assume without loss of generality  that 
	$$\supp(u_0)\subset[\delta, R+\delta]\subset[0, l_2]=:I_b.$$
	
	Take $\varepsilon_0\in(0,\theta)$ small enough.  Due to $f_2$ satisfies~\eqref{f2} with $\int_0^{K_2} f_2(s)
 \mathrm{d} s<0$, we can construct a function $\bar{f}_2\in C^1(\mathbb{R})$
 such that $\bar{f}_2 \geq f_2$ in $\mathbb{R}, \bar{f}_2(\varepsilon_0/2)=
 \bar{f}_2\left(\theta\right)=\bar{f}_2(A+K_2)=0,
 \bar{f}_2^{\prime}(\varepsilon_0/2)<0, \bar{f}_2^{\prime}(A+K_2)<0, \bar{f}_2>0$ in
 $(-\infty, \varepsilon_0/2) \cup\left(\theta, A+K_2\right), \bar{f}_2<0$ in $\left(\varepsilon_0/2,
 \theta\right) \cup(A+K_2,+\infty)$, and $\int_{\varepsilon_0/2}^{A+K_2} \bar{f}_2(s)
 \mathrm{d} s<0$. There is then a decreasing front profile $\bar{\phi}_2$
 solving the following system
\begin{equation}\label{TW1}
\left\{\begin{array}{l}
d_2 \bar{\phi}_2^{\prime \prime}+\bar{c}_2 \bar{\phi}_2^{\prime}+\bar{f}_2(\bar{\phi}_2)=0 \text { in } \mathbb{R}, \\
\bar{\phi}_2(-\infty)=A+K_2,~~ \bar{\phi}_2(+\infty)=\varepsilon_0/2, ~~\bar{\phi}_2(R/2)\geq A,
\end{array}\right.
\end{equation}
with wave speed $\bar{c}_2<0$. Furthermore, there exists $T>0$ such that
\begin{equation}\label{bound}
\dfrac{\varepsilon_0}{2}\leq \min(\bar{\phi}_2(x-\bar{c}_2T-x_0), \bar{\phi}_2(-x-\bar{c}_2T+x_0))\leq\varepsilon_0, \quad \forall x\in\R,
\end{equation}
where $x_0=\delta+R/2$, and $T$ only depends on $\varepsilon_0$, $A$, $R$, and $\bar{\phi}_2(R)$. 

With $\varepsilon=\varepsilon_0/2$ and $T>0$ given above, we deduce from Lemma~\ref{lem-ex22} that
up to increasing  $\delta>0$, 
\begin{equation}\label{u-bound}
0 \le u(t,x) \le \frac{\varepsilon_0}{2},
\qquad 0 \le t \le T,\; x \in \R \backslash I_b .
\end{equation}
Define
$$
\bar{u}(t,x)=\min\left(\bar{\phi}_2(x-\bar{c}_2t-x_0), \bar{\phi}_2(-x-\bar{c}_2t+x_0)\right), \quad 0\leq t\leq T, x\in I_b.
$$
It follows from $\|u_0\|_{L^{\infty}(\R)}\leq A$ and  $\bar{\phi}_2(R/2)\geq A$ that 
\begin{equation*}\label{u_0}
u_0(x)\leq \bar{u}(0,x),\quad x\in I_b.
\end{equation*}
It is also seen from~\eqref{bound} and~\eqref{u-bound} that
\begin{equation*}
	\bar u(t,x)\ge \frac{\varepsilon_0}{2}\ge u(t,x),~~~~t\in[0,T],~~x=0, l_2.
\end{equation*} 
Furthermore, using~\eqref{TW1}, we have that for $t\in(0,T]$
\begin{align*}\label{baru}
&\partial_t\bar\phi_2(x-\bar c_2t-x_0)-d_2\bar{\phi}_2^{\prime\prime}(x-\bar{c}_2t-x_0)-f_2(\bar{\phi}_2(x-\bar c_2t-x_0))\\
=&-\bar{c}_2\bar{\phi}_2(x-\bar{c}_2t-x_0)-d_2\bar{\phi}_2^{\prime\prime}(x-\bar{c}_2t-x_0)-f_2(\bar{\phi}_2(x-\bar c_2t-x_0))\\
=&\bar{f}_2(\bar{\phi}_2(x-\bar c_2t-x_0))-f_2(\bar{\phi}_2(x-\bar c_2t-x_0))\\
\geq&0
\end{align*}
in bistable patch.
By symmetry, the above inequality also holds for  $\bar{\phi}(-x-\bar{c}_2t+x_0)$ with $0<t\leq T$ in bistable patch. Consequently, 
\begin{equation*}
	\partial_t\bar{u}-d_2\partial_{xx}\bar{u}-f_2(\bar{u})\ge 0,~~~~t\in[0,T],~x\in I_b.
\end{equation*}
It then follows from the comparison principle that
\begin{equation}\label{u}
u(t,x)\leq\bar{u}(t,x),\quad 0\leq t\leq T,~ x\in I_{b}.
\end{equation}
Therefore, gathering~\eqref{bound},~\eqref{u-bound}, as well as~\eqref{u}, we deduce that
\begin{equation*}
	\label{thm2.6-1}
	0\leq u(T,\cdot)\leq\varepsilon_0,\quad x\in\R.
\end{equation*}

Furthermore, recalling~\eqref{lam*} and~\eqref{eta} from the proof of Lemma~\ref{lem-ex21}, the constants $\eta$ and $\lambda_1^{\infty}$ are independent of $l_2$. Since $l_2 \ge \max\{R+2\delta, \tilde{l}_2\}$, we may choose 
$$
0 < \varepsilon_0 \leq \min\{\kappa\gamma, 2\theta\},
$$
where $0 < \gamma \leq \frac{\eta \lambda_1^{\infty}}{2f_1^{\prime}(0)}$ and $\kappa \in (0, \kappa_0]$, so that $\varepsilon_0$ is independent of $l_2$. Therefore, the result of Theorem~\ref{thmEXT2} follows directly from Lemma~\ref{lem-ex21}. This completes the proof of Theorem~\ref{thmEXT2}.
\end{proof}


\appendix

\section{Appendix}

We collect several preliminary comparison principles established in~\cite{HLZ-2024} in different settings of patch models, as they will be frequently used throughout this paper.

Let $(a, b) \subset \mathbb{R}$ be an interval composed of finitely many patches. More precisely, we assume that $-\infty \leq a=x_0<x_1<\cdots<x_n=b \leq+\infty$ and define the patches $I_i=\left(x_{i-1}, x_i\right)$ for $i=1, \ldots, n$. We consider a one-dimensional parabolic operator
$$
\mathcal{L} u:=u_t-d(x) u_{x x}-c(t, x) u_x-F(x, u), \text { for } t>0 \text { and } x \in\R\backslash\left\{x_1, \ldots, x_{n-1}\right\}=\bigcup_{i=1}^n I_i,
$$
with interface conditions
\begin{equation}\label{inter-A.1}
	u\left(t, x_i^{-}\right)=u\left(t, x_i^{+}\right) \text {and } u_x\left(t, x_i^{-}\right)=\sigma_i u_x\left(t, x_i^{+}\right), \text {for } t>0 \text { and } i=1, \ldots, n-1 .
\end{equation}
If $a$ or $b$ is finite, we impose Dirichlet-type boundary conditions:
\begin{equation}\label{inter-A.2}
	u(t, a)=\varphi^{-}(t) \text { or } u(t, b)=\varphi^{+}(t), \text { for } t \geq 0,
\end{equation}
where $\varphi^{ \pm}:[0,+\infty) \rightarrow \mathbb{R}$ are given continuous functions. Here, the function $x \mapsto d(x)$ is assumed to be constant and positive in each patch, i.e., $\left.d\right|_{I_i}=d_i>0$ for some constant $d_i$, the function $c$ is assumed to be continuous and bounded in $\left(0, T_0\right) \times \cup_{i=1}^n I_i$ for every $T_0 \in (0,+\infty)$, the $\sigma_i$'s are given positive real numbers, and, for each $1 \leq i \leq n, F(x, s)=f_i(s)$ for $(x, s) \in I_i \times \mathbb{R}$, with $f_i \in C^1(\mathbb{R})$.

The notion of super- and sub-solutions of $\mathcal{L} u=0$ associated with the interface and boundary conditions~\eqref{inter-A.1}-\eqref{inter-A.2} given in~\cite[Definition A.1]{HLZ-2024} is recalled below.

\begin{definition}
	For $T \in(0,+\infty]$, we say that a continuous function $\bar{u}:[0, T) \times \overline{(a, b)} \rightarrow \mathbb{R}$, which is assumed to be bounded in $\left[0, T_0\right] \times \overline{(a, b)}$ for every $T_0 \in(0, T)$, is a supersolution for the problem $\mathcal{L} u=0$ with interface and boundary conditions~\eqref{inter-A.1}-\eqref{inter-A.2}, if $\left.\bar{u}\right|_{(0, T) \times \overline{I_i}} \in C_{t ; x}^{1 ; 2}\left((0, T) \times \overline{I_i}\right)$ satisfies $\left.\mathcal{L} \bar{u}\right|_{(0, T) \times I_i} \geq 0$ in the classical sense for each $1 \leq i \leq n$, and if
	$$
	\bar{u}_x\left(t, x_i^{-}\right) \geq \sigma_i \bar{u}_x\left(t, x_i^{+}\right), \text { for } t \in(0, T) \text { and } i=1, \ldots, n-1,
	$$
	and
	$$
	\bar{u}(t, a) \geq \varphi^{-}(t) \text { or } \bar{u}(t, b) \geq \varphi^{+}(t), \text { for } t \in[0, T),
	$$
	provided that $a$ or $b$ is finite. A subsolution can be defined in a similar way with all the inequality signs above reversed.
\end{definition}

The following proposition provides a comparison principle between
super- and subsolutions on a bounded interval $(a, b)$, as established in~\cite[Proposition A.2]{HLZ-2024}.

\begin{proposition}[Comparison principle on bounded intervals]
	\label{proCPBD}
	Assume that $-\infty<a< b<+\infty$. For $T \in(0,+\infty]$, let $\bar{u}$ and $\underline{u}$ be, respectively, a super- and a subsolution in $[0, T) \times[a, b]$ of $\mathcal{L} u=0$ with~\eqref{inter-A.1}-\eqref{inter-A.2}, and assume that $\bar{u}(0, \cdot) \geq \underline{u}(0, \cdot)$ in $[a, b]$. Then, $\bar{u} \geq \underline{u}$ in $[0, T) \times[a, b]$ and, if $\bar{u}(0, \cdot) \not \equiv \underline{u}(0, \cdot)$, then $\bar{u}>\underline{u}$ in $(0, T) \times(a, b)$.
\end{proposition}

We finally recall a comparison principle in the whole space with countably many interfaces. 
Set $H=\left\{x_i: i \in \mathbb{Z}\right\} \subset \mathbb{R}$ with
$$
\delta:=\inf _{i \in \mathbb{Z}}\left(x_{i+1}-x_i\right)>0,
$$
and we consider the problem
\begin{equation}\label{cp-eq}
	\left\{\begin{aligned}
		u_t-d(x) u_{x x}-c(t, x) u_x & =F(x, u), & & t>0, x \in \mathbb{R} \backslash H, \\
		u\left(t, x_i^{-}\right) & =u\left(t, x_i^{+}\right), & & t>0, i \in \mathbb{Z}, \\
		u_x\left(t, x_i^{-}\right) & =\sigma_i u_x\left(t, x_i^{+}\right), & & t>0, i \in \mathbb{Z} .
	\end{aligned}\right.
\end{equation}
We assume that the function $x \mapsto d(x)$ is equal to a positive constant $d_i$ in each interval $\left(x_i, x_{i+1}\right)$, and that $\sup _{i \in \mathbb{Z}} d_i<+\infty$. The function $c$ is assumed to be continuous and bounded in $\left(0, T_0\right) \times(\mathbb{R} \backslash H)$ for every $T_0 \in(0,+\infty)$, the $\sigma_i$'s are given positive real numbers, and there are $C^1(\mathbb{R})$ functions $\left(f_i\right)_{i \in \mathbb{Z}}$ such that $F(x, s)=f_i(s)$ for every $(x, s) \in\left(x_i, x_{i+1}\right) \times \mathbb{R}$ and $i \in \mathbb{Z}$, with $\sup _{i \in \mathbb{Z}}\left\|f_i^{\prime}\right\|_{L^{\infty}([-L, L])}<+\infty$ for every $L>0$.

According to~\cite[Proposition A.4]{HLZ-2024}, the following result provides a comparison between sub- and supersolutions of~\eqref{cp-eq} under ordered initial conditions.

\begin{proposition}[Comparison principle for problems of type~\eqref{cp-eq}]\label{proCP}
	For $0<T\leq +\infty$, let $\bar{u}$ and $\underline{u}$ be, respectively, a super- and a subsolution of~\eqref{cp-eq}
	in $[0, T) \times \mathbb{R}$ with $\bar{u}(0, \cdot) \geq \underline{u}(0, \cdot)$ in $\mathbb{R}$. Then, $\bar{u} \geq \underline{u}$ in $[0, T) \times \mathbb{R}$, and, if $\bar{u}(0, \cdot) \not \equiv \underline{u}(0, \cdot)$, then $\bar{u}>\underline{u}$ in $(0, T) \times \mathbb{R}$.
\end{proposition}


\end{document}